\documentclass[a4paper, 11pt]{amsart}    
\usepackage{latexsym, amsmath, amsthm, amssymb, setspace, verbatim, bbm}
\usepackage[all]{xy}
\usepackage{longtable}
\usepackage[utf8]{inputenc} \usepackage[T1]{fontenc} \usepackage{lmodern}
\usepackage{hyperref, aliascnt}
\usepackage{enumitem}

\title[The cocycle category]{On a categorical framework for classifying \cstar-dynamics up to cocycle conjugacy}
\author{Gábor Szabó}
\address{KU Leuven, Department of Mathematics, Celestijnenlaan 200b box 2400, \phantom{-..-}B-3001 Leuven, Belgium}
\email{gabor.szabo@kuleuven.be}
\thanks{This work was supported by the start-up grant STG/18/019 of KU Leuven, the research project C14/19/088 funded by the research council of KU Leuven, and the research project G085020N funded by the Research Foundation Flanders (FWO)}

\numberwithin{equation}{section}

\begin{document}

% math
\renewcommand\matrix[1]{\left(\begin{array}{*{10}{c}} #1 \end{array}\right)}  % Matrix
\newcommand\set[1]{\left\{#1\right\}}  % Menge

%% Besondere Variablen
%Zahlmengen-Stil
\newcommand{\IA}[0]{\mathbb{A}} \newcommand{\IB}[0]{\mathbb{B}}
\newcommand{\IC}[0]{\mathbb{C}} \newcommand{\ID}[0]{\mathbb{D}}
\newcommand{\IE}[0]{\mathbb{E}} \newcommand{\IF}[0]{\mathbb{F}}
\newcommand{\IG}[0]{\mathbb{G}} \newcommand{\IH}[0]{\mathbb{H}}
\newcommand{\II}[0]{\mathbb{I}} \renewcommand{\IJ}[0]{\mathbb{J}}
\newcommand{\IK}[0]{\mathbb{K}} \newcommand{\IL}[0]{\mathbb{L}}
\newcommand{\IM}[0]{\mathbb{M}} \newcommand{\IN}[0]{\mathbb{N}}
\newcommand{\IO}[0]{\mathbb{O}} \newcommand{\IP}[0]{\mathbb{P}}
\newcommand{\IQ}[0]{\mathbb{Q}} \newcommand{\IR}[0]{\mathbb{R}}
\newcommand{\IS}[0]{\mathbb{S}} \newcommand{\IT}[0]{\mathbb{T}}
\newcommand{\IU}[0]{\mathbb{U}} \newcommand{\IV}[0]{\mathbb{V}}
\newcommand{\IW}[0]{\mathbb{W}} \newcommand{\IX}[0]{\mathbb{X}}
\newcommand{\IY}[0]{\mathbb{Y}} \newcommand{\IZ}[0]{\mathbb{Z}}

\newcommand{\Ia}[0]{\mathbbmss{a}} \newcommand{\Ib}[0]{\mathbbmss{b}}
\newcommand{\Ic}[0]{\mathbbmss{c}} \newcommand{\Id}[0]{\mathbbmss{d}}
\newcommand{\Ie}[0]{\mathbbmss{e}} \newcommand{\If}[0]{\mathbbmss{f}}
\newcommand{\Ig}[0]{\mathbbmss{g}} \newcommand{\Ih}[0]{\mathbbmss{h}}
\newcommand{\Ii}[0]{\mathbbmss{i}} \newcommand{\Ij}[0]{\mathbbmss{j}}
\newcommand{\Ik}[0]{\mathbbmss{k}} \newcommand{\Il}[0]{\mathbbmss{l}}
\renewcommand{\Im}[0]{\mathbbmss{m}} \newcommand{\In}[0]{\mathbbmss{n}}
\newcommand{\Io}[0]{\mathbbmss{o}} \newcommand{\Ip}[0]{\mathbbmss{p}}
\newcommand{\Iq}[0]{\mathbbmss{q}} \newcommand{\Ir}[0]{\mathbbmss{r}}
\newcommand{\Is}[0]{\mathbbmss{s}} \newcommand{\It}[0]{\mathbbmss{t}}
\newcommand{\Iu}[0]{\mathbbmss{u}} \newcommand{\Iv}[0]{\mathbbmss{v}}
\newcommand{\Iw}[0]{\mathbbmss{w}} \newcommand{\Ix}[0]{\mathbbmss{x}}
\newcommand{\Iy}[0]{\mathbbmss{y}} \newcommand{\Iz}[0]{\mathbbmss{z}}

%Geschwungener Stil
\newcommand{\CA}[0]{\mathcal{A}} \newcommand{\CB}[0]{\mathcal{B}}
\newcommand{\CC}[0]{\mathcal{C}} \newcommand{\CD}[0]{\mathcal{D}}
\newcommand{\CE}[0]{\mathcal{E}} \newcommand{\CF}[0]{\mathcal{F}}
\newcommand{\CG}[0]{\mathcal{G}} \newcommand{\CH}[0]{\mathcal{H}}
\newcommand{\CI}[0]{\mathcal{I}} \newcommand{\CJ}[0]{\mathcal{J}}
\newcommand{\CK}[0]{\mathcal{K}} \newcommand{\CL}[0]{\mathcal{L}}
\newcommand{\CM}[0]{\mathcal{M}} \newcommand{\CN}[0]{\mathcal{N}}
\newcommand{\CO}[0]{\mathcal{O}} \newcommand{\CP}[0]{\mathcal{P}}
\newcommand{\CQ}[0]{\mathcal{Q}} \newcommand{\CR}[0]{\mathcal{R}}
\newcommand{\CS}[0]{\mathcal{S}} \newcommand{\CT}[0]{\mathcal{T}}
\newcommand{\CU}[0]{\mathcal{U}} \newcommand{\CV}[0]{\mathcal{V}}
\newcommand{\CW}[0]{\mathcal{W}} \newcommand{\CX}[0]{\mathcal{X}}
\newcommand{\CY}[0]{\mathcal{Y}} \newcommand{\CZ}[0]{\mathcal{Z}}

%Script Stil
\newcommand{\FA}[0]{\mathfrak{A}} \newcommand{\FB}[0]{\mathfrak{B}}
\newcommand{\FC}[0]{\mathfrak{C}} \newcommand{\FD}[0]{\mathfrak{D}}
\newcommand{\FE}[0]{\mathfrak{E}} \newcommand{\FF}[0]{\mathfrak{F}}
\newcommand{\FG}[0]{\mathfrak{G}} \newcommand{\FH}[0]{\mathfrak{H}}
\newcommand{\FI}[0]{\mathfrak{I}} \newcommand{\FJ}[0]{\mathfrak{J}}
\newcommand{\FK}[0]{\mathfrak{K}} \newcommand{\FL}[0]{\mathfrak{L}}
\newcommand{\FM}[0]{\mathfrak{M}} \newcommand{\FN}[0]{\mathfrak{N}}
\newcommand{\FO}[0]{\mathfrak{O}} \newcommand{\FP}[0]{\mathfrak{P}}
\newcommand{\FQ}[0]{\mathfrak{Q}} \newcommand{\FR}[0]{\mathfrak{R}}
\newcommand{\FS}[0]{\mathfrak{S}} \newcommand{\FT}[0]{\mathfrak{T}}
\newcommand{\FU}[0]{\mathfrak{U}} \newcommand{\FV}[0]{\mathfrak{V}}
\newcommand{\FW}[0]{\mathfrak{W}} \newcommand{\FX}[0]{\mathfrak{X}}
\newcommand{\FY}[0]{\mathfrak{Y}} \newcommand{\FZ}[0]{\mathfrak{Z}}

\newcommand{\Fa}[0]{\mathfrak{a}} \newcommand{\Fb}[0]{\mathfrak{b}}
\newcommand{\Fc}[0]{\mathfrak{c}} \newcommand{\Fd}[0]{\mathfrak{d}}
\newcommand{\Fe}[0]{\mathfrak{e}} \newcommand{\Ff}[0]{\mathfrak{f}}
\newcommand{\Fg}[0]{\mathfrak{g}} \newcommand{\Fh}[0]{\mathfrak{h}}
\newcommand{\Fi}[0]{\mathfrak{i}} \newcommand{\Fj}[0]{\mathfrak{j}}
\newcommand{\Fk}[0]{\mathfrak{k}} \newcommand{\Fl}[0]{\mathfrak{l}}
\newcommand{\Fm}[0]{\mathfrak{m}} \newcommand{\Fn}[0]{\mathfrak{n}}
\newcommand{\Fo}[0]{\mathfrak{o}} \newcommand{\Fp}[0]{\mathfrak{p}}
\newcommand{\Fq}[0]{\mathfrak{q}} \newcommand{\Fr}[0]{\mathfrak{r}}
\newcommand{\Fs}[0]{\mathfrak{s}} \newcommand{\Ft}[0]{\mathfrak{t}}
\newcommand{\Fu}[0]{\mathfrak{u}} \newcommand{\Fv}[0]{\mathfrak{v}}
\newcommand{\Fw}[0]{\mathfrak{w}} \newcommand{\Fx}[0]{\mathfrak{x}}
\newcommand{\Fy}[0]{\mathfrak{y}} \newcommand{\Fz}[0]{\mathfrak{z}}

%Modifikation der Variablen
\renewcommand{\phi}[0]{\varphi}
\newcommand{\eps}[0]{\varepsilon}

%zusätzliche Features
\newcommand{\id}[0]{\operatorname{id}}		% Identität
\renewcommand{\sp}[0]{\operatorname{Sp}}		% Spektrum eines Elements
\newcommand{\eins}[0]{\mathbf{1}}			% Eine Eins in allgemeinerem Kontext, z.B. in einem Ring
\newcommand{\diag}[0]{\operatorname{diag}}
\newcommand{\ad}[0]{\operatorname{Ad}}
\newcommand{\ev}[0]{\operatorname{ev}}
\newcommand{\fin}[0]{{\subset\!\!\!\subset}}
\newcommand{\Aut}[0]{\operatorname{Aut}}
\newcommand{\dimrok}[0]{\dim_{\mathrm{Rok}}}
\newcommand{\dst}[0]{\displaystyle}
\newcommand{\cstar}[0]{\ensuremath{\mathrm{C}^*}}
\newcommand{\wstar}[0]{\ensuremath{\mathrm{W}^*}}
\newcommand{\dist}[0]{\operatorname{dist}}
\newcommand{\ann}[0]{\operatorname{Ann}}
\newcommand{\cc}[0]{\simeq_{\mathrm{cc}}}
\newcommand{\scc}[0]{\simeq_{\mathrm{scc}}}
\newcommand{\vscc}[0]{\simeq_{\mathrm{vscc}}}
\newcommand{\scd}[0]{\preceq_{\mathrm{scd}}}
\newcommand{\cel}[0]{\ensuremath{\mathrm{cel}}}
\newcommand{\acel}[0]{\ensuremath{\mathrm{acel}}}
\newcommand{\sacel}[0]{\ensuremath{\mathrm{sacel}}}
\newcommand{\prim}[0]{\ensuremath{\mathrm{Prim}}}
\newcommand{\co}[0]{\ensuremath{\mathrm{co}}}
\newcommand{\GL}[0]{\operatorname{GL}}
\newcommand{\Bott}[0]{\ensuremath{\mathrm{Bott}}}
\newcommand{\tK}[0]{\ensuremath{\underline{K}}}
\newcommand{\Hom}[0]{\operatorname{Hom}}
\newcommand{\sep}[0]{\ensuremath{\mathrm{sep}}}
\newcommand{\inv}[0]{\ensuremath{\operatorname{Inv}}}

\newcommand{\corep}[0]{\ensuremath{\mathrm{CoRep}}}
\newcommand{\cohom}[0]{\ensuremath{\mathrm{CoMor}}}
\newcommand{\cohomr}[0]{\ensuremath{\mathrm{CoMor}_\Fp}}

\newcommand{\uee}[0]{\hspace{0.5mm}\sim_{\makebox[0pt]{\footnotesize\hspace{1.5mm}$\mathrm{u}$}}\hspace{1mm}}

\newcommand{\puee}[0]{\hspace{1mm}\sim_{\makebox[0pt]{\footnotesize\hspace{2.8mm}$\mathrm{pu}$}}\hspace{2mm}}

\newcommand{\ue}[0]{\hspace{0.7mm}\approx_{\makebox[0pt]{\footnotesize\hspace{1.5mm}$\mathrm{u}$}}\hspace{1mm}}

\newcommand{\subue}[0]{\hspace{0.7mm}\precapprox_{\makebox[0pt]{\footnotesize\hspace{1.5mm}$\mathrm{u}$}}\hspace{1mm}}

\newcommand{\asue}[0]{\hspace{0.7mm}\approxeq_{\makebox[0pt]{\footnotesize\hspace{1.5mm}$\mathrm{u}$}}\hspace{1mm}}

\newcommand{\subasue}[0]{\hspace{0.7mm}\preceq_{\makebox[0pt]{\footnotesize\hspace{1.5mm}$\mathrm{u}$}}\hspace{1mm}}

\newcommand{\pue}[0]{\hspace{1mm}\approx_{\makebox[0pt]{\footnotesize\hspace{2.75mm}$\mathrm{pu}$}}\hspace{2mm}}

\newcommand{\pasue}[0]{\hspace{0.7mm}\approxeq_{\makebox[0pt]{\footnotesize\hspace{3mm}$\mathrm{pu}$}}\hspace{2.2mm}}

\newcommand{\greater}[0]{>}
\newcommand{\vslash}[0]{|}
\newcommand{\norm}[0]{\ensuremath{\|}}
\newcommand{\msout}[1]{\text{\sout{$#1$}}}

% theorems
\newtheorem{satz}{Satz}[section]		% <--- optional, zählt so mit den Abschnitten

\newaliascnt{corCT}{satz}
\newtheorem{cor}[corCT]{Corollary}
\aliascntresetthe{corCT}
\providecommand*{\corCTautorefname}{Corollary}
\newaliascnt{lemmaCT}{satz}
\newtheorem{lemma}[lemmaCT]{Lemma}
\aliascntresetthe{lemmaCT}
\providecommand*{\lemmaCTautorefname}{Lemma}
\newaliascnt{propCT}{satz}
\newtheorem{prop}[propCT]{Proposition}
\aliascntresetthe{propCT}
\providecommand*{\propCTautorefname}{Proposition}
\newaliascnt{theoremCT}{satz}
\newtheorem{theorem}[theoremCT]{Theorem}
\aliascntresetthe{theoremCT}
\providecommand*{\theoremCTautorefname}{Theorem}
\newtheorem*{theoreme}{Theorem}

\theoremstyle{definition}

\newaliascnt{conjectureCT}{satz}
\newtheorem{conjecture}[conjectureCT]{Conjecture}
\aliascntresetthe{conjectureCT}
\providecommand*{\conjectureCTautorefname}{Conjecture}
\newaliascnt{defiCT}{satz}
\newtheorem{defi}[defiCT]{Definition}
\aliascntresetthe{defiCT}
\providecommand*{\defiCTautorefname}{Definition}
\newtheorem*{defie}{Definition}
\newaliascnt{notaCT}{satz}
\newtheorem{nota}[notaCT]{Notation}
\aliascntresetthe{notaCT}
\providecommand*{\notaCTautorefname}{Notation}
\newtheorem*{notae}{Notation}
\newaliascnt{remCT}{satz}
\newtheorem{rem}[remCT]{Remark}
\aliascntresetthe{remCT}
\providecommand*{\remCTautorefname}{Remark}
\newtheorem*{reme}{Remark}
\newaliascnt{exampleCT}{satz}
\newtheorem{example}[exampleCT]{Example}
\aliascntresetthe{exampleCT}
\providecommand*{\exampleCTautorefname}{Example}
\newaliascnt{questionCT}{satz}
\newtheorem{question}[questionCT]{Question}
\aliascntresetthe{questionCT}
\providecommand*{\questionCTautorefname}{Question}
\newtheorem*{questione}{Question}

\newcounter{theoremintro}
\renewcommand*{\thetheoremintro}{\Alph{theoremintro}}
\newaliascnt{theoremiCT}{theoremintro}
\newtheorem{theoremi}[theoremiCT]{Theorem}
\aliascntresetthe{theoremiCT}
\providecommand*{\theoremiCTautorefname}{Theorem}
\newaliascnt{coriCT}{theoremintro}
\newtheorem{cori}[coriCT]{Corollary}
\aliascntresetthe{coriCT}
\providecommand*{\coriCTautorefname}{Corollary}
\newaliascnt{conjectureiCT}{theoremintro}
\newtheorem{conjecturei}[conjectureiCT]{Conjecture}
\aliascntresetthe{conjectureiCT}
\providecommand*{\conjectureiCTautorefname}{Conjecture}
\newaliascnt{defiiCT}{theoremintro}
\newtheorem{defii}[defiiCT]{Definition}
\aliascntresetthe{defiiCT}
\providecommand*{\defiiCTautorefname}{Definition}

%%%%%%%%%%%%%%%%%%%%%%%%%%%%%%%%%%%%%%%%%%%%

\begin{abstract} 
We provide the rigorous foundations for a categorical approach to the classification of \cstar-dynamics up to cocycle conjugacy.
Given a locally compact group $G$, we consider a category of (twisted) $G$-\cstar-algebras, where morphisms between two objects are allowed to be equivariant maps or exterior equivalences, which leads to the concept of so-called cocycle morphisms.
An isomorphism in this category is precisely a cocycle conjugacy in the known sense.
We show that this category allows sequential inductive limits, and that some known functors on the usual category of $G$-\cstar-algebras extend.
After observing that this setup allows a natural notion of (approximate) unitary equivalence, the main aim of the paper is to generalize the fundamental intertwining results commonly employed in the Elliott program for classifying \cstar-algebras.
This reduces a given classification problem for \cstar-dynamics to the prevalence of certain uniqueness and existence theorems, and may provide a useful alternative to the Evans--Kishimoto intertwining argument in future research.
\end{abstract}

\maketitle

\tableofcontents

%%%%%%%%%%%%%%%%%%%%%%%%%%%

\section*{Introduction}

The concept of noncommutative dynamics is deeply rooted and ubiquitous within the subject of operator algebras, arguably because of the interesting ways in which groups can act on noncommutative structures.
On the one hand, elementary models of quantum mechanics conceptualize the passage of time in a physical system via an action of the real numbers $\IR$ on a suitable operator algebra of observables.
On the other hand, every classical topological dynamical system, say consisting of a locally compact group $G$ acting on a locally compact space $X$, can be considered as a \cstar-dynamical system $G\curvearrowright\CC(X)$ and then studied with the aid of operator algebraic methods.
The overarching concepts of \cstar-dynamics and \wstar-dynamics unify these ideas in one neat package.

Noncommutative dynamics have been taking center stage on the \wstar-side ever since the development of the fundamental structure theory of factors, notably including Tomita--Takesaki theory and the work of Connes--Haagerup for injective \wstar-algebras.
The structure of actions of (discrete) amenable groups on injective factors has subsequently been unraveled by the combined work of many researchers.
Through groundbreaking research spear-headed by Popa, various superrigidity phenomena were discovered in the context of \wstar-dynamics for actions of nonamenable groups, which stand in great contrast to the amenable setting via results such as the Connes--Feldman--Weiss theorem.

In a similar fashion and in part motivated by the above, \cstar-algebraists have long been interested to either apply or explore the structure of \cstar-dynamical systems.
The focus has mostly been on two natural but opposite extremes.
One possibility is to consider certain topological dynamical systems $G\curvearrowright X$ and investigate the crossed products $\CC(X)\rtimes G$, either for their classifiability \cite{Putnam89, ElliottEvans93, LinPhillips10, GiolKerr10, TomsWinter13, HirshbergWinterZacharias15, Szabo15plms, SzaboWuZacharias19, HirshbergWu17, HirshbergSzaboWinterWu17, HirshbergWu18, Kerr19, KerrSzabo20} 
or for the information that their isomorphism classes retain about the original dynamics \cite{HermanPutnamSkau92, GiordanoPutnamSkau95}. 
Another possibility is to start with certain group actions $G\curvearrowright A$ on simple \cstar-algebras and ask for their classification up to cocycle conjugacy, preferably using computable invariants such as $K$-theory.
The present work should be considered as a piece in the latter line of research.

In direct comparison to the \wstar-side, our understanding of \cstar-dynamics on simple \cstar-algebras is still somewhat underdeveloped. 
Nevertheless, the past attempts to classify automorphisms or group actions on simple \cstar-algebras have been both numerous and highly inventive.
We shall give a brief (albeit incomplete) summary of the state of the art here.
Many of the attempts to classify \cstar-dynamics have so far been underpinned by some kind of Rokhlin-type property.
After ideas of this type have surfaced in work of Connes, a propotypical definition of sort made its way to \cstar-algebras via work of Herman--Jones and Herman--Ocneanu \cite{HermanJones82, HermanOcneanu84}.
After a few more applications of these ideas \cite{BratteliKishimotoRordamStormer93, Rordam93, BratteliEvansKishimoto95, Kishimoto96}, notably in the study of certain purely infinite \cstar-algebras, the correct version of the Rokhlin property for single automorphisms was pinned down and exploited to great success \cite{Kishimoto95, EvansKishimoto97, ElliottEvansKishimoto98, Kishimoto98, Kishimoto98II, BratteliKishimoto00, Kishimoto00}.
This line of research has been considerably extended and perfected to this day in subsequent work by many researchers to classify automorphisms on more general \cstar-algebras or actions of certain higher-rank groups \cite{Nakamura99, Nakamura00, Matui08, KatsuraMatui08, Sato10, IzumiMatui10, Matui10, Matui11, Nawata19, IzumiMatui18, IzumiMatui19}.
The Rokhlin property for actions of finite groups was introduced by Izumi \cite{Izumi04, Izumi04II}, who also showcased its rigid behavior by giving a satisfactory classification theory up to conjugcy.
Results of this type were subsequently found for not necessarily unital \cstar-algebras \cite{Nawata16, GardellaSantiago16, Nawata19_2}, actions of compact groups \cite{HirshbergWinter07, Gardella14_1, Gardella14_2}, and compact quantum groups \cite{BarlakSzaboVoigt17}.
The Rokhlin property for flows was introduced by Kishimoto in \cite{Kishimoto96_R}, who provided evidence why one should expect that these can be classified up to cocycle conjugacy \cite{Kishimoto02, BratteliKishimotoRobinson07}.
This was confirmed in my recent work \cite{Szabo20R}, which was in part inspired by \cite{MasudaTomatsu16}.

A related but not identical line of research relating to tensorial absorption of strongly self-absorbing dynamical systems was fleshed out and applied in \cite{Szabo18ssa, Szabo18ssa2, Szabo17ssa3, Szabo18kp, Szabo19rd, Szabo19ssa4, Szabo19si}.
This was in turn heavily based on pioneering work in \cite{Kishimoto01, IzumiMatui10, GoldsteinIzumi11, MatuiSato12, MatuiSato14} that exploited ideas of this nature long before.
Related recent results include \cite{Liao16, Liao17, Sato16, GardellaHirshberg18}.

The purpose of the present work can be seen as an attempt to reconsider and improve on the overarching framework and some of the elementary analytic methods involved in the classification of group actions up to cocycle conjugacy.
As motivation, let us recall a simple fact that has fundamental consequences for the Elliott classification program of \cstar-algebras:

\begin{theoreme}[cf.\ {\cite[Corollary 2.3.4]{Rordam} and \cite{Elliott10}}]
Let $A$ and $B$ be two separable \cstar-algebras.
Suppose that $\phi: A\to B$ and $\psi: B\to A$ are two (extendible)\footnote{To be precise, one is faced with a choice here for the original proof in \cite{Rordam} to be correct.
Either one has to assume that the involved $*$-homomorphisms are extendible (see \autoref{def:extendible}) or one has to use the notion of approximate innerness involving unitaries in the smallest unitizations of $A$ and $B$, respectively, instead of their multiplier algebras.
The latter is suggested on the Errata page for \cite{Rordam}; cf.\ \url{http://web.math.ku.dk/~rordam/Encyclopaedia.html}.}
$*$-homomorphisms such that the compositions $\psi\circ\phi$ and $\phi\circ\psi$ are approximately inner endomorphisms on $A$ and $B$, respectively.
Then $\phi$ and $\psi$ are approximately unitarily equivalent to mutually inverse isomorphisms between $A$ and $B$.
\end{theoreme}

The utility of this statement comes from the fact that it implicitly gives us a roadmap on how to classify a given category $\FC$ of separable \cstar-algebras by a functorial invariant ``\inv''.
The first step is to establish a \emph{uniqueness theorem} for the invariant, which typically asserts that two (sufficiently non-trivial) $*$-homomorphisms $\phi,\psi: A\to B$ with $A,B\in\FC$ are approximately unitarily equivalent if $\inv(\phi)=\inv(\psi)$.
The second step is to establish an \emph{existence theorem} for the invariant, which typically asserts that for $A,B\in\FC$, every suitable arrow $\inv(A)\to\inv(B)$ lifts to a $*$-homomorphism $A\to B$.
The third step is to tie it all together with the above theorem for $A,B\in\FC$: Combing the first two steps, an invertible arrow $\inv(A)\to\inv(B)$ gives rise to two $*$-homomorphisms $A\to B\to A$ as required by the theorem above, which implies that $A$ and $B$ are isomorphic.
This framework for classification has been at the heart of the Elliott program since the early Bratteli--Elliott classification \cite{Bratteli72, Elliott76} of AF algebras (with the ordered $K_0$-group playing the role of ``\inv''), and has essentially remained unchanged notwithstanding the staggeringly increasing depth behind the modern versions of the correct invariant as well as the uniqueness and existence theorems.

Despite the importance of attempting to classify \cstar-dynamical systems up to cocycle conjugacy \cite{Izumi10}, there is still no clear-cut analog of this roadmap in general.
For compact acting groups, it is relatively straightforward to come up with some suitable generalization of the statement in the above theorem, and to come up with analogous uniqueness and existence theorems to classify actions up to genuine conjugacy.
In fact this viewpoint has been leveraged to classify actions with the Rokhlin property in \cite{GardellaSantiago16, BarlakSzaboVoigt17}.

For non-compact acting groups, on the other hand, it is much less clear what an analogous roadmap for classification would be from consulting the literature.
In light of some very elementary examples \cite[Section 2.2]{Izumi10}, it is clear that classification up to genuine conjugacy on sufficiently noncommutative \cstar-algebras is not feasible using a computable invariant, and has to be weakened to cocycle conjugacy instead.
The usual way to classify actions of groups like $\IZ$ or $\IR$, as pioneered by Kishimoto, follows a somewhat different path to the above:
Instead of the uniqueness theorem, one needs to bother with establishing a \emph{cohomology-vanishing} statement for the group actions under consideration.
Instead of the existence theorem, one needs to determine when two actions on the same \cstar-algebra are \emph{approximate cocycle perturbations} of each other.
Instead of the intertwining theorem above, one then combines these pieces with a fairly involved \emph{Evans--Kishimoto intertwining} argument (cf.\ \cite{EvansKishimoto97}), which directly gives one cocycle conjugacy between two dynamical systems.
One may certainly argue that this runs somewhat parallel to the \emph{three step recipe} sketched in the context of classifying \cstar-algebras, but the analogy only goes so far.
Despite being an undoubtedly ingenious method, the Evans--Kishimoto intertwining argument does not shed much light on the category of \cstar-dynamics being classified as a whole, but rather helps one to compare cocycle conjugacy classes directly.
In a way, the assumptions entering into the Evans--Kishimoto intertwining argument to relate two actions $\alpha,\beta: G\curvearrowright A$ should (in general) be seen as extremely strong by design, in order to ensure that some kind of \emph{existence proof} for both $\alpha\to\beta$ and $\beta\to\alpha$ and the \emph{uniqueness proof} for their compositions is taken care of in one fell swoop.
It is therefore not surprising that classification via this method gets increasingly difficult to implement with more complicated acting groups, as the necessary assumptions for performing the intertwining argument appear to be more and more difficult to handle; see \cite{IzumiMatui18, IzumiMatui19} for recent state of the art results that serve as a good example for this phenomenon.

The subject of this paper is to provide a categorical framework in which Elliott's approach to classification carries over directly to the context of \cstar-dynamics.
The obvious first obstacle is given by the fact that the category of $G$-\cstar-algebras, with $G$ being a locally compact group, is usually equipped with the genuine equivariant $*$-homomorphisms as morphisms.
The naive way to overcome it is to add the least amount of extra information to the category as to include the language necessary to describe cocycle conjugacy:

\begin{defii} \label{definition-A}
Let $G$ be a locally compact group and let $\alpha: G\curvearrowright A$ and $\beta: G\curvearrowright B$ be two actions on \cstar-algebras.
A \emph{cocycle representation} from $(A,\alpha)$ to $(\CM(B),\beta)$ is a pair $(\phi,\Iu)$ consisting of a $*$-homomorphism $\phi: A\to\CM(B)$ and a $\beta$-cocycle $\Iu: G\to\CU(\CM(B))$ satisfying $\ad(\Iu_g)\circ\beta_g\circ\phi=\phi\circ\alpha_g$ for all $g\in G$.
If $\phi(A)\subseteq B$, we call the pair $(\phi,\Iu)$ a \emph{cocycle morphism} from $(A,\alpha)$ to $(B,\beta)$.
\end{defii}

Notice that in the special case $\Iu=\eins$ above, we simply recover what it means for the $*$-homomorphism $\phi$ to be equivariant with respect to $\alpha$ and $\beta$.
On the other hand, if $(A,\alpha)=(B,\beta)$ and $\phi=\id_A$, we recover what it means for two $G$-actions on the same \cstar-algebra to be exterior equivalent via the cocycle $\Iu$.
In other words, the usual category of $G$-\cstar-algebras is enlarged by declaring that exterior equivalences are also morphisms.
It turns out that there is a natural way to define a composition between cocycle morphisms\footnote{Similarly to the previous footnote, one is again faced with two possibilities here, depending on whether one prefers to work with unitaries in unitizations or multiplier algebras.
Both perspectives merit investigation and will be pursed in the paper.}; 
see \autoref{prop:cocycle-composition}.
This enables one to define a larger category of $G$-\cstar-algebras, which we choose to call the \emph{cocycle category}.
Anybody who is familiar with the concept of cocycle conjugacy will certainly not be taken by surprise by the fact that in this category, an isomorphism between $G$-\cstar-algebras is precisely a cocycle conjugacy in the known sense.

The rigorous introduction to this framework along with some elementary examples and observations will occupy most of the first section.
The level of generality in the main body of the paper is higher than indicated up to this point, because this framework makes sense and will be developed in the context of twisted $G$-\cstar-dynamical systems à la Busby--Smith \cite{BusbySmith70, PackerRaeburn89}.
It is noteworthy that the categorical framework pursued here is not entirely new.
A similar framework was conceived in a higher category approach to (twisted) \cstar-dynamics due to Buss--Meyer--Zhu \cite{BussMeyerZhu13} and in the context of crossed product dualities by Kaliszewski--Omland--Quigg \cite{KaliszewskiOmlandQuigg15}, although ultimately used for different purposes than here.
In the terminology of \cite{BussMeyerZhu13}, a \emph{weakly equivariant map} or \emph{transformation} is precisely a cocycle morphism $(\phi,\Iu)$ with the extra assumption that the map $\phi$ is non-degenerate.
Looking back even further, one may argue for $G=\IR$ that traces of this idea can be found in Kishimoto's work \cite{Kishimoto04}, where in the context of a flow $\alpha: \IR\curvearrowright A$, the cocycle conjugacies $(A,\alpha)\to (A,\alpha)$ have been studied as \emph{core symmetries} and identified as an interesting invariant.

The second section deals with inductive limits.
We show that in the cocycle category, sequential inductive limits always exist, and can be constructed explicitly in the language of the ordinary \cstar-algebra inductive limit construction.
We furthermore discuss topologies on the Hom-sets, which for pairs $(\phi,\Iu)$ as in \autoref{definition-A} will usually boil down to the point-norm topology in the first variable, and the uniform strict topology over compact subsets of $G$ in the second variable.
This gives rise to a suitable notion of approximate unitary equivalence; cf.\ \autoref{def:cocycle-hom-equivalences} and \autoref{rem:ue-not-symmetric}:

\begin{defii}
Let $(\phi,\Iu), (\psi,\Iv): (A,\alpha)\to (B,\beta)$ be two non-degenerate cocycle morphisms.
We say $(\phi,\Iu)$ is approximately unitarily equivalent to $(\psi,\Iv)$, if there exists a net $u_\lambda\in\CU(\CM(B))$ such that 
\[
\lim_{\lambda\to\infty} ( \ad(u_\lambda)\circ\phi, u_\lambda\Iu_\bullet\beta_\bullet(u_\lambda)^* ) = (\psi,\Iv).
\]
In the special case where $(B,\beta)=(A,\alpha)$ and $(\phi,\Iu)=(\id_A,\eins)$, we say that $(\psi,\Iv)$ is approximately inner. 
\end{defii}

The third and fourth sections are concerned with providing detailed generalizations of the elementary intertwining results commonly employed in the classification program of \cstar-algebras, in particular utilizing approximate unitary equivalence.
This part of the paper can be seen as a variant of some early sections in \cite{Rordam} for \cstar-dynamics.
There are several layers of generality treated in this context, but we shall for now cherry-pick the following result from the third section to illustrate the parallels to the theorem above; the following is a special case of \autoref{cor:special-ue-intertwining}:

\begin{theoremi} \label{theorem-C}
Let $(A,\alpha)$ and $(B,\beta)$ be two separable $G$-\cstar-algebras.
Suppose that $(\phi,\Iu): (A,\alpha)\to (B,\beta)$ and $(\psi,\Iv): (B,\beta)\to (A,\alpha)$ are two non-degenerate cocycle morphisms such that the compositions $(\psi,\Iv)\circ(\phi,\Iu)$ and $(\phi,\Iu)\circ(\psi,\Iv)$ are approximately inner on $(A,\alpha)$ and $(B,\beta)$, respectively.
Then $(\phi,\Iu)$ and $(\psi,\Iv)$ are approximately unitarily equivalent to mutually inverse cocycle conjugacies between $(A,\alpha)$ and $(B,\beta)$.
\end{theoremi} 

In a nutshell, this means that the framework discussed here comes equipped with a clear-cut roadmap for classification of \cstar-dynamics up to cocycle conjugacy, involving \emph{uniqueness} and \emph{existence} theorems just like explained before in the context of the Elliott program.
In some cases, this may provide a useful alternative to classification via the Evans--Kishimoto intertwining method.
A good example of this is the homotopy rigidity result \cite[Theorem F]{Szabo18kp}\footnote{We note that the current published version of this result \cite[Corollary 6.5]{Szabo18kp} mistakenly claims in its conclusion the conjugacy of two actions instead of cocycle conjugacy.
The correct statement is proved for example via \autoref{thm:special-pue-intertwining}, whereas the published version tried to appeal to an erronous version of this intertwining theorem from \cite{Szabo18ssa}, which has since then been corrected in \cite{Szabo18ssaCOR}.} 
for outer actions of amenable groups on Kirchberg algebras, which to my knowledge nobody can currently prove via Evans--Kishimoto intertwining.
In fact, the framework developed here and the intertwining results will be a crucial ingredient in forthcoming work joint with James Gabe, in which we classify outer actions of discrete amenable groups on Kirchberg algebras via equivariant Kasparov theory.
The content of the last section can be seen as preliminary evidence why the framework appears to be useful for extracting information about \cstar-dynamics from $K$-theoretical data, which will be fleshed out further in our forthcoming work with James Gabe.
Our approach to classification of \cstar-dynamics will as a result substantially differ from previous work in this area, such as for example the recent articles \cite{IzumiMatui18, IzumiMatui19} handling the case of poly-$\IZ$ groups based on variants of the Evans--Kishimoto intertwining method.

We note that the main result of \cite{Elliott10} has been an apparent template for \autoref{theorem-C}, and in fact the former directly implies the latter under the assumpion that both $A$ and $B$ are unital.
However, the assumptions in the main result of \cite{Elliott10} no longer apply in any obvious way to the non-unital case.
Since we are also interested in comparing more general inductive limits via intertwining arguments, which is not handled in \cite{Elliott10}, we will build the intertwining theory in our context from scratch.

In the fourth section, we will also deal with some one-sided intertwining arguments, as well as prove a version of \autoref{theorem-C} involving asymptotic unitary equivalence; see \autoref{thm:cont-intertwining}.\footnote{To the best of my knowledge, the non-dynamical version of this fact is folklore but has so far remained unpublished.}
In the fifth section we briefly revisit the concept of strong self-absorption for \cstar-dynamical systems, which turns out to be nicely expressible in the framework of this paper.
The main result of that section is a stronger version of the existing equivariant McDuff-type theorems \cite{Szabo18ssa, Szabo17ssa3} for non-unital \cstar-algebras (see \autoref{thm:McDuff}), which holds under the assumption of equivariant Jiang--Su stability, and arises as a straightforward consequence of an intertwining argument in the preceding section.
In the sixth section, we take a look at $G$-equivariant Kasparov theory, and show that it naturally extends to a (bi-)functor on the cocycle category of $G$-\cstar-algebras.
This holds true both in Kasparov's original Fredholm picture \cite{Kasparov88}, as well as the \emph{Cuntz--Thomsen} picture, which was in fact conceptualized in Thomsen's original work \cite{Thomsen98} by using cocycle representations in the sense of \autoref{definition-A}.
We also take the opportunity to show that a certain continuity assumption in the definition of an equivariant Cuntz pair describing elements in a $KK^G$-group is redundant; cf.\ \autoref{prop:automatic-continuity}.\bigskip

\textbf{Acknowledgement.}
The key external impulse to pursue this line of research came to me while I was giving a lecture series at the NCGOA 2018 in Münster, the subject of which boiled down to the classification of group actions via the Evans--Kishimoto intertwining method.
During the lengthy discussion of its technical details, George Elliott made several suggestive and open-ended remarks about the framework as a whole, insisting (or perhaps speculating) that it ought to be possible to intertwine at the level of the morphisms connecting two \cstar-dynamics.
To make a long story short, this paper can in large part be understood as a rigorous positive answer to said suspicion, and I would like to thank George Elliott for voicing it.

%%%%%%%%%%%%%%%%%%%%%%%%%%%%%%%%%%%%%%%%%%%%%%%%%%%%%

\section{Introduction to the cocycle category}

\subsection{Preliminaries}

Throughout, $G$ will denote a second-countable, locally compact group.
The capital letters $A,B,C$ will usually denote \cstar-algebras. 
For elements $a$ and $b$ in some \cstar-algebra, we frequently write $a=_\eps b$ as short-hand for $\|a-b\|\leq\eps$.
The multiplier algebra of $A$ is denoted as $\CM(A)$, whereas its unitary group will be denoted by $\CU(A)$ if $A$ is unital.
We define $\CU(\eins+A)=(\eins+A)\cap\CU(A^\dagger)$, where $A^\dagger$ is the \cstar-algebra arising from adding a unit to $A$.
One has a canonical isomorphism of groups $\CU(A)\cong\CU(\eins+A)$ in case $A$ is unital, even though the two sets are not equal.
Normal alphabetical letters such as $u,v,U,V$ will be frequently used for unitary elements in \cstar-algebras.
For a given unitary $u\in\CU(\CM(A))$ or $u\in\CU(\eins+A)$, we denote by $\ad(u)\in\Aut(A)$ the inner automorphism given by $a\mapsto uau^*$.
Greek letters such as $\alpha,\beta,\gamma$ are used for point-norm continuous maps $G\to\Aut(A)$. 
Most of the time these are either group actions or part of a twisted action on $A$.
Fraktur letters such as $\Fu, \Fv, \Fw$ are used for strictly continuous maps $G\times G\to\CU(\CM(A))$, and usually denote the 2-cocycles belonging to a twisted action on $A$.
Double-struck letters such as $\Iu, \Iv, \IU, \IV$ are used for strictly continuous maps $G\to\CU(\CM(A))$.
In practice, these are often 1-cocycles with respect to a $G$-action or more generally are part of a \emph{cocycle morphism}, as defined in the main body of the article.

\begin{defi}
Let $G$ be a locally compact group.
A twisted action of $G$ on a \cstar-algebra $A$ is a pair $(\alpha,\Fu)$ for a point-norm continuous map $\alpha: G\to\Aut(A)$ and a strictly continuous map $\Fu: G\times G\to\CU(\CM(A))$ satisfying
\[
\alpha_1=\id_A,\quad \alpha_s\circ\alpha_t = \ad(\Fu_{s,t})\circ\alpha_{st}
\]
and
\[
\Fu_{s,1}=\Fu_{1,s}=\eins,\quad \alpha_r(\Fu_{s,t})\Fu_{r,st} = \Fu_{r,s}\Fu_{rs,t}
\]
for all $r,s,t\in G$.
We commonly write $(\alpha,\Fu): G\curvearrowright A$.
In this context, the map $\Fu$ is referred to as a 2-cocycle, and the triple $(A,\alpha,\Fu)$ is referred to as a twisted $G$-\cstar-algebra.

In the case that $\Fu$ is norm-continuous and takes values in $\CU(\eins+A)$, we say that $(\alpha,\Fu)$ is a gently twisted action and call the triple $(A,\alpha,\Fu)$ a gently twisted $G$-\cstar-algebra.
\end{defi}

\begin{rem} \label{rem:genuine-actions}
In the special case $\Fu=\eins$, we recover what it means for $\alpha$ to be a $G$-action, and the pair $(A,\alpha)$ is then called a $G$-\cstar-algebra, which we will identify with the triple $(A,\alpha,\eins)$ when we view it as a (gently) twisted $G$-\cstar-algebra in the canonical way.
\end{rem}

Let $(A,\alpha,\Fu)$ be a twisted $G$-\cstar-algebra as in the above definition.
Sometimes it will be important to keep track of the triple $(\CM(A),\alpha,\Fu)$, which in general is not a twisted $G$-\cstar-algebra in the above sense.
The canonical induced map $\alpha: G\to\Aut(\CM(A))$ is only point-strictly continuous instead of point-norm continuous, but all the other axioms of a twisted $G$-\cstar-algebra make sense for this triple.

In what follows, we need to recall a few basic concepts from the theory of twisted actions; see \cite{PackerRaeburn89}.

\begin{rem} \label{rem:twisted-crossed-product}
Let $(A,\alpha,\Fu)$ be a twisted $G$-\cstar-algebra.
Recall that the (maximal) twisted crossed product $A\rtimes_{\alpha,\Fu} G$ comes with a covariant representation $(\pi,\Iu)$ of $(A,\alpha,\Fu)$, i.e., a non-degenerate $*$-homom\-orphism $\pi: A\to\CM(A\rtimes_{\alpha,\Fu} G)$ and a strictly continuous map $\Iu: G\to\CU(\CM(A\rtimes_{\alpha,\Fu} G))$ such that
\[
\Iu_g\Iu_h = \pi(\Fu_{g,h})\Iu_{gh} \quad\text{and}\quad \ad(\Iu_g)\circ\pi = \pi\circ\alpha_g
\]
for all $g,h\in G$.
This covariant representation has the universal property that every other one factors through it.
To be more precise, whenever $B$ is another \cstar-algebra and the pair $\pi^B: A\to\CM(B)$ and $\Iu^B: G\to\CU(\CM(B))$ forms another (non-degenerate) covariant representation into $\CM(B)$, then there exists a unique non-degenerate $*$-homomorphism $\Phi: A\rtimes_{\alpha,\Fu} G\to \CM(B)$ such that $\pi^B=\Phi\circ\pi$ and $\Iu^B=\Phi\circ\Iu$.
\end{rem}

\begin{defi} \label{def:ext-equ}
Let $(\alpha,\Fu): G\curvearrowright A$ be a twisted action.
Given any strictly continuous map $\Is: G\to\CU(\CM(A))$ with $\Is_1=\eins$, we may define a new twisted action $(\alpha,\Fu)^\Is := (\alpha^\Is,\Fu^\Is)$ via $\alpha^\Is_g=\ad(\Is_g)\circ\alpha_g$ and $\Fu^\Is_{g,h}=\Is_g\alpha_g(\Is_h)\Fu_{g,h}\Is_{gh}^*$ for all $g,h\in G$.
If there exists a map $\Is$ such that $\Fu^\Is=\eins$, then the 2-cocycle $\Fu$ is called a coboundary, and it means that $(\alpha,\Fu)$ is equivalent to a genuine action in the sense below.
\end{defi}

\begin{defi}
Two twisted actions $(\alpha,\Fu), (\beta,\Fv): G\curvearrowright A$ on the same \cstar-algebra are called exterior equivalent, if there exists a strictly continuous map $\Is: G\to\CU(\CM(A))$ with $\Is_1=\eins$ such that $(\beta,\Fv)=(\alpha,\Fu)^\Is$.
\end{defi}

\begin{rem}
In light of \autoref{rem:genuine-actions}, if we restrict \autoref{def:ext-equ} to genuine actions $\alpha, \beta: G\curvearrowright A$, then we write $\alpha^\Is=\beta$ to mean $(\alpha^\Is,\eins)=(\beta,\eins)$, which automatically forces $\Is$ to be an $\alpha$-1-cocycle (or briefly $\alpha$-cocycle), and $\beta$ to be the associated cocycle perturbation. 
That is, we have $\Is_{gh}=\Is_g\alpha_g(\Is_h)$ for all $g,h\in G$.
\end{rem}

\begin{defi} \label{def:extendible}
A $*$-homomorphism $\phi: A\to \CM(B)$ is called extendible, if for any increasing approximate unit $e_\lambda\in A$, the net $\phi(e_\lambda)\in\CM(B)$ converges strictly to a projection $p$ in $\CM(B)$.\footnote{In particular, if $\phi$ is non-degenerate, then it is extendible.}
In this case $\phi$ factorizes through $\CM(pBp)\cong p\CM(B)p\subseteq \CM(B)$, where it becomes a non-degenerate $*$-homomorphism $\phi_p: A\to \CM(pBp)$, which uniquely extends to a $*$-homomorphism from the multiplier algebra $\CM(A)$ that becomes strictly continuous on the unit ball.

The map $\phi^+: \CU(\CM(A))\to\CU(\CM(B))$ given by 
\[
\phi^+(u)=\phi_p(u)+(\eins_{\CM(B)}-p)
\]
is then a well-defined, strictly continuous homomorphism between unitary groups.
It satisfies the formula $\phi^+(u)\phi(a)=\phi(ua)$ for all $a\in A$ and $u\in\CU(\CM(A))$, which we will frequently use without mention. 
If $\phi$ happens to be non-degenerate, then it extends to a unital $*$-homomorphism $\CM(A)\to\CM(B)$, in which case we can write $\phi(u)$ instead of $\phi^+(u)$.
\end{defi}

\begin{rem}
To avoid confusion, here is a word of caution for the reader.
Depending on context, we may use the similarly looking expression $\phi^\dagger$ to denote the unital $*$-homomorphism $A^\dagger\to B^\dagger$ induced by a $*$-homomorphism $\phi: A\to B$.
Under this notation, if $\phi$ is assumed to be extendible and $u$ is a unitary in $A^\dagger$, then the unitary multiplier in $\CM(B)$ arising from $\phi^\dagger(u)$ in $B^\dagger$ agrees with the multiplier $\phi^+(u)$ in $\CM(B)$ in the sense of the definition above.
In particular, these two pieces of notation, which could a priori mean different things, always yield the same elements in $\CM(B)$ whenever both are defined.
\end{rem}

\begin{rem}
If $\phi: A\to \CM(B)$ and $\psi: B\to \CM(C)$ are two extendible $*$-homomorphisms with corresponding projections $p\in\CM(B)$ and $q\in\CM(C)$, then one may compose them via
\[
A\stackrel{\phi_p}{\longrightarrow} \CM(pBp) \subseteq \CM(B) \stackrel{\psi_q}{\longrightarrow} \CM(qCq) \subseteq \CM(C)
\]
and thus write $\psi\circ\phi: A\to\CM(C)$.
This is clearly again an extendible $*$-homomorphism with respect to the corner spanned by $\psi_q(p)\leq q$ in $\CM(C)$.
It follows moreover that $(\psi\circ\phi)^+ = \psi^+\circ\phi^+$ as group homomorphisms $\CU(\CM(A))\to\CU(\CM(C))$.
We will use all this without further mention.
\end{rem}

\subsection{Cocycle representations and morphisms}

\begin{defi} \label{def:cocycle-hom}
Let $(\alpha,\Fu): G\curvearrowright A$ and $(\beta,\Fv): G\curvearrowright B$ be two twisted actions on \cstar-algebras.
\begin{enumerate}[label=\textup{(\roman*)},leftmargin=*]
\item \label{def:cocycle-hom:1}
A \emph{cocycle representation} from $(A,\alpha,\Fu)$ to $(\CM(B),\beta,\Fv)$ is a pair $(\psi,\Iv)$, where
$\psi: A\to \CM(B)$ is an extendible $*$-homomorphism and $\Iv: G\to\CU(\CM(B))$ is a strictly continuous map such that
\begin{equation} \label{eq:equivariance}
\ad(\Iv_g)\circ\beta_g\circ\psi = \psi\circ\alpha_g
\end{equation}
and
\begin{equation} \label{eq:cocycle-identity}
\psi^+(\Fu_{g,h}) = \Iv_g\beta_g(\Iv_h)\Fv_{g,h}\Iv_{gh}^*
\end{equation}
for all $g,h\in G$.
\item \label{def:cocycle-hom:2}
A \emph{cocycle morphism} from $(A,\alpha,\Fu)$ to $(B,\beta,\Fv)$ is a cocycle representation $(\psi,\Iv)$ as above, with the additional requirement that $\psi(A)\subseteq B$.
We will write $(\psi,\Iv): (A,\alpha,\Fu) \to (B,\beta,\Fv)$.
The set of all such pairs will be denoted by $\cohom((\alpha,\Fu),(\beta,\Fv))$.
\item \label{def:cocycle-hom:3}
Suppose that $(\alpha,\Fu)$ and $(\beta,\Fv)$ are both gently twisted actions.
A \emph{proper cocycle morphism} from $(A,\alpha)$ to $(B,\beta)$ is a pair $(\psi,\Iv): (A,\alpha,\Fu) \to (B,\beta,\Fv)$, where $\psi: A\to B$ is an arbitrary $*$-homomor\-phism, $\Iv$ is a norm-continuous map taking values in $\CU(\eins+B)$, and such that the conditions \eqref{eq:equivariance} and \eqref{eq:cocycle-identity} hold, with ``$\psi^+(\Fu_{g,h})$'' being replaced by ``$\psi^\dagger(\Fu_{g,h})$''.
The set of all such pairs will be denoted by $\cohomr((\alpha,\Fu),(\beta,\Fv))$.
\end{enumerate}
We will refer to \eqref{eq:equivariance} as the \emph{equivariance condition} and to \eqref{eq:cocycle-identity} as the \emph{cocycle identity} for the pair $(\psi,\Iv)$.
In the special case that $\Iv=\eins$ is the trivial map, we identify $\psi$ with the pair $(\psi,\eins)$ and say that
\[
\psi: (A,\alpha,\Fu) \to (B,\beta,\Fv) \quad\text{or}\quad \psi: (A,\alpha,\Fu) \to (\CM(B),\beta,\Fv)
\]
is an equivariant $*$-homomorphism.
\end{defi}

\begin{rem}
In the above definition, the equivariance condition implies $\Iv_1=\eins$ due to the fact that we have $\Fu_{1,1}=\eins_A$, $\Fv_{1,1}=\eins_B$ and $\beta_1=\id_B$.
\end{rem}

\begin{rem} \label{rem:general-coc-reps}
If we restrict \autoref{def:cocycle-hom} to the case of genuine $G$-\cstar-algebras $(A,\alpha)$ and $(B,\beta)$, then a cocycle representation
\[
(\psi,\Iv): (A,\alpha) \to (\CM(B),\beta)
\]
is a pair such that $\Iv: G\to\CU(\CM(B))$ is a $\beta$-cocycle in the ordinary sense by \eqref{eq:cocycle-identity}, and \eqref{eq:equivariance} means that the corresponding cocycle perturbation $\beta^\Iv: G\curvearrowright B$ turns $\psi$ into an $\alpha$-$\beta^\Iu$-equivariant $*$-homomorphism.
In this particular case, we may sometimes call such a pair $(\psi,\Iv)$ a cocycle representation even if $\psi$ is not assumed to be extendible.
Although this definition is too weak to define a category as we will do below, it is nevertheless relevant in the description of $G$-equivariant $KK$-theory as treated in the last section.
The analogous observation as above is true for cocycle morphisms to $(B,\beta)$, and being a proper cocycle morphism additionally means that $\Iv$ is a (norm-continuous) cocycle in $\CU(\eins+B)$.
\end{rem}

\begin{example}
Given a twisted action $(\alpha,\Fu): G\curvearrowright A$, a covariant representation into $\CM(B)$ in the sense of \autoref{rem:twisted-crossed-product} is just a non-degenerate cocycle representation $(\pi^B,\Iu^B): (A,\alpha,\Fu)\to (\CM(B),\id,\eins)$.
\end{example}

\begin{example} \label{ex:exterior-equivalence}
Let $(\alpha,\Fu): G\curvearrowright A$ be a twisted action.
For any strictly continuous map $\Is: G\to\CU(\CM(A))$ with $\Is_1=\eins$, we obtain (by definition) a cocycle morphism
\[
(\id_A,\Is): (A,\alpha^\Is,\Fu^\Is) \to (A,\alpha,\Fu).
\]
We refer to such an example as an \emph{exterior equivalence}.
If $\Is$ is norm-continuous and takes values in $\CU(\eins+A)$, then we call this a proper exterior equivalence.
In particular, if $\alpha$ is a genuine action and $\Iu$ is an $\alpha$-cocycle, then
\[
(\id_A,\Iu): (A,\alpha^\Iu)\to (A,\alpha)
\]
is an exterior equivalence, and it is moreover a proper exterior equivalence precisely when $\Iu$ is a norm-continuous map with values in $\CU(\eins+A)$.\footnote{We will see later in \autoref{prop:automatic-continuity} that norm-continuity of $\Iu$ is a redundant assumption that happens to hold automatically if it takes values in $\CU(\eins+A)$.}
\end{example}

\subsection{The cocycle category}

\begin{prop} \label{prop:cocycle-composition}
Let $(\alpha,\Fu): G\curvearrowright A$, $(\beta,\Fv): G\curvearrowright B$ and $(\gamma,\Fw): G\curvearrowright C$ be three twisted actions on \cstar-algebras.
\begin{enumerate}[label=\textup{(\roman*)},leftmargin=*]
\item \label{prop:cocycle-composition:1}
Suppose that
\[
(A,\alpha,\Fu) \stackrel{(\phi,\Iu)}{\longrightarrow} (\CM(B),\beta,\Fv) \quad\text{and}\quad (B,\beta,\Fv)
\stackrel{(\psi,\Iv)}{\longrightarrow} (\CM(C),\gamma,\Fw)
\]
are two cocycle representations. 
Then the composition, which we define as the pair
\[
(\psi,\Iv)\circ (\phi,\Iu) := (\psi\circ\phi, \psi^+(\Iu)\cdot\Iv)
\]
defines a cocycle representation from $(A,\alpha,\Fu)$ to $(\CM(C),\gamma,\Fw)$.
This binary operation is associative.
\item \label{prop:cocycle-composition:2}
In this sense, the composition of two cocycle morphisms is again a cocycle morphism.
\item \label{prop:cocycle-composition:3}
Suppose that all three of them are gently twisted actions.
If
\[
(A,\alpha) \stackrel{(\phi,\Iu)}{\longrightarrow} (B,\beta) \quad\text{and}\quad (B,\beta)
\stackrel{(\psi,\Iv)}{\longrightarrow} (C,\gamma)
\]
are two proper cocycle morphisms, then the composition formula above, with ``$\psi^+$'' replaced by ``$\psi^\dagger$'', defines a proper cocycle morphism.
\end{enumerate}
\end{prop}
\begin{proof}
\ref{prop:cocycle-composition:1}: Let us first verify the equivariance condition.
Using the equivariance condition for both pairs $(\phi,\Iu)$ and $(\psi,\Iv)$, one has for all $g\in G$ that
\[
\begin{array}{ccl}
\ad(\psi^+(\Iu_g)\Iv_g)\circ\gamma_g\circ\psi\circ\phi &=& \ad(\psi^+(\Iu_g))\circ\Big(\ad(\Iv_g)\circ\gamma_g\circ\psi\Big)\circ\phi \\
&=& \ad(\psi^+(\Iu_g))\circ\psi\circ\beta_g\circ\phi \\
&=& \psi\circ\Big( \ad(\Iu_g)\circ\beta_g\circ\phi\Big) \\
&=& \psi\circ\phi\circ\alpha_g.
\end{array}
\]
Let us now verify the cocycle identity.
For all $g,h\in G$, we use equivariance and cocycle identities for our given cocycle representations to see that
\[
\begin{array}{cl}
\multicolumn{2}{l}{ \psi^+(\Iu_g)\Iv_g\gamma_g\big(\psi^+(\Iu_h)\Iv_h\big)\Fw_{g,h}\Iv_{gh}^*\psi^+(\Iu_{gh})^* } \\
=& \psi^+(\Iu_g) \Big( \Iv_g\gamma_g\big(\psi^+(\Iu_h)\big)\Iv_g^* \Big) \Big( \Iv_g\gamma_g(\Iv_h)\Fw_{g,h}\Iv_{gh}^* \Big)\psi^+(\Iu_{gh}^*) \\
=& \psi^+(\Iu_g) \psi^+(\beta_g(\Iu_h)) \psi^+(\Fv_{g,h}) \psi^+(\Iu_{gh})^* \\
=& \psi^+\Big( \Iu_g\beta_g(\Iu_h)\Fv_{g,h}\Iu_{gh}^* \big) \ = \ (\psi\circ\phi)^+(\Fu_{g,h}).
\end{array}
\]
So the pair $(\psi\circ\phi,\psi^+(\Iu)\Iv)$ is indeed a cocycle representation from $(A,\alpha,\Fu)$ to $(\CM(C),\gamma,\Fw)$.

In order to obtain associativity, let $(\delta,\Fx): G\curvearrowright D$ be a twisted action and $(\theta,\Iw): (C,\gamma,\Fw) \to (\CM(D),\delta,\Fx)$ another cocycle representation.
Then evidently $\theta\circ (\psi\circ\phi) = (\theta\circ\psi)\circ\phi$ and 
\[
\theta^+\big( \psi^+(\Iu_g)\Iv_g \big)\cdot \Iw_g = (\theta\circ\psi)^+(\Iu_g)\cdot \psi^+(\Iv_g)\Iw_g
\]
for all $g\in G$, which by definition yields
\[
(\theta,\Iw)\circ\big( (\psi,\Iv)\circ(\phi,\Iu) \big) = \big( (\theta,\Iw)\circ (\psi,\Iv) \big)\circ(\phi,\Iu)
\]

Part \ref{prop:cocycle-composition:2} is trivial, and \ref{prop:cocycle-composition:3} follows from the same computations if we replace all instances of ``$\psi^+$'' with ``$\psi^\dagger$''.
\end{proof}

\begin{defi} \label{def:cocycle-category}
We define the following categories:
\begin{enumerate}[label=\textup{(\roman*)},leftmargin=*]
\item The \emph{twisted cocycle category} $\cstar_{G,t}$ is defined as the category with objects being the twisted $G$-\cstar-algebras, and the morphisms from $(A,\alpha,\Fu)$ to $(B,\beta,\Fv)$ being the cocycle morphisms
\[
(\phi,\Iu): (A,\alpha,\Fu) \to (B,\beta,\Fv).
\]
Composition of morphisms in $\cstar_{G,t}$ is defined via the operation in \autoref{prop:cocycle-composition}.
On any object $(A,\alpha,\Fu)$, the identity morphism in this category is given by $(\id_A,\eins)$, which we will sometimes abbreviate as $\id_A$ if there is no cause for confusion.
\item The \emph{cocycle category} $\cstar_G$ is the full subcategory of $\cstar_{G,t}$ with objects being all the genuine $G$-\cstar-algebras $(A,\alpha)$.
\item The \emph{twisted proper cocycle category} $\cstar_{G,t,\Fp}$ has the gently twisted $G$-\cstar-algebras as objects, and proper cocycle morphisms as morphisms.
The full category given by all genuine $G$-\cstar-algebras is called the \emph{proper cocycle category} and is denoted $\cstar_{G,\Fp}$.
\end{enumerate}
\end{defi}

\begin{rem}
Any cocycle morphism $(\phi,\Iu): (A,\alpha,\Fu)\to (B,\beta,\Fv)$ can always be expressed as the composition
\[
(A,\alpha,\Fu) \stackrel{(\phi,\eins)}{\longrightarrow} (B,\beta^\Iu,\Fv^\Iu) \stackrel{(\id_B,\Iu)}{\longrightarrow} (B,\beta,\Fv)
\]
of a genuine equivariant $*$-homomorphism and an exterior equivalence.
\end{rem}

\begin{rem}
A cocycle morphism $(\phi,\Iu): (A,\alpha,\Fu)\to (B,\beta,\Fv)$ is invertible in the category $\cstar_{G,t}$ if and only if $\phi$ is an isomorphism between \cstar-algebras.
The inverse is then given by the formula
\[
(\phi,\Iu)^{-1} = (\phi^{-1},\phi^{-1}(\Iu)^*): (B,\beta,\Fv) \to (A,\alpha,\Fu).
\]
We will refer to an invertible cocycle morphism as a \emph{cocycle conjugacy}.

As the terminology suggests, there exists a cocycle conjugacy between the twisted $G$-\cstar-algebras $(A,\alpha,\Fu)$ and $(B,\beta,\Fv)$ if and only if the twisted actions $(\alpha,\Fu): G\curvearrowright A$ and $(\beta,\Fv): G\curvearrowright B$ are cocycle conjugate in the ordinary sense.

In an analogous fashion, we define a \emph{proper cocycle conjugacy} between gently twisted $G$-\cstar-algebras as an isomorphism in the category $\cstar_{G,t,\Fp}$.
This is a stronger requirement than cocycle conjugacy in the sense that exterior equivalence has to be implemented by unitaries coming directly from the (unitized) underlying \cstar-algebra rather than its multiplier algebra.
\end{rem}

\begin{rem} \label{rem:from-genuine-to-twisted}
Let $\alpha: G\curvearrowright A$ be a genuine action, and $(\beta,\Fv): G\curvearrowright B$ a twisted action.
If there exists a cocycle representation $(\phi,\Iu): (A,\alpha)\to (\CM(B),\beta,\Fv)$, then the cocycle identity \eqref{eq:cocycle-identity} immediately implies that $\Fv$ is a coboundary, i.e., $(\beta,\Fv)$ is exterior equivalent to a genuine action.
Note that this is the case even when $A=\IC$.

The reverse statement is not true, on the other hand, as there generally exist many cocycle representations from $(B,\beta,\Fv)$ to $(\CM(\CK),\id,\eins)$, which are nothing but covariant representations, provided that $B$ is separable.
\end{rem}

\begin{example} \label{ex:inner-morphisms}
Let $(\alpha,\Fu): G\curvearrowright A$ be a twisted action.
Then any unitary element $v\in\CU(\CM(A))$ gives rise to the cocycle conjugacy
\[
\big( \ad(v),v\alpha_\bullet(v)^*\big): (A,\alpha,\Fu)\to (A,\alpha,\Fu).\footnote{Throughout the paper we will occasionally use this bullet-point notation where convenient. In this instance ``$v\alpha_\bullet(v)^*$'' denotes the map $G\to\CU(\CM(A))$ given by the assignment $g\mapsto v\alpha_g(v)^*$.}
\]
If it is clear from context that we are making statements within the category $\cstar_{G,t}$, then we simply write $\ad(v):=\big( \ad(v),v\alpha_\bullet(v)^*\big)$.
\end{example}
\begin{proof}
Equivariance follows from
\[
\ad(v\alpha_g(v)^*)\circ\alpha_g\circ\ad(v) = \ad(v)\circ\alpha_g
\]
and the cocycle identity follows from
\[
\begin{array}{ccl}
v\alpha_g(v)^*\alpha_g\big(v\alpha_h(v)^*\big)\Fu_{g,h}\alpha_{gh}(v)v^* &=& v(\alpha_g\circ\alpha_h)(v)^*\Fu_{g,h}\alpha_{gh}(v)v^* \\
&=& v\Fu_{g,h}v^* \ = \ \ad(v)(\Fu_{g,h}).
\end{array}
\]
So indeed $\ad(v)$ always defines a cocycle morphism.
\end{proof}

\begin{defi}
Let $(\alpha,\Fu): G\curvearrowright A$ be a twisted action.
We call a cocycle morphism $(\phi,\Iu): (A,\alpha,\Fu)\to (A,\alpha,\Fu)$ inner, if it is of the form $(\phi,\Iu)=\ad(v)$ in the sense of \autoref{ex:inner-morphisms} for some unitary $v\in\CU(\CM(A))$.
\end{defi}

\begin{example} \label{ex:inner-morphisms:2}
Let $(\alpha,\Fu): G\curvearrowright A$ be a gently twisted action.
Then for any unitary $v\in\CU(\eins+A)$, the cocycle conjugacy 
\[
\ad(v)=\big( \ad(v),v\alpha_\bullet(v)^*\big): (A,\alpha,\Fu)\to (A,\alpha,\Fu).
\]
as defined above becomes a proper cocycle conjugacy.
\end{example}

\begin{rem}
Let $(\alpha,\Fu): G\curvearrowright A$ and $(\beta,\Fv): G\curvearrowright B$ be two twisted actions.
If $(\phi,\Iu): (A,\alpha,\Fu)\to (\CM(B),\beta,\Fv)$ is a cocycle representation and $v\in\CU(\CM(B))$ is a unitary, then the composition with its inner cocycle automorphism is given as
\[
\begin{array}{ccl}
\ad(v)\circ(\phi,\Iu) &=& \big(\ad(v)\circ\phi,\ad(v)(\Iu_\bullet)\cdot (v\beta_\bullet(v)^*) \big)\\ 
&=& (\ad(v)\circ\phi, v\Iu_\bullet\beta_\bullet(v)^*).
\end{array}
\]
\end{rem}

\begin{prop} \label{prop:inner-normal-subgroup}
Let $(\alpha,\Fu): G\curvearrowright A$ be a twisted action.
Then the inner cocycle automorphisms of $(A,\alpha,\Fu)$ form a normal subgroup in the cocycle automorphism group of $(A,\alpha,\Fu)$.
\end{prop}
\begin{proof}
Let $(\phi,\Iu): (A,\alpha,\Fu)\to(A,\alpha,\Fu)$ be a cocycle conjugacy and $v\in\CU(\CM(A))$.
Then
\[
\begin{array}{ccl}
(\phi,\Iu)\circ\ad(v)\circ (\phi,\Iu)^{-1} &=& (\phi,\Iu)\circ\ad(v)\circ (\phi^{-1},\phi^{-1}(\Iu_\bullet)^*) \\
&=& (\phi,\Iu)\circ (\ad(v)\circ\phi^{-1}, v\phi^{-1}(\Iu_\bullet)^*\alpha_\bullet(v)^* ) \\
&=& (\ad(\phi(v)), \phi(v)\Iu_\bullet^*\phi(\alpha_\bullet(v))^*\cdot\Iu_\bullet) \\
&=& (\ad(\phi(v)), \phi(v)\alpha_\bullet(\phi(v))^*) \\
&=& \ad(\phi(v)).
\end{array}
\]
\end{proof}

%%%%%%%%

\subsection{Functoriality of crossed products}

\begin{defi} \label{def:crossed-product-hom}
Let $(\alpha,\Fu): G\curvearrowright A$ and $(\beta,\Fv): G\curvearrowright B$ be two twisted actions on \cstar-algebras.
We denote by $(\pi^A,\Iu^\alpha)$ the universal covariant representation of $(A,\alpha,\Fu)$, and by $(\pi^B,\Iu^\beta)$ the universal covariant representation of $(B,\beta,\Fv)$, which appear in the definition of their respective twisted crossed products.
For a non-degenerate cocycle representation $(\theta,\Ix): (A,\alpha,\Fu)\to (\CM(B),\beta,\Fv)$, we define the non-degenerate $*$-homomorphism $\Theta=(\theta,\Ix)\rtimes G: A\rtimes_{\alpha,\Fu} G \to \CM(B\rtimes_{\beta,\Fv} G)$ between the (full) twisted crossed products via
\[
\Theta\circ\pi^A = \pi^B\circ\phi,\quad \Theta(\Iu^\alpha_g)=\pi^B(\Ix_g)\Iu^\beta_g,\ g\in G.
\]
Notice that there is indeed a unique such $*$-homomorphism by the universal property of $A\rtimes_{\alpha,\Fu}G$, as these formulas define a covariant representation into $\CM(B\rtimes_{\beta,\Fv} G)$.
If $(\phi,\Iu)$ is a cocycle morphism into $(B,\beta,\Fv)$, then the image of $\Phi$ is contained in $B\rtimes_{\beta,\Fv} G$.
\end{defi}

\begin{prop} \label{prop:crossed-product-functorial}
Let
\[
(A,\alpha,\Fu) \stackrel{(\phi,\Iu)}{\longrightarrow} (\CM(B),\beta,\Fv),\quad (B,\beta,\Fv) \stackrel{(\psi,\Iv)}{\longrightarrow} (\CM(C),\gamma,\Fw)
\]
be two non-degenerate cocycle representations between twisted $G$-\cstar-algebras.
Then the construction in \autoref{def:crossed-product-hom} is compatible with respect to compositions, i.e., one has
\[
\Big( (\psi,\Iv)\rtimes G\Big)\circ\Big( (\phi,\Iu)\rtimes G\Big) = \big( (\psi,\Iv)\circ (\phi,\Iu) \big)\rtimes G.
\]
In particular, if we restrict to non-degenerate maps, then the (twisted) crossed product construction is functorial on the (twisted) cocycle category $\cstar_{G,t}$.
\end{prop}
\begin{proof}
Let us denote 
\[
\Phi=(\phi,\Iu)\rtimes G,\ \Psi=(\psi,\Iv)\rtimes G,\ \Theta = (\psi\circ\phi, \psi(\Iu)\Iv)\rtimes G.
\]
Then we can see directly at the level of covariant representations that
\[
\Psi\circ\Phi\circ\pi^A = \Psi\circ\pi^B\circ\phi = \pi^C\circ\psi\circ\phi = \Theta\circ\pi^A
\]
and
\[
\Psi(\Phi(\Iu^\alpha_g)) = \Psi(\pi^B(\Iu_g)\Iu^\beta_g) = \pi^C(\psi(\Iu_g))\pi^C(\Iv_g)\Iu^\gamma_g = \Theta(\Iu^\alpha_g)
\]
for all $g\in G$.
By the universal property of $A\rtimes_{\alpha,\Fu} G$ with respect to covariant representations, this forces $\Theta=\Psi\circ\Phi$ and finishes the proof.
\end{proof}

%%%%%%%%%%%%%%%%%%%%%%%%%%%%%%%%%%%%%%%%%%%%%%%%%%%%%

\section{Inductive limits}

%Throughout this section we fix a second-countable, locally compact group $G$.

\begin{nota} \label{nota:limit-notation}
Let $A_n$ be a sequence of \cstar-algebras and $\phi_n: A_n\to A_{n+1}$ a sequence of $*$-homomorphisms.
If we view this as an inductive system so that we can define the inductive limit $\dst A=\lim_{\longrightarrow} \set{A_n,\phi_n}$, we will denote
\[
\phi_{n,m}=\phi_{m-1}\circ\dots\circ\phi_{n+1}\circ\phi_n: A_n\to A_m,\quad m>n\geq 1.
\]
For notational convenience we also set $\phi_{n,n}=\id_{A_n}$.
Furthermore the universal map from the $n$-th building block into the limit will be denoted as
\[
\phi_{n,\infty}: A_n\to A,\quad n\geq 1.
\]
If every connecting map $\phi_n$ is extendible, then so is every $*$-homomorphism of the form $\phi_{n,m}$ or $\phi_{n,\infty}$.
Analogously, if 
\[
(\phi_n,\Iu^{(n)}): (A_n,\alpha^{(n)},\Fu^{(n)}) \to (A_{n+1},\alpha^{(n+1)},\Fu^{(n+1)})
\]
is an inductive system of cocycle morphisms between twisted $G$-\cstar-algebras, we introduce the symbolic notation $(\phi,\Iu)_{n,m}$ for $m>n\geq 1$, and 
\[
(\phi,\Iu)_{n,\infty}: (A_n,\alpha^{(n)},\Fu^{(n)})\to (A,\alpha,\Fu),
\] 
provided that $(A,\alpha,\Fu)$ is the inductive limit in the category $\cstar_{G,t}$.
The fact that this always exists is the subject of the next proposition.
\end{nota}

\begin{prop} \label{prop:limits}
Sequential inductive limits exist in the category $\cstar_{G,t}$.

To be more precise, let
\[
(\phi_n,\Iu^{(n)}): (A_n,\alpha^{(n)},\Fu^{(n)}) \to (A_{n+1},\alpha^{(n+1)},\Fu^{(n+1)})
\]
be an inductive system of cocycle morphisms between twisted $G$-\cstar-algebras.
Then the inductive limit in $\cstar_{G,t}$ is of the form $(A,\alpha,\Fu)$, where $\dst A=\lim_{\longrightarrow} \set{ A_n,\phi_n}$ is the ordinary inductive limit \cstar-algebra. 
The universal cocycle morphisms
\[
(\phi,\Iu)_{n,\infty}: (A_n,\alpha^{(n)},\Fu^{(n)}) \to (A,\alpha,\Fu)
\]
are of the form
\[
(\phi,\Iu)_{n,\infty}= \big( \phi_{n,\infty},\phi_{n,\infty}^+(\IU^{(n-1)*}_\bullet) \big),
\]
where 
\[
\IU^{(n)}: G\to\CU(\CM(A_{n+1}))
\]
are strictly continuous maps defined inductively via $\IU^{(0)}=\eins$ and $\IU^{(n)}_g = \phi_n^+(\IU^{(n-1)}_g)\Iu^{(n)}_g$ for $n\geq 1$ and $g\in G$.
\end{prop}
\begin{proof}
As all maps $\phi_n$ are assumed to be extendible, this will be the case for $\phi_{n,\infty}: A_n\to A$ as well.

For each $g\in G$, we have $\ad(\Iu^{(n)}_g)\circ\alpha^{(n+1)}_g\circ\phi_n = \phi_n\circ\alpha^{(n)}_g$.
By induction this translates to the condition
\[
\ad(\IU^{(n)}_g)\circ\alpha^{(n+1)}_g\circ\phi_n = \phi_n\circ\ad(\IU^{(n-1)}_g)\circ\alpha^{(n)}_g
\]
for all $g\in G$ and $n\geq 1$.
By the universal property of the limit $A$, there is a unique (point-norm continuous) map
\[
\alpha: G\to\Aut(A) \quad\text{with}\quad \alpha_g\circ\phi_{n,\infty} = \phi_{n,\infty}\circ\ad(\IU^{(n-1)}_g)\circ\alpha^{(n)}_g.
\]
For each $g,h\in G$, we have
\[
\phi_n^+(\Fu_{g,h}^{(n)}) = \Iu_g\alpha^{(n+1)}_g(\Iu_h)\Fu^{(n+1)}_{g,h}\Iu_{gh}^*.
\]
By induction, this translates to
\begin{equation} \label{eq:limits:1}
\phi_{1,n}^+(\Fu^{(1)}_{g,h}) = \IU^{(n-1)}_g \alpha^{(n)}_g (\IU^{(n-1)}_h ) \Fu^{(n)}_{g,h} \IU^{(n-1)*}_{gh}
\end{equation}
for all $n$ and $g,h\in G$.
Consider the strictly continuous map $\Fu: G\times G\to\CU(\CM(A))$ via $\Fu_{g,h}=\phi_{1,\infty}^+(\Fu^{(1)}_{g,h})$.

We claim that the pair $(\alpha,\Fu)$ defines a twisted action on $A$.
For every $n$ and $g,h\in G$, we have
\[
\begin{array}{ccl}
\alpha_g\circ\alpha_h\circ\phi_{n,\infty} &=& \phi_{n,\infty}\circ\ad(\IU^{(n-1)}_g)\circ\alpha_g^{(n)}\circ\ad(\IU^{(n-1)}_h)\circ\alpha_h^{(n)} \\
&=& \phi_{n,\infty}\circ\ad(\IU^{(n-1)}_g\alpha^{(n)}_g(\IU^{(n-1)}_h))\circ\alpha^{(n)}_g\circ\alpha^{(n)}_h \\
&=& \phi_{n,\infty}\circ\ad(\IU^{(n-1)}_g\alpha^{(n)}_g(\IU^{(n-1)}_h)\Fu^{(n)}_{g,h})\circ\alpha^{(n)}_{gh} \\
&\stackrel{\eqref{eq:limits:1}}{=}& \phi_{n,\infty}\circ\ad\big( \phi_{1,n}^+(\Fu^{(1)}_{gh})\IU^{(n-1)}_{gh} \big)\circ\alpha^{(n)}_{gh} \\
&=& \ad(\Fu_{g,h})\circ\phi_{n,\infty}\circ\ad(\IU^{(n-1)}_{gh})\circ\alpha^{(n)}_{gh} \\
&=& \ad(\Fu_{g,h})\circ\alpha_{gh}\circ\phi_{n,\infty}.
\end{array}
\]
Since $n\geq 1$ is arbitrary, this yields $\alpha_g\circ\alpha_h=\ad(\Fu_{g,h})\circ\alpha_{gh}$.
Moreover, as $(\alpha,\Fu)$ is defined so that $\phi_{1,\infty}: A_1\to A$ satisfies 
$\alpha_g\circ\phi_{1,\infty}=\phi_{1,\infty}\circ\alpha^{(1)}_g$ and $\Fu_{g,h}=\phi_{1,\infty}^+(\Fu^{(1)}_{g,h})$ for all $g,h\in G$, this immediately forces the map $\Fu$ to satisfy the 2-cocycle identity.
This verifies our claim that the triple $(A,\alpha,\Fu)$ becomes a twisted $G$-\cstar-algebra.

Let us now verify that $(\phi,\Iu)_{n,\infty}$ are well-defined cocycle morphisms.
We have arranged by construction that
\[
\alpha_g\circ\phi_{n,\infty} = \phi_{n,\infty}\circ\ad(\IU^{(n-1)}_g)\circ\alpha^{(n)}_g
\]
and
\[
\Fu_{g,h} = \phi_{n,\infty}^+\Big( \IU^{(n-1)}_g \alpha^{(n)}_g (\IU^{(n-1)}_h ) \Fu^{(n)}_{g,h} \IU^{(n-1)*}_{gh} \Big).
\]
In particular,
\[
\renewcommand\arraystretch{1.5}
\begin{array}{cl}
\multicolumn{2}{l}{ \phi_{n,\infty}^+( \IU^{(n-1)*}_g ) \alpha_g\big( \phi_{n,\infty}^+(\IU^{(n-1)*}_h) \big) \Fu_{g,h} \phi_{n,\infty}^+(\IU^{(n-1)}_{gh}) }\\
=& \phi_{n,\infty}^+\Big( \alpha_g^{(n)}(\IU^{(n-1)}_h)^* \IU_g^{(n-1)*} \Big) \Fu_{g,h} \phi_{n,\infty}^+(\IU^{(n-1)}_{gh}) \\
=& \phi_{n,\infty}^+(\Fu^{(n)}_{g,h}).
\end{array}
\]
So we see that the pair $\big( \phi_{n,\infty},\phi_{n,\infty}^+(\IU^{(n-1)*}_\bullet) \big)$ yields a well-defined cocycle morphism
\[
(\phi,\Iu)_{n,\infty} = \big( \phi_{n,\infty},\phi_{n,\infty}^+(\IU^{(n-1)*}_\bullet) \big): (A_n,\alpha^{(n)},\Fu^{(n)}) \to (A,\alpha,\Fu).
\]
Moreover it holds for all $n\geq 1$ that
\[
\begin{array}{cl}
\multicolumn{2}{l}{ (\phi,\Iu)_{n,\infty} \ = \ \big( \phi_{n,\infty},\phi_{n,\infty}^+(\IU^{(n-1)*}_\bullet) \big) } \\
=& \big( \phi_{n+1,\infty}\circ\phi_n, (\phi_{n+1,\infty}\circ\phi_n)^+(\IU^{(n-1)*}_\bullet) \big) \\
=& \big( \phi_{n+1,\infty}\circ\phi_n, \phi_{n+1,\infty}^+(\Iu^{(n)}_\bullet\cdot\IU^{(n)*}_\bullet) \big) \\
=& \big( \phi_{n+1,\infty},\phi_{n+1,\infty}^+(\IU^{(n)*}_\bullet) \big)\circ (\phi_n,\Iu^{(n)}) \ = \ (\phi,\Iu)_{n+1,\infty}\circ (\phi_n,\Iu^{(n)}).
\end{array}
\]
Finally, we claim that $(A,\alpha,\Fu)$ together with the family of morphisms $(\phi,\Iu)_{n,\infty}$ satisfies the universal property of the inductive limit in the category $\cstar_{G,t}$.
Suppose that $(B,\beta,\Fv)$ is a twisted $G$-\cstar-algebra and that for each $n$, we are given a cocycle morphism
\[
(\theta_n,\Iv^{(n)}): (A_n,\alpha^{(n)},\Fu^{(n)}) \to (B,\beta,\Fv)
\]
such that
\[
(\theta_n,\Iv^{(n)}) = (\theta_{n+1},\Iv^{(n+1)})\circ(\phi_n,\Iu^{(n)}) = (\theta_{n+1}\circ\phi_n,\theta_{n+1}^+(\Iu^{(n)})\Iv^{(n+1)} )
\]
holds for all $n$.
Then the universal property of the \cstar-algebraic limit gives us a unique $*$-homomorphism $\Theta: A\to B$ with $\Theta\circ\phi_{n,\infty}=\theta_n$ for all $n$.
From the composition formula above it follows inductively that
\[
\theta_{n+1}^+(\IU_g^{(n)})\Iv^{(n+1)}=\theta_{n+1}^+\big( \phi_n^+(\IU^{(n-1)}_g)\Iu^{(n)}_g \big)\Iv^{(n+1)}_g = \theta_n^+(\IU^{(n-1)}_g)\Iv^{(n)}_g =\Iv^{(1)}_g
\]
Setting $\IV=\Iv^{(1)}$, we claim that $(\Theta,\IV)$ yields a cocycle morphism from $(A,\alpha,\Fu)$ to $(B,\beta,\Fv)$.
Indeed, we get equivariance from
\[
\begin{array}{cl}
\multicolumn{2}{l}{ \ad(\IV_g)\circ\beta_g\circ\Theta\circ\phi_{n,\infty} } \\
=& \ad(\theta_n^+(\IU^{(n-1)}_g)\Iv^{(n)}_g)\circ\beta_g\circ\theta_n \\
=& \ad(\theta_n^+(\IU^{(n-1)}_g))\circ\theta_n\circ\alpha^{(n)}_g \\
=& \Theta\circ\phi_{n,\infty}\circ\ad(\IU^{(n-1)}_g)\circ\alpha_g^{(n)} \\
=& \Theta\circ\alpha_g\circ\phi_{n,\infty}.
\end{array}
\]
As $n$ is arbitrary, this implies $\ad(\IV_g)\circ\beta_g\circ\Theta=\Theta\circ\alpha_g$ for all $g\in G$.
The cocycle identity again follows immediately from the fact that $\phi_{1,\infty}$ is genuinely equivariant and $(\theta_1,\Iv^{(1)})$ was subject to the cocycle identity.

So $(\Theta,\IV)$ is indeed a cocycle morphism.
It follows from construction that $(\Theta,\IV)\circ (\phi,\Iu)_{n,\infty}=(\theta_n,\Iv^{(n)})$ for all $n\geq 1$, which finishes the proof.
\end{proof}

\begin{prop} \label{prop:proper-limits}
Sequential inductive limits exist in the category $\cstar_{G,t,\Fp}$.
More specifically, for any inductive system in the (twisted) proper cocycle category $\cstar_{G,t,\Fp}$, the same construction as in \autoref{prop:limits}, with all instances of ``$\bullet^+$'' replaced by ``$\bullet^\dagger$'', yields the limit in the category $\cstar_{G,t,\Fp}$.
\end{prop}
\begin{proof}
We can read this off immediately from the proof of \autoref{prop:limits}.
If $(\phi_n,\Iu^{(n)})$ are proper cocycle morphisms, then evidently $\IU^{(n)}$ is a norm-continuous map with values in $\CU(\eins+A_{n+1})$ for all $n$, and hence the universal morphisms $(\phi,\Iu)_{n,\infty}= \big( \phi_{n,\infty}, \phi_{n,\infty}^\dagger(\IU^{(n-1)*}) \big)$ are also proper cocycle morphisms.
Every other part of the proof can be copied verbatim.
\end{proof}

\begin{rem}
From the proof of \autoref{prop:limits} we can also deduce that if $\Fu^{(n)}=\eins$ for all $n\geq 1$, then $\Fu=\eins$. In fact it suffices to assume $\Fu^{(1)}=\eins$; cf.\ \autoref{rem:from-genuine-to-twisted}.
In particular, the inductive limit construction respects the subcategory $\cstar_G$ of genuine $G$-\cstar-algebras.
For the same reason, the proper cocycle category $\cstar_{G,\Fp}$ is closed under inductive limits.
\end{rem}

\subsection{The topology on Hom-sets}

\begin{defi} \label{def:cohom-topology}
Let $(\alpha,\Fu): G\curvearrowright A$ and $(\beta,\Fv): G\curvearrowright B$ be two twisted actions on \cstar-algebras.
We equip the possible sets of arrows between them with the following uniform topologies:
\begin{enumerate}[label=\textup{(\roman*)},leftmargin=*]
\item \label{def:cohom-topology:1}
On $\cohom\big( (\alpha,\Fu), (\beta,\Fv) \big)$, we consider the topology generated by the family of pseudometrics defined via
\[
d_{\CF^A,\CF^B,K}\big( (\psi,\Iv), (\phi,\Iu) \big) = \max_{a\in\CF^A} \|\psi(a)-\phi(a)\|+\max_{g\in K} \max_{b\in\CF^B} \|b(\Iv_g-\Iu_g)\|
\]
for compact sets $\CF^A\subset A$, $\CF^B\subset B$, $1_G\in K\subseteq G$, and two elements $(\psi,\Iv),(\phi,\Iu)\in \cohom\big( (\alpha,\Fu), (\beta,\Fv) \big)$.
\item \label{def:cohom-topology:2}
Suppose that both $(\alpha,\Fu)$ and $(\beta,\Fv)$ are gently twisted actions.
On $\cohomr\big( (\alpha,\Fu), (\beta,\Fv) \big)$, we consider the topology generated by the family of pseudometrics defined via
\[
d_{\CF,K}^\Fp\big( (\psi,\Iv), (\phi,\Iu) \big) = \max_{a\in\CF} \|\psi(a)-\phi(a)\|+\max_{g\in K} \|\Iv_g-\Iu_g\|
\]
for compact sets $\CF\subset A$ and $1_G\in K\subseteq G$, and $(\psi,\Iv),(\phi,\Iu)\in \cohomr\big( (\alpha,\Fu), (\beta,\Fv) \big)$.
\end{enumerate}
\end{defi}

\begin{lemma} \label{lem:cauchy-nets}
Let $(\alpha,\Fu): G\curvearrowright A$ and $(\beta,\Fv): G\curvearrowright B$ be two twisted actions on \cstar-algebras.
Let $(\phi^{\lambda},\Iu^{\lambda})\in \cohom\big( (\alpha,\Fu), (\beta,\Fv) \big)$ be a net.
\begin{enumerate}[label=\textup{(\roman*)},leftmargin=*]
\item \label{lem:cauchy-nets:1}
$(\phi^\lambda,\Iu^\lambda)$ satisfies the Cauchy criterion with respect to every pseudometric in \autoref{def:cohom-topology}\ref{def:cohom-topology:1} if and only if the nets
\[
[\lambda\mapsto \phi^\lambda(a)],\quad [\lambda \mapsto \Iu^\lambda_g b],\quad\text{and}\quad  [\lambda \mapsto b\Iu^\lambda_g ]
\]
are Cauchy for all $a\in A$, $b\in B$, $g\in G$, the latter two uniformly over compact sets $K\subseteq G$.
\item \label{lem:cauchy-nets:2}
$(\phi^\lambda,\Iu^\lambda)$ converges to $(\phi,\Iu)$ if and only if $\phi^\lambda\to\phi$ in the point-norm topology, and $\Iu^\lambda_g\to\Iu_g$ in the strict topology, uniformly over compact sets $K\subseteq G$.
In particular, the topology from \autoref{def:cohom-topology}\ref{def:cohom-topology:1} is complete.
\end{enumerate}
\end{lemma}
\begin{proof}
Let us make an intermediate observation:
If $(\phi,\Iu): (A,\alpha,\Fu)\to (B,\beta,\Fv)$ is a cocycle morphism, then it follows from the cocycle identity\eqref{eq:cocycle-identity} applied to $h=g^{-1}$ that
\begin{equation} \label{eq:cocycle-star}
\Iu_g^* = \beta_g(\Iu_{g^{-1}}) \Fv_{g,g^{-1}} \phi^+(\Fu_{g,g^{-1}})^*
\end{equation}
for all $g\in G$.

\ref{lem:cauchy-nets:1}: 
The ``if'' part is true by definition, so let us show the ``only if'' part.
Let $(\phi^{\lambda},\Iu^{\lambda})\in \cohom\big( (\alpha,\Fu), (\beta,\Fv) \big)$ be a net that satisfies the Cauchy criterion with respect to every pseudometric $d_{\CF^A,\CF^B,K}$ above.
Then by definition, we can see that $\phi^\lambda(a)$ is a Cauchy net in $B$ for every $a\in A$.

Now let $b\in B$ be a fixed element and $1_G\in K\subseteq G$ a compact set.
Evidently, the net $\lambda\mapsto b\cdot\Iu^\lambda_g$ satisfies the Cauchy criterion uniformly over $K$.
We want to show that this is also the case for the net $\lambda\mapsto \Iu^\lambda_g \cdot b$.
We consider the compact set $\CF^B\subset B$ defined as
\[
\CF^B=\set{ \beta_g^{-1}( b^*) \mid g\in K }.
\]
Let $\eps>0$.
Since $(\phi^\lambda,\Iu^\lambda)$ satisfies the Cauchy criterion with respect to the pseudometric $d_{0,\CF^B,K^{-1}}$, let us choose $\lambda_0$ such that for all $\lambda_1,\lambda_2\geq \lambda_0$, one has $d_{0,\CF^B,K^{-1}}\big( (\phi^{\lambda_1},\Iu^{\lambda_1}), (\phi^{\lambda_2},\Iu^{\lambda_2}) \big)<\eps$.
Fix such $\lambda_1,\lambda_2\geq\lambda_0$ and $g\in K$.
Then
\[
\begin{array}{ccl}
\|(\Iu_g^{\lambda_1}-\Iu_g^{\lambda_2})b\| &=& \|b^*(\Iu_g^{\lambda_1}-\Iu_g^{\lambda_2})^*\| \\
&\stackrel{\eqref{eq:cocycle-star}}{=}& \Big\| b^* \big( \beta_g(\Iu_{g^{-1}}^{\lambda_1}) - \beta_g(\Iu_{g^{-1}}^{\lambda_2}) \Big) \Fv_{g,g^{-1}}\phi^+(\Fu_{g,g^{-1}})^*   \Big\| \\
&=& \Big\| b^* \big( \beta_g(\Iu_{g^{-1}}^{\lambda_1}) - \beta_g(\Iu_{g^{-1}}^{\lambda_2}) \Big)   \Big\| \\
&=& \Big\| \beta_g^{-1}(b^*) \big( \Iu_{g^{-1}}^{\lambda_1} - \Iu_{g^{-1}}^{\lambda_2} \big)   \Big\| \\
&\leq& d_{0,\CF^B,K^{-1}}\big( (\phi^{\lambda_1},\Iu^{\lambda_1}), (\phi^{\lambda_2},\Iu^{\lambda_2}) \big) \ < \ \eps.
\end{array}
\]

\ref{lem:cauchy-nets:2}: 
The ``if'' part is true by definition of the pseudometrics in \autoref{def:cohom-topology}\ref{def:cohom-topology:1}.
For the ``only if'' part, assume that indeed $(\psi^\lambda,\Iu^\lambda)$ converges to $(\psi,\Iu)$.
Then due to \ref{lem:cauchy-nets:1}, the net $\phi^\lambda$ is Cauchy in point-norm, and the nets $[\lambda\mapsto b\Iu^\lambda_g]$ and $[\lambda\mapsto \Iu^\lambda_g b]$ are Cauchy for all $b\in B$ and $g\in G$, uniformly over compact sets $K\subseteq G$.
Since the set of $*$-homomorphisms $A\to B$ is complete in the point-norm topology and the unitary group $\CU(\CM(B))$ is complete in the strict topology, it follows that $\phi^\lambda$ converges in point-norm to a $*$-homomorphism $\psi: A\to B$, and $\Iu^\lambda$ pointwise-strictly converges to a strictly continuous map $\Iv: G\to\CU(\CM(B))$.
Since multiplication on $\CM(B)$ is (jointly) strictly continuous on bounded sets, it follows immediately that the pair $(\psi,\Iv)$ satisfies conditions \eqref{eq:equivariance} and \eqref{eq:cocycle-identity}, turning it into a cocycle morphism.
Since all the topologies under consideration are Hausdorff, it follows that this limit $(\psi,\Iv)$ must be equal to the pair $(\phi,\Iu)$.
\end{proof}

\begin{lemma} \label{lem:composition-continuity}
Let $(\alpha,\Fu): G\curvearrowright A$, $(\beta,\Fv): G\curvearrowright B$ and $(\gamma,\Fw): G\curvearrowright C$ be three twisted actions on \cstar-algebras.
Then the composition map
\[
 \cohom\big( (\alpha,\Fu), (\beta,\Fv) \big) \times  \cohom\big( (\beta,\Fv), (\gamma,\Fw) \big) \to  \cohom\big( (\alpha,\Fu), (\gamma,\Fw) \big)
\]
given by
\[
[ (\phi,\Iu), (\psi,\Iv) ] \mapsto (\psi,\Iv)\circ(\phi,\Iu)
\]
is jointly continuous.
\end{lemma}
\begin{proof}
Suppose that $(\phi^\lambda,\Iu^\lambda): (A,\alpha,\Fu)\to(B,\beta,\Fv)$ and $(\psi^\lambda,\Iv^\lambda): (B,\beta,\Fv)\to (C,\gamma,\Fw)$ are two convergent nets (indexed over the same directed set) with respective limits $(\phi,\Iu)$ and $(\psi,\Iv)$.
With respect to the ordinary composition of $*$-homomorphisms, it is evident that
\[
\psi\circ\phi = \lim_{\lambda\to\infty} \psi^\lambda\circ\phi^\lambda
\]
holds in the point-norm topology.
In order to show the claim, we have to show that also
\[
\max_{g\in K} \lim_{\lambda\to\infty} \|c \big( \psi^+(\Iu_g)\Iv_g - \psi^{\lambda+}(\Iu^\lambda_g)\Iv^\lambda_g \big) \|
\]
holds for all $c\in C$ and compact sets $K\subseteq G$.

Using the fact that $\psi^+$ is strictly continuous and that $\psi^\lambda$ converges to $\psi$ in the point-norm topology, we indeed obtain for every compact set $K\subseteq G$ and $c\in C$ that
\[
\begin{array}{cl}
\multicolumn{2}{l}{ \dst \max_{g\in K} \Big\| c \Big( \psi^+(\Iu_g) \Iv_g - \psi^{\lambda+}(\Iu^\lambda_g) \Iv^{\lambda}_g \Big) \Big\| } \\
\leq & \dst \max_{g\in K} \big\| c\big( \psi^+(\Iu_g) -  \psi^{\lambda+}(\Iu^\lambda_g) \big) \big\| + \max_{g\in K} \big\| c\psi^+(\Iu_g)\big( \Iv_g - \Iv^\lambda_g \big) \big\| \\
\stackrel{\lambda\to\infty}{\longrightarrow} & 0.
\end{array}
\]
\end{proof}

\subsection{Approximate unitary equivalence}

\begin{defi} \label{def:cocycle-hom-equivalences}
Let $(\alpha,\Fu): G\curvearrowright A$ and $(\beta,\Fv): G\curvearrowright B$ be two twisted actions on \cstar-algebras.
Let
\[
(\phi,\Iu): (A,\alpha,\Fu) \to (\CM(B),\beta,\Fv)
\]
be a cocycle representation and let $(\psi,\Iv)$ be a pair consisting of a $*$-homomor\-phism $\psi: A\to \CM(B)$ and a strictly continuous map $\Iv: G\to\CU(\CM(B))$. 
\begin{enumerate}[label=\textup{(\roman*)},leftmargin=*]
\item \label{G-ue}
The pairs $(\phi,\Iu)$ and $(\psi,\Iv)$ are unitarily equivalent, if there exists a unitary $v\in\CU(\CM(B))$ such that
\[
\psi=\ad(v)\circ\phi \quad\text{and}\quad \Iv_g=v\Iu_g\beta_g(v)^*,\ g\in G.
\]
If this is the case, then $(\psi,\Iv)$ automatically becomes another cocycle representation.
We write $(\phi,\Iu) \uee (\psi,\Iv)$.
If one can in fact choose $v\in\CU(\eins+B)$, then we say that $(\phi,\Iu)$ and $(\psi,\Iv)$ are properly unitarily equivalent.
\item \label{approx-G-ue}
We say that $(\psi,\Iv)$ is an approximate unitary conjugate of $(\phi,\Iu)$, if there exists a net of unitaries $v_\lambda\in\CU(\CM(B))$ such that 
\[
\psi(a) = \lim_{\lambda\to\infty} v_\lambda\phi(a)v_\lambda^*,
\]
and 
\[
\max_{g\in K}  \| b \big(\Iv_g - v_\lambda \Iu_g \beta_g(v_\lambda)^* \big) \|  \stackrel{\lambda\to\infty}{\longrightarrow} 0
\]
for all $a\in A$, $b\in B$ and every compact set $K\subseteq G$. 
We denote this relation by $(\psi,\Iv)\subue(\phi,\Iu)$.
It follows that $(\psi,\Iv)$ is automatically a cocycle representation.
If $(\phi,\Iu)$ is in fact a cocycle morphism into $(B,\beta,\Fv)$, then so is $(\psi,\Iv)$, and we have that the net $\ad(v_\lambda)\circ(\phi,\Iu)$ converges to $(\psi,\Iv)$ in the topology from \autoref{def:cohom-topology}\ref{def:cohom-topology:1}.
We furthermore say that $(\phi,\Iu)$ and $(\psi,\Iv)$ are approximately unitarily equivalent, written $(\phi,\Iu)\ue(\psi,\Iv)$, if $(\psi,\Iv)\subue(\phi,\Iu)$ and $(\phi,\Iu)\subue(\psi,\Iv)$.
\item \label{proper-G-ue}
The pairs $(\phi,\Iu)$ and $(\psi,\Iv)$ are called properly approximately unitarily equivalent, if there exists a net of unitaries $v_\lambda\in\CU(\eins+B)$ such that
\[
\psi(a) = \lim_{\lambda\to\infty} v_\lambda\phi(a)v_\lambda^*
\]
for all $a\in A$, and 
\[
\max_{g\in K}  \| \Iv_g-v_\lambda\Iu_g\beta_g(v_\lambda)^* \| \stackrel{\lambda\to\infty}{\longrightarrow} 0
\]
for every compact set $K\subseteq G$.
We write $(\phi,\Iu)\pue(\psi,\Iv)$.
If $(\alpha,\Fu)$ and $(\beta,\Fv)$ are gently twisted actions and $(\phi,\Iu)$ and $(\psi,\Iv)$ are proper cocycle morphisms, then this means that the net $\ad(v_\lambda)\circ(\phi,\Iu)$ converges to $(\psi,\Iv)$ in the topology from \autoref{def:cohom-topology}\ref{def:cohom-topology:2}.
\end{enumerate}
\end{defi}

\begin{rem} \label{rem:ue-not-symmetric}
The notion above in \autoref{def:cocycle-hom-equivalences}\ref{approx-G-ue} warrants a word of caution that does not apply in \ref{proper-G-ue}.
It would be tempting to guess that the relation $(\psi,\Iv)\subue(\phi,\Iu)$ ought to imply $(\phi,\Iu)\subue(\psi,\Iv)$, or even that $(\psi,\Iv)=\lim_{\lambda\to\infty} \ad(v_\lambda)\circ(\phi,\Iu)$ should imply $(\phi,\Iu)=\lim_{\lambda\to\infty} \ad(v_\lambda^*)\circ(\psi,\Iv)$, but this is false in general.
This is due to the fact that if one equips (the unit ball of) a multiplier algebra with a metric inducing the strict topology, it can happen that inner automorphisms severely distort this metric.
The following idea yields many possible counterexamples to the above.
Consider $A=\IC$, $B=\CK$, and equip these with the trivial $G$-actions.
If we assume that we are dealing with unital cocycle representations $A\to\CM(B)$, then these are nothing but strictly continuous unitary representations of $G$ on the Hilbert space $\CH=\ell^2(\IN)$.
For the purpose of this remark let us denote the trivial one just by the symbol $\eins$.
It is clear that if $\Iv: G\to\CU(\CH)$ is any unitary representation, then $\Iv\subue\eins$ implies $\Iv=\eins$.
On the other hand, for any unitary representation $\Iu: G\to\CU(\eins+\CK)$, we claim that $\eins\subue\Iu$.
Indeed, if $\set{e_n\mid n\geq 1}\subset\CH$ is an orthonormal basis, simply consider the sequence of unitaries $u_n\in\CU(\CH)$ given by
\[
u_n(e_k) = \begin{cases} e_{k+n} &,\quad k\leq n \\
e_{k-n} &,\quad n<k\leq 2n \\
e_k &,\quad k>2n.
\end{cases}
\]
Due to the fact that unitaries in the range of $\Iu$ are compact perturbations of the unit, it is immediate that $u_n \Iu_g u_n^* \to \eins$ strictly and uniformly over compact subsets of $G$.
For familiar examples of groups such as $G=\IZ$, it is of course possible to choose such a map $\Iu$ to be injective, hence we see that the relation $\subue$ is not at all symmetric in general.

In conclusion, in order to avoid logical pitfalls and to ensure that one is indeed working with an equivalence relation, one should distinguish $\ue$ and $\subue$ , unless a given set of assumptions forces the symmetry already.
For example, if $(\psi,\Iv)=\lim_{\lambda\to\infty} \ad(v_\lambda)\circ(\phi,\Iu)$ and $\phi: A\to B$ is non-degenerate, then it indeed follows that also $(\phi,\Iu)=\lim_{\lambda\to\infty} \ad(v_\lambda)^*\circ(\psi,\Iv)$.
This is a consequence of the following (heuristic) calculation for large enough $\lambda$:
\[
\begin{array}{ccl}
\phi(a) v_\lambda^* \Iv_g \beta_g(v_\lambda) &\approx& v_\lambda^* \psi(a) \Iv_g \beta_g(v_\lambda) \\
&\approx& v_\lambda^*\cdot \psi(a) v_\lambda \Iu_g \beta_g(v_\lambda)^*\cdot \beta_g(v_\lambda) \\
&\approx& \phi(a) \Iu_g.
\end{array}
\]
In summary, if we consider only non-degenerate cocycle morphisms, then the relations $\ue$ and $\subue$ do indeed coincide.
\end{rem}

\begin{rem}
If we restrict this general formalism to the special case of genuine $G$-\cstar-algebras and equivariant $*$-homomorphisms $\phi,\psi: (A,\alpha) \to (B,\beta)$, then unitary equivalence in the above sense implemented by $v\in\CU(\CM(B))$ forces the identity $v\beta_g(v)^*=\eins$.
In other words, the unitary $v$ needs to be a fixed point under the action $\beta$.

Furthermore, if a net $v_\lambda\in\CU(\CM(B))$ implements approximate unitary equivalence between two non-degenerate equivariant $*$-homomorphisms (but viewed as cocycle morphisms), then this implies that $\big( v_\lambda-\beta_g(v_\lambda) \big)$ strictly converges to zero uniformly over compact subsets of $G$.
In particular, we recover the concept of (approximate) $G$-unitary equivalence for equivariant maps, which has been considered in prior work; see for example \cite[Definition 1.15]{Szabo18ssa}.
The main thrust of the next few (sub-)sections is that the resulting formalism allows us to obtain a natural framework for the Elliott intertwining machinery on (twisted) $G$-\cstar-algebras.
\end{rem}

\begin{prop} \label{prop:ue-compositions}
The relation of (proper) approximate unitary equivalence behaves well with respect to compositions in the following sense.
\begin{enumerate}[label=\textup{(\roman*)},leftmargin=*]
\item \label{prop:ue-compositions:1}
Suppose that 
\[
(A,\alpha) \stackrel{(\phi_j,\Iu^{(j)})}{\longrightarrow} (B,\beta) \quad\text{and}\quad (B,\beta)
\stackrel{(\psi_j,\Iv^{(j)} )}{\longrightarrow} (C,\gamma)
\]
for $j=1,2$ are two pairs of cocycle morphisms between twisted $G$-\cstar-algebras.
If 
\[
(\phi_2,\Iu^{(2)}) \subue (\phi_1,\Iu^{(1)}) \quad\text{and}\quad (\psi_2,\Iv^{(2)}) \subue (\psi_1,\Iv^{(1)}),
\] 
then
\[
(\psi_2,\Iv^{(2)})\circ (\phi_2,\Iu^{(2)}) \subue (\psi_1,\Iv^{(1)}) \circ (\phi_1,\Iu^{(1)}).
\]
\item \label{prop:ue-compositions:2}
Suppose that 
\[
(A,\alpha) \stackrel{(\phi_j,\Iu^{(j)})}{\longrightarrow} (B,\beta) \quad\text{and}\quad (B,\beta)
\stackrel{(\psi_j,\Iv^{(j)} )}{\longrightarrow} (C,\gamma)
\]
for $j=1,2$ are two pairs of proper cocycle morphisms between gently twisted $G$-\cstar-algebras.
If 
\[
(\phi_1,\Iu^{(1)})\pue(\phi_2,\Iu^{(2)}) \quad\text{and}\quad (\psi_1,\Iv^{(1)})\pue(\psi_2,\Iv^{(2)}),
\] 
then
\[
(\psi_1,\Iv^{(1)}) \circ (\phi_1,\Iu^{(1)}) \pue (\psi_2,\Iv^{(2)})\circ (\phi_2,\Iu^{(2)}).
\]
\end{enumerate}
\end{prop}
\begin{proof}
\ref{prop:ue-compositions:1}:
Let $(u_\lambda)_{\lambda\in\Lambda_1}\in\CU(\CM(B))$ be a net that witnesses the relation $(\phi_2,\Iu^{(2)}) \subue (\phi_1,\Iu^{(1)})$ and $(v_\mu)_{\mu\in\Lambda_2}\in\CU(\CM(C))$ a net witnessing $(\psi_2,\Iv^{(2)}) \subue (\psi_1,\Iv^{(1)})$. 
We consider the product $\Lambda=\Lambda_1\times\Lambda_2$ as directed sets, i.e., we equip it with the order $(\lambda_1,\mu_1) \leq (\lambda_2,\mu_2)$ precisely when $\lambda_1\leq\lambda_2$ and $\mu_1\leq\mu_2$.
After replacing these nets, if necessary, we may assume that they are both indexed by $\Lambda$, i.e., via setting $u_{(\lambda,\mu)}:=u_\lambda$ and $v_{(\lambda,\mu)}=v_\mu$.

Then it follows directly from \autoref{lem:composition-continuity} that
\[
\begin{array}{ccl}
(\psi_2,\Iv^{(2)})\circ(\phi_2,\Iu^{(2)}) &=& \dst \lim_{\lambda\to\infty} \ad(v_\lambda)\circ(\psi_1,\Iv^{(1)})\circ\ad(u_\lambda)\circ(\phi_1,\Iu^{(1)}) \\
&=& \dst \lim_{\lambda\to\infty} \ad(v_\lambda\psi_1^+(u_\lambda)) \circ (\psi_1,\Iv^{(1)})\circ (\phi_1,\Iu^{(1)}).
\end{array}
\]
In other words, the net $\big( v_\lambda\psi_1^+(u_\lambda) \big)_{\lambda\in\Lambda}$ witnesses the relation $(\psi_2,\Iv^{(2)})\circ (\phi_2,\Iu^{(2)}) \subue (\psi_1,\Iv^{(1)}) \circ (\phi_1,\Iu^{(1)})$.

\ref{prop:ue-compositions:2}:
Since the involved objects are gently twisted actions and the involved morphisms are proper cocycle morphisms, it follows that we get well-defined unital cocycle morphisms
\[
(A^\dagger,\alpha) \stackrel{(\phi_j^\dagger,\Iu^{(j)})}{\longrightarrow} (B^\dagger,\beta) \quad\text{and}\quad (B^\dagger,\beta)
\stackrel{(\psi_j^\dagger,\Iv^{(j)} )}{\longrightarrow} (C^\dagger,\gamma).
\]
The assumption
\[
(\phi_1,\Iu^{(1)})\pue(\phi_2,\Iu^{(2)}) \quad\text{and}\quad (\psi_1,\Iv^{(1)})\pue(\psi_2,\Iv^{(2)})
\] 
is the same as
\[
(\phi_1^\dagger,\Iu^{(1)})\ue(\phi_2^\dagger,\Iu^{(2)}) \quad\text{and}\quad (\psi_1^\dagger,\Iv^{(1)})\ue(\psi_2^\dagger,\Iv^{(2)}),
\] 
and as above we may find nets $(u_\lambda)_{\lambda\in\Lambda_1}\in\CU(\eins+B)$ and $(v_\mu)_{\mu\in\Lambda_2}\in\CU(\eins+C)$ witnessing this fact.
It follows from the first part that the net $\big( v_\mu \psi_1^\dagger(u_\lambda) \big)_{(\lambda,\mu)\in\Lambda_1\times\Lambda_2}\in\CU(\eins+C)$ witnesses proper approximate unitary equivalence between the compositions $(\psi_1,\Iv^{(1)}) \circ (\phi_1,\Iu^{(1)})$ and $(\psi_2,\Iv^{(2)})\circ (\phi_2,\Iu^{(2)})$.
\end{proof}

%%%%%%%%%%%%%%%

\section{Two-sided Elliott intertwining}

In order to motivate what will follow below, let us already deduce a non-trivial consequence of what we have so far:

\begin{theorem} \label{thm:special-pue-intertwining}
Let $(\alpha,\Fu): G\curvearrowright A$ and $(\beta,\Fv): G\curvearrowright B$ be gently twisted actions on separable \cstar-algebras.
Suppose that
\[
(A,\alpha,\Fu) \stackrel{(\phi,\Iu)}{\longrightarrow} (B,\beta,\Fv) \quad\text{and}\quad (B,\beta,\Fv)
\stackrel{(\psi,\Iv)}{\longrightarrow} (A,\alpha,\Fu)
\]
are two proper cocycle morphisms such that $(\psi,\Iv)\circ(\phi,\Iu)\pue\id_A$ and $(\phi,\Iu)\circ(\psi,\Iv)\pue\id_B$.
Then there exist mutually inverse proper cocycle conjugacies
\[
(A,\alpha,\Fu) \stackrel{(\Phi,\IU)}{\longrightarrow} (B,\beta,\Fv) \quad\text{and}\quad (B,\beta,\Fv)
\stackrel{(\Psi,\IV)}{\longrightarrow} (A,\alpha,\Fu)
\]
with $(\phi,\Iu)\pue(\Phi,\IU)$ and $(\psi,\Iv)\pue(\Psi,\IV)$.
\end{theorem}
\begin{proof}
Let us argue that this arises as a special case of ``Theorem 2'' in \cite[pp. 35--36]{Elliott10}.
Our underlying category is the full subcategory of the twisted proper cocycle category $\cstar_{G,t,\Fp}$ (cf.\ \autoref{def:cocycle-category}) defined by the separable objects.
The definition of $\pue$ is based on a compatible notion of inner automorphisms which always form a normal subgroup in the automorphism group; see \autoref{ex:inner-morphisms:2} and \autoref{prop:inner-normal-subgroup}.

Let $(\alpha,\Fu): G\curvearrowright A$ and $(\beta,\Fv): G\curvearrowright B$ be two gently twisted actions on separable \cstar-algebras.
Choose a sequence of contractions $a_n\in A$ that is dense in the unit ball, and an increasing sequence of compact sets $K_n\subseteq G$ with $G=\bigcup_{n\in\IN} K_n$.
Then the assignment
\[
\Big( (\psi,\Iv), (\phi,\Iu) \Big) \mapsto \sum_{n=1}^\infty 2^{-n}\Big( \|\psi(a_n)-\phi(a_n)\|+\max_{g\in K_n}\|\Iv_g-\Iu_g\| \Big)
\]
defines a metric on $\cohomr( (\alpha,\Fu), (\beta,\Fv) )$ that describes the topology in \autoref{def:cohom-topology}\ref{def:cohom-topology:2}, and turns it into a Polish space.
Moreover, we see that composition with an inner cocycle automorphism (from the left) is isometric.
It is also clear that composition of proper cocycle morphisms is continuous in both variables with respect to the topology in \autoref{def:cohom-topology}\ref{def:cohom-topology:2}; the proof of this is analogous to but much easier than \autoref{lem:composition-continuity}.
In summary, Elliott's theorem applies in this context and directly gives us the desired statement.
\end{proof}

\begin{rem}
For the rest of this section we will focus on establishing other versions or generalizations of the theorem above.
Firstly, it is desirable to have an Elliott intertwining machinery in a more general context, at the very least concerning more general inductive limits like in \cite[Corollary 2.3.3]{Rordam}.
Unfortunately this is not an outcome of Elliott's high-level abstract approach in \cite{Elliott10}, so we need to work this out explicitly.
Secondly, it is desirable to also have an Elliott intertwining machinery in the non-proper picture.
Although the completely analogous theorem to the above is true in that context, as we will show below, it is important to note that this does not arise as a direct consequence from Elliott's theorem.
The latter only applies if in every space of morphisms, there is a compatible metric that turns the composition map with an inner automorphism (from the left) into an isometric map.
The discussion in \autoref{rem:ue-not-symmetric} shows that for some non-unital \cstar-algebras such as $\CK$, there does not exist such a metric with respect to the topology in \autoref{def:cohom-topology}\ref{def:cohom-topology:1}.
\end{rem}

\begin{defi}[cf.~{\cite[Definition 2.3.1]{Rordam}}]
\label{def:intertwining-diagram}
Let $(\alpha^{(n)},\Fu^{(n)}): G\curvearrowright A_n$ and $(\beta^{(n)},\Fv^{(n)}): G\curvearrowright B_n$ be sequences of twisted actions on separable \cstar-algebras.
Let
\[
(\phi_n,\Iu^{(n)}): (A_n,\alpha^{(n)},\Fu^{(n)}) \to (A_{n+1},\alpha^{(n+1)},\Fu^{(n+1)})
\]
and
\[
(\psi_n,\Iv^{(n)}): (B_n,\beta^{(n)},\Fv^{(n)}) \to (B_{n+1},\beta^{(n+1)},\Fv^{(n+1)})
\]
be sequences of cocycle morphisms, which we view as two inductive systems in the category $\cstar_{G,t}$.
Adopt the notations for the maps $\IU^{(n)}: G\to\CU(\CM(A_{n+1}))$ and $\IV^{(n)}: G\to\CU(\CM(B_{n+1}))$ as in \autoref{prop:limits}.

Consider two sequences of cocycle morphisms 
\[
(\kappa_n,\Ix^{(n)}): (B_n,\beta^{(n)},\Fv^{(n)}) \to (A_n,\alpha^{(n)},\Fu^{(n)})
\]
and 
\[
(\theta_n,\Iy^{(n)}): (A_n,\alpha^{(n)},\Fu^{(n)}) \to (B_{n+1},\beta^{(n+1)},\Fv^{(n+1)})
\]
fitting into the (not necessarily commutative) diagram\footnote{To lighten notation, we shall only carry around the underlying \cstar-algebras and $*$-homomorphisms to denote the objects and arrows in the diagram.}
\begin{equation} \label{eq:int-diagram}
\xymatrix{
\dots\ar[rr] && A_n \ar[rd]^{\theta_n} \ar[rr]^{\phi_n} && A_{n+1} \ar[r] \ar[rd] & \dots\\
\dots\ar[r] & B_n \ar[ru]^{\kappa_n} \ar[rr]^{\psi_n} && B_{n+1} \ar[ru]^{\kappa_{n+1}} \ar[rr] && \dots
}
\end{equation}
Let us call diagram \eqref{eq:int-diagram} an approximate cocycle intertwining, if the following holds: 
There exist self-adjoint compact sets $\CF_n^A\subset A_n$, $\CF_n^B\subset B_n$, increasing compact sets $1_G\in K_n\subseteq G$ and numbers $\delta_n>0$ satisfying\footnote{This list of approximate relations is supposed to specify that the composition $(\theta_n,\Iy^{(n)}) \circ (\kappa_n,\Ix^{(n)})$ is \emph{sufficiently close} to $(\psi_n,\Iv^{(n)})$ and that  $ (\kappa_{n+1},\Ix^{(n+1)}) \circ (\theta_n,\Iy^{(n)})$ is sufficiently close to $(\phi_n,\Iu^{(n)})$ with respect to the topology in \autoref{def:cohom-topology}\ref{def:cohom-topology:1}.}
\begin{enumerate}[label=\textup{(\roman*)},leftmargin=*]
\item $\phi_n(a) =_{\delta_n} \kappa_{n+1}\circ\theta_n(a)$ for all $a\in \CF_n^A$; \label{ai:1}
\item $\psi_n(b) =_{\delta_n} \theta_n\circ\kappa_n(b)$ for all $b\in \CF_n^B$; \label{ai:2}
\item $a \kappa_{n+1}^+(\Iy^{(n)}_g)\Ix^{(n+1)}_g =_{\delta_n} a \Iu^{(n)}_g$
for all $g\in K_n$, $a\in \phi_n(\CF^A_{n})$; \label{ai:3}
\item $b \theta_n^+(\Ix^{(n)}_g)\Iy^{(n)}_g  =_{\delta_n} b \Iv^{(n)}_g $ 
for all $g\in K_n$, $b\in \psi_n(\CF^B_{n})$; \label{ai:4}
\item $\bigcup_{n\in\IN} K_n = G$; \label{ai:5}
\item $\kappa_n(\CF_n^B)\subseteq \CF_n^A$ and $\theta_n(\CF_n^A)\subseteq \CF_{n+1}^B$; \label{ai:6}
\item 
$\psi_n(\CF^B_n)\cdot \set{ \theta_n^+(\IU^{(n-1)}_{g_1}) \Iy^{(n)}_{g_2} \IV^{(n)}_{g_3} }_{g_1,g_2,g_3\in K_n} \subseteq\CF^B_{n+1}$; \\
$\phi_n(\CF^A_n)\cdot \set{ \kappa_{n+1}^+(\IV^{(n)}_{g_1}) \Ix^{(n+1)}_{g_2} \IU^{(n)}_{g_3} }_{g_1,g_2,g_3\in K_n} \subseteq\CF^A_{n+1}$; \label{ai:7} 
\item $\dst\bigcup_{m\geq n} \phi_{n,m}^{-1}(\CF_m^A)\subset A_n$ and $\dst\bigcup_{m\geq n} \psi_{n,m}^{-1}(\CF_m^B)\subset B_n$ are dense for all $n$; \label{ai:8}
\item $\dst\sum_{n\in\IN} \delta_n <\infty$. \label{ai:9}
\end{enumerate}
\end{defi}

\begin{prop}[cf.~{\cite[Proposition 2.3.2]{Rordam}}]
\label{prop:general-intertwining}
Let the diagram \eqref{eq:int-diagram} describe an approximate cocycle intertwining.
Consider the inductive limits 
\[
(A,\alpha,\Fu)=\lim_{\longrightarrow} \set{ (A_n,\alpha^{(n)},\Fu^{(n)}), (\phi_n,\Iu^{(n)}) }
\] 
and 
\[
(B,\beta,\Fv)=\lim_{\longrightarrow} \set{ (B_n,\beta^{(n)},\Fv^{(n)}), (\psi_n,\Iv^{(n)}) }.
\]
Then the formulas\footnote{Here ``s-lim'' signifies that the limit is taken in the strict topology.}
\begin{equation} \label{eq:intertwining-theta}
\theta(\phi_{n,\infty}(a)) = \lim_{k\to\infty} (\psi_{k+1,\infty}\circ\theta_k\circ\phi_{n,k})(a),\quad a\in A_n
\end{equation}
\begin{equation} \label{eq:intertwining-iy}
\Iy_g = \textup{s-}\!\!\!\lim_{k\to\infty} \psi_{k+1,\infty}^+\big( \theta_k^+(\IU^{(k-1)}_g) \Iy^{(k)}_g  \IV^{(k)*}_g \big),\quad g\in G
\end{equation}
and
\begin{equation} \label{eq:intertwining-kappa}
\kappa(\psi_{n,\infty}(b)) = \lim_{k\to\infty} (\phi_{k,\infty}\circ\kappa_k\circ\psi_{n,k})(b),\quad b\in B_n,
\end{equation}
\begin{equation} \label{eq:intertwining-ix}
\Ix_g = \textup{s-}\!\!\!\lim_{k\to\infty} \phi_{k,\infty}^+\big( \kappa_k^+(\IV^{(k-1)}_g)\Ix^{(k)}_g\IU^{(k-1)*}_g \big),\quad g\in G
\end{equation}
define mutually inverse cocycle conjugacies $(\theta,\Iy): (A,\alpha,\Fu)\to (B,\beta,\Fv)$ and $(\kappa,\Ix): (B,\beta,\Fv)\to (A,\alpha,\Fu)$.
\end{prop}
\begin{proof}
The fact that the formulas \eqref{eq:intertwining-theta}+\eqref{eq:intertwining-kappa} define mutually inverse isomorphisms between the \cstar-algebras $A$ and $B$ is already proved in \cite[Proposition 2.3.2]{Rordam}. 

We shall prove that $(\theta,\Iy)$ is indeed a well-defined cocycle morphism.
First we need to show that the sequence of maps described in \eqref{eq:intertwining-iy} converges strictly and uniformly over compact sets in $G$.
We note that by the definition of composition, we have for all $k\geq 1$ that
\[
\begin{array}{ll}
\multicolumn{2}{l}{ \Big( \psi_{k+1,\infty}\circ\theta_k\circ\phi_{1,k}, \psi^+_{k+1,\infty}\big( \theta_k^+(\IU_\bullet^{(k-1)}) \Iy_\bullet^{(k)} \IV_\bullet^{(k)*} \big) \Big) } \\
=& \big( \psi_{k+1,\infty}, \psi_{k+1,\infty}^+(\IV^{(k)*}) \big) \circ (\theta_k,\Iy^{(k)}) \circ (\phi_{1,k},\IU^{(k-1)}).
\end{array}
\]
All three pairs in the second line of the equation are cocycle morphisms, and hence so is the pair on the first line.
Fix one of the compact sets $K_n\subseteq G$ and an element $b\in \CF^B_n$.
In light of \ref{ai:5}, \ref{ai:8} and \autoref{lem:cauchy-nets}, it is enough to show that maps of the form
\begin{equation} \label{eq:left-limit-goal}
g\mapsto \psi_{n,\infty}(b)\cdot \psi_{k+1,\infty}^+\big( \theta_k^+(\IU^{(k-1)}_g) \Iy^{(k)}_g  \IV^{(k)*}_g \big)
\end{equation}
satisfy the Cauchy criterion (over $k$) uniformly in $g\in K_n$.

For this we compute for all $k>n$ that
\begin{longtable}{cl}
\multicolumn{2}{l}{ $\dst \psi_{n,k+2}(b) \cdot  \theta_{k+1}^+(\IU^{(k)}_g) $ }\\
$=$ & $\dst \psi_{n,k+2}(b) \theta_{k+1}^+(\IU^{(k)}_g)  $ \\
$\stackrel{\ref{ai:2}}{=}_{\makebox[0pt]{\footnotesize\hspace{5mm}$\delta_{k+1}$}}$ & $\dst  \theta_{k+1}\Big(  \kappa_{k+1}(\psi_{n,k+1}(b))\IU^{(k)}_g  \Big)  $ \\
$\stackrel{\ref{ai:1},\ref{ai:2}}{=}_{\makebox[0pt]{\footnotesize$2\delta_k$}}$ & $\dst \theta_{k+1}\Big( \phi_{k}(\kappa_{k}(\psi_{n,k}(b))) \IU^{(k)}_g  \Big)$ \\
$=$ & $\dst \theta_{k+1}\Big( \phi_{k}\big(  \kappa_{k}(\psi_{n,k}(b)) \IU^{(k-1)}_g   \big) \Iu^{(k)}_g \Big)$ \\
$\hspace{2mm}\stackrel{\ref{ai:1},\ref{ai:2}}{=}_{\makebox[0pt]{\footnotesize\hspace{3mm}$2\delta_{k-1}$}}\hspace{2mm}$ & $\dst  \theta_{k+1}\Big( \phi_{k}\Big(  \underbrace{ \phi_{k-1}(\kappa_{k-1}(\psi_{n,k-1}(b))) \IU^{(k-1)}_g }_{\in \CF^A_k \ \ref{ai:6}, \ref{ai:7} }   \Big) \Iu^{(k)}_g \Big) $ \\
$\stackrel{\ref{ai:3}}{=}_{\makebox[0pt]{\footnotesize\hspace{1mm}$\delta_k$}}$ & $\dst  \theta_{k+1}\Big( \phi_{k}\Big( \underbrace{ \phi_{k-1}(\kappa_{k-1}(\psi_{n,k-1}(b))) \IU^{(k-1)}_g }_{\in \CF^A_k } \Big) \kappa_{k+1}^+(\Iy^{(k)}_g) \Ix_g^{(k+1)} \Big) $ \\
$\stackrel{\ref{ai:1}}{=}_{\makebox[0pt]{\footnotesize\hspace{3mm}$\delta_k$}}$ & $\dst  \theta_{k+1}\Big( \kappa_{k+1}\Big( (\theta_k\circ\phi_{k-1}\circ\kappa_{k-1}\circ\psi_{n,k-1})(b) \theta_k^+( \IU^{(k-1)}_g) \Iy^{(k)}_g  \Big) \Ix_g^{(k+1)} \Big)$ \\
$\stackrel{\ref{ai:1},\ref{ai:2}}{=}_{\makebox[0pt]{\footnotesize\hspace{3mm}$3\delta_{k-1}$}}$ & $\dst  \theta_{k+1}\Big( \kappa_{k+1}\Big( \underbrace{ \psi_{n,k+1}(b) \theta_k^+( \IU^{(k-1)}_g) \Iy^{(k)}_g }_{\in\CF_{k+1}^B \ \ref{ai:7}}  \Big) \Ix_g^{(k+1)}  \Big) $ \\
$\stackrel{\ref{ai:2}}{=}_{\makebox[0pt]{\footnotesize\hspace{6mm}$\delta_{k+1}$}}$ & $\dst  \psi_{k+1}\Big( \underbrace{ \psi_{n,k+1}(b) \theta_k^+( \IU^{(k-1)}_g) \Iy^{(k)}_g }_{\in\CF_{k+1}^B}  \Big) \theta_{k+1}^+(\Ix_g^{(k+1)}) $ \\
$\stackrel{\ref{ai:4}}{=}_{\makebox[0pt]{\footnotesize\hspace{2mm}$\delta_k$}}$ & $\dst  \psi_{k+1}\Big(  \psi_{n,k+1}(b) \theta_k^+( \IU^{(k-1)}_g) \Iy^{(k)}_g  \Big) 
\Iv^{(k+1)}_g \Iy^{(k+1)*}_g$ \\
$=$ & $\dst \psi_{n,k+2}(b)\cdot  \psi_{k+1}^+ \Big( \theta_k^+( \IU^{(k-1)}_g) \Iy^{(k)}_g   \Big) \Iv^{(k+1)}_g \Iy^{(k+1)*}_g $
\end{longtable}
For convenience, let us define $\eps_k=\max\set{\delta_{k-1},\delta_k,\delta_{k+1}}$ for $k\geq 2$, which defines another summable sequence of positive numbers as this was the case for $\delta_k$ by \ref{ai:9}.
If we apply the $*$-homomorphism $\psi_{k+2,\infty}$ to all of the above and multiply from the right with the unitary element $\psi_{k+2,\infty}^+\big( \Iy^{(k+1)}_g \IV^{(k+1)*}_g \big)$, we obtain
\[
\begin{array}{ll}
\multicolumn{2}{l}{ \dst \psi_{n,\infty}(b) \cdot \psi_{k+2,\infty}^+\big( \theta_{k+1}^+(\IU^{(k)}_g) \Iy^{(k+1)}_g \IV^{(k+1)*}_g \big) }\\
=_{12\eps_k} & \dst \psi_{n,\infty}(b)\cdot \psi_{k+2,\infty}^+\Big( \psi_{k+1}^+ \Big( \theta_k^+( \IU^{(k-1)}_g) \Iy^{(k)}_g   \Big) \Iv^{(k+1)}_g \IV^{(k+1)*}_g \Big) \\
=& \dst \dst \psi_{n,\infty}(b)\cdot \psi_{k+2,\infty}^+\Big( \psi_{k+1}^+ \Big( \theta_k^+( \IU^{(k-1)}_g) \Iy^{(k)}_g   \Big) \psi_{k+1}^+(\IV^{(k)*}_g) \Big) \\
=& \dst \dst \psi_{n,\infty}(b)\cdot \psi_{k+1,\infty}^+\Big( \theta_k^+( \IU^{(k-1)}_g) \Iy^{(k)}_g  \IV^{(k)*}_g  \Big) \\
\end{array}
\]
It now follows from the summability of the sequence $\eps_k$ that the sequence of maps in \eqref{eq:left-limit-goal} satisfies the Cauchy criterion uniformly over compact subsets of $G$.
Applying \autoref{lem:cauchy-nets} as indicated above, it follows that the pointwise strict limit in \eqref{eq:intertwining-iy} exists, and it follows that $(\theta\circ\phi_{1,\infty},\Iy)$ is a cocycle morphism from $(A_1,\alpha^{(1)},\Fu^{(1)})$ to $(B,\beta,\Fv)$.
Since by definition of the limit twisted action $(\alpha,\Fu)$ on $A$, we have $\Fu_{g,h}=\phi_{1,\infty}^+(\Fu^{(1)}_{g,h})$ for all $g,h\in G$, we immediately get the cocycle identity \eqref{eq:cocycle-identity} for the pair $(\theta,\Iy)$.

In order for $(\theta,\Iy)$ to be a cocycle morphism, we still need to verify the equivariance condition \eqref{eq:equivariance}.
This follows from the following calculation for all $n$ and $g\in G$, where we make use of the inductive limit construction from \autoref{prop:limits}:\footnote{The limits in this calculation refer to the point-norm limit of maps.}
\begin{longtable}{cl}
\multicolumn{2}{l}{ $ \ad(\Iy_g)\circ\beta_g\circ\theta\circ\phi_{n,\infty} $}\\
$\stackrel{\eqref{eq:intertwining-theta}}{=}$ & $\dst \lim_{k\to\infty} \ad(\Iy_g)\circ\beta_g\circ\psi_{k+1,\infty}\circ\theta_k\circ\phi_{n,k}$ \\
$=$ & $\dst \lim_{k\to\infty} \ad(\Iy_g)\circ\psi_{k+1,\infty}\circ\ad(\IV^{(k)}_g)\circ\beta^{(k+1)}_g\circ\theta_k\circ\phi_{n,k}$ \\
$\stackrel{\eqref{eq:intertwining-iy}}{=}$ & $\dst \lim_{k\to\infty} \psi_{k+1,\infty}\circ\ad\big( \theta_k^+(\IU^{(k-1)}_g) \Iy^{(k)}_g \big)\circ\beta^{(k+1)}_g\circ\theta_k\circ\phi_{n,k}$ \\
$=$ & $\dst \lim_{k\to\infty} \psi_{k+1,\infty}\circ\ad(\theta_k^+(\IU^{(k-1)}_g))\circ\theta_k\circ\alpha^{(k)}_g\circ\phi_{n,k}$ \\
$=$ & $\dst \lim_{k\to\infty} \psi_{k+1,\infty}\circ\theta_k\circ\ad(\IU^{(k-1)}_g)\circ\alpha^{(k)}_g\circ\phi_{n,k}$ \\
$=$ & $\dst \lim_{k\to\infty} \psi_{k+1,\infty}\circ\theta_k\circ\phi_{n,k}\circ\ad(\IU^{(n-1)}_g)\circ\alpha^{(n)}_g$ \\
$=$ & $\theta\circ\phi_{n,\infty}\circ\ad(\IU^{(n-1)}_g)\circ\alpha^{(n)}_g$ \\
$=$ & $\theta\circ\alpha_g\circ\phi_{n,\infty}$.
\end{longtable}
Since $n$ is arbitrary, we obtain $\ad(\Iy_g)\circ\beta_g\circ\theta=\theta\circ\alpha_g$, which concludes the proof that $(\theta,\Iy)$ is a cocycle conjugacy.

The analogous justification is valid for why the pair $(\kappa,\Ix)$ is a well-defined cocycle conjugacy arising from the limit of maps in \eqref{eq:intertwining-kappa} and \eqref{eq:intertwining-ix}.
We will omit the detailed argument as it follows by simply repeating the arguments above but exchanging the roles of $A_n$ and $B_n$.

Lastly, we need show that $\theta(\Ix)=\Iy^*$, which will imply that $(\kappa,\Ix)$ and $(\theta,\Iy)$ are mutually inverse.

%%%%
We compute from \eqref{eq:intertwining-theta}+\eqref{eq:intertwining-ix} and \ref{ai:1}+\ref{ai:2}+\ref{ai:3}+\ref{ai:8} that
\[
\begin{array}{ccl}
\theta(\Ix_g) &=& \dst \lim_{k\to\infty} (\psi_{k+2,\infty}\circ\theta_{k+1}\circ\phi_k)^+\big( \kappa_k^+(\IV^{(k-1)}_g)\Ix^{(k)}_g\IU^{(k-1)*}_g \big) \\

&=& \dst \lim_{k\to\infty} (\psi_{k+2,\infty}\circ\theta_{k+1}\circ\kappa_{k+1}\circ\theta_k)^+\big( \kappa_k^+(\IV^{(k-1)}_g)\Ix^{(k)}_g\IU^{(k-1)*}_g \big) \\
&=& \dst \lim_{k\to\infty} (\psi_{k+1,\infty}\circ\theta_k)^+\big( \kappa_k^+(\IV^{(k-1)}_g)\Ix^{(k)}_g\IU^{(k-1)*}_g \big) \\
&=& \dst \lim_{k\to\infty} (\psi_{k+1,\infty}\circ\theta_k)^+\big( \kappa_k^+(\IV^{(k-1)}_g) \Ix^{(k)}_g\Iu^{(k-1)*}_g \phi_k^+(\IU^{(k-2)*}_g) \big) \\
&=& \dst \lim_{k\to\infty} (\psi_{k+1,\infty}\circ\theta_k)^+\big( \kappa_k^+(\IV^{(k-1)}_g\Iy^{(k-1)*}_g) \phi_k^+(\IU^{(k-2)*}_g) \big) \\
&=& \dst \lim_{k\to\infty} (\psi_{k,\infty})^+\big( \IV^{(k-1)}_g\Iy^{(k-1)*}_g \theta_{k-1}^+(\IU^{(k-2)*}_g) \big) \\
&=& \Iy_g^*.
\end{array}
\]
Here all limits are with respect to the strict topology.
This finishes the proof.
\end{proof}

\begin{rem} \label{rem:general-intertwining}
Let us once more reflect on what the statement of \autoref{prop:general-intertwining} means.
Keeping in mind \autoref{nota:limit-notation}, we recall from \autoref{prop:limits} that the universal morphism $(\phi,\Iu)_{n,\infty}: (A_n,\alpha^{(n)},\Fu^{(n)})\to (A,\alpha,\Fu)$ is given by the pair $(\phi_{n,\infty},\phi_{n,\infty}^+(\IU^{(n-1)})^*)$ and similarly $(\psi,\Iv)_{n+1,\infty}$ is given by the pair $(\psi_{n+1,\infty},\psi_{n+1,\infty}^+(\IV^{(n)})^*)$.
Moreover we have for all $k>n$ that
\[
(\phi,\Iu)_{n,k}=(\phi_{k-1},\Iu^{(k-1)})\circ\dots\circ(\phi_n,\Iu^{(n)})=(\phi_{n,k},\phi_{n,k}^+(\IU^{(n-1)}_\bullet)^*\IU^{(k-1)}_\bullet)
\]
and
\[
(\psi,\Iv)_{n,k}=(\psi_{k-1},\Iv^{(k-1)})\circ\dots\circ(\psi_n,\Iv^{(n)})=(\psi_{n,k},\psi_{n,k}^+(\IV^{(n-1)}_\bullet)^*\IV^{(k-1)}_\bullet).
\]
Then we have 
\[
(\theta,\Ix)\circ (\phi,\Iu)_{n,\infty}=(\theta\circ\phi_{n,\infty}, (\theta\circ\phi_{n,\infty})^+(\IU^{(n-1)}_\bullet)^*\Ix_\bullet).
\]
The second argument in this pair is given as the pointwise strict limit
\[
\begin{array}{cl}
\multicolumn{2}{l}{ (\theta\circ\phi_{n,\infty})^+(\IU^{(n-1)}_g)^*\Ix_g } \\
=& \dst \lim_{k\to\infty} (\psi_{k+1,\infty}\circ\theta_k\circ\phi_{n,k})^+(\IU^{(n-1)}_g)^* \cdot \psi_{k+1,\infty}^+( \theta_k^+(\IU^{(k-1)}_g)\Iy^{(k)}_g\IV^{(k)*}_g )
\end{array}
\]
Using the topology on $\cohom( (\alpha^{(n)},\Fu^{(n)}), (\beta,\Fv) )$ as in \autoref{def:cohom-topology}, this leads to the limit formula
\[
(\theta,\Ix)\circ (\phi,\Iu)_{n,\infty} = \lim_{k\to\infty} (\psi,\Iv)_{k+1,\infty}\circ (\theta_k,\Iy^{(k)})\circ (\phi,\Iu)_{n,k}.
\]
Arguing in the other direction we can similarly express
\[
(\kappa,\Ix)\circ (\psi,\Iv)_{n,\infty} = \lim_{k\to\infty} (\phi,\Iu)_{k,\infty}\circ (\kappa_k,\Ix^{(k)})\circ (\psi,\Iv)_{n,k}.
\]
This better highlights the conceptual parallel between \autoref{prop:general-intertwining} and \cite[Proposition 2.3.2]{Rordam}.
\end{rem}

We can now use \autoref{prop:general-intertwining} and apply it to get an refined analog of \cite[Corollary 2.3.3]{Rordam}. 

\begin{theorem}
\label{thm:general-ue-intertwining}
Let $G$ be a second-countable, locally compact group.
Let $(\alpha^{(n)},\Fu^{(n)}): G\curvearrowright A_n$ and $(\beta^{(n)},\Fv^{(n)}): G\curvearrowright B_n$ be sequences of twisted actions on separable \cstar-algebras.
Let
\[
(\phi_n,\Iu^{(n)}): (A_n,\alpha^{(n)},\Fu^{(n)}) \to (A_{n+1},\alpha^{(n+1)},\Fu^{(n+1)})
\]
and
\[
(\psi_n,\Iv^{(n)}): (B_n,\beta^{(n)},\Fv^{(n)}) \to (B_{n+1},\beta^{(n+1)},\Fv^{(n+1)})
\]
be sequences of cocycle morphisms, which we view as two inductive systems in the category $\cstar_{G,t}$.
Denote by $(\phi,\Iu)_{n,\infty}: (A_n,\alpha^{(n)},\Fu^{(n)}) \to (A,\alpha,\Fu)$ and $(\psi,\Iv)_{n,\infty}: (B_n,\beta^{(n)},\Fv^{(n)}) \to (B,\beta,\Fv)$ the universal cocycle morphisms into the respective limit.
%Adopt the maps $\IU^{(n)}: G\to\CU(\CM(A_{n+1}))$ and analogously $\IV^{(n)}: G\to\CU(\CM(B_{n+1}))$ from \autoref{prop:limits}.

Consider two sequences of cocycle morphisms 
\[
(\kappa_n,\Ix^{(n)}): (B_n,\beta^{(n)},\Fv^{(n)}) \to (A_n,\alpha^{(n)},\Fu^{(n)})
\]
and 
\[
(\theta_n,\Iy^{(n)}): (A_n,\alpha^{(n)},\Fu^{(n)}) \to (B_{n+1},\beta^{(n+1)},\Fv^{(n+1)})
\]
fitting into the (not necessarily commutative) diagram
\[
\xymatrix{
\dots\ar[rr] && A_n \ar[rd]^{\theta_n} \ar[rr]^{\phi_n} && A_{n+1} \ar[r] \ar[rd] & \dots\\
\dots\ar[r] & B_n \ar[ru]^{\kappa_n} \ar[rr]^{\psi_n} && B_{n+1} \ar[ru]^{\kappa_{n+1}} \ar[rr] && \dots
}
\]
Suppose that 
\[
 (\psi_n,\Iv^{(n)}) \subue (\theta_n,\Iy^{(n)})\circ(\kappa_n,\Ix^{(n)}) \ \text{ and } \ (\phi_n,\Iu^{(n)}) \subue (\kappa_{n+1},\Ix^{(n+1)})\circ(\theta_n,\Iy^{(n)}) 
\]
holds for all $n\in\IN$. 
Then there exist mutually inverse cocycle conjugacies $(\theta,\Iy): (A,\alpha,\Fu)\to (B,\beta,\Fv)$ and $(\kappa,\Ix): (B,\beta,\Fv)\to (A,\alpha,\Fu)$.

If we may in fact assume that
\[
 (\psi_n,\Iv^{(n)}) \ue (\theta_n,\Iy^{(n)})\circ(\kappa_n,\Ix^{(n)})
\]
and 
\[
(\phi_n,\Iu^{(n)}) \ue (\kappa_{n+1},\Ix^{(n+1)})\circ(\theta_n,\Iy^{(n)}) 
\] 
holds for all $n\in\IN$, then we can arrange the cocycle conjugacies to satisfy
\[
(\theta,\Iy)\circ (\phi,\Iu)_{n,\infty} \subue (\psi,\Iv)_{n+1,\infty}\circ(\theta_n,\Iy^{(n)})  
\]
and
\[
(\kappa,\Ix)\circ (\psi,\Iv)_{n,\infty} \subue (\phi,\Iu)_{n,\infty}\circ(\kappa_n,\Ix^{(n)})
\]
for all $n\in\IN$.
\end{theorem}
\begin{proof}
We show that this arises as an application of \autoref{prop:general-intertwining}.
For this purpose we need to arrange for an approximate cocycle intertwining diagram in the sense of \autoref{def:intertwining-diagram}.
Choose any (decreasing) summable null sequence $\delta_n>0$, for example $\delta_n=2^{-n}$.
Let $1_G\in K_n\subseteq G$ be any increasing sequence of compact sets with $\bigcup_{n\in\IN} K_n=G$.
Inductively, choose finite sets of elements $\{a^{(n)}_m\}_{1\leq m\leq \ell_n}\subset A_n$ and $\{b^{(n)}_m\}_{1\leq m\leq \ell_n}\subset B_n$ such that the inclusions
\begin{equation} \label{eq:ensure-denseness}
\bigcup_{k>n} \phi_{n,k}^{-1}(\{a^{(k)}_m\}_{1\leq m\leq \ell_k})\subset A_n \quad\text{and}\quad \bigcup_{k>n}\psi_{n,k}^{-1}( \{b^{(k)}_m\}_{1\leq m\leq \ell_k}) \subset B_n
\end{equation}
are dense for all $n$.

Set $(\eta_1,\IX^{(1)})=(\kappa^{(1)},\Ix^{(1)})$, $\CF_1^B=\{b_{m}^{(1)}\}_{m\leq\ell_1}\subset B_1$ and
\[
\CF_1^A = \{ a_{m}^{(1)} \}_{m\leq\ell_1}\cup\kappa_{1}(\CF^B_1).
\]
Using the assumption that
\[
(\psi_1,\Iv^{(1)}) \subue (\theta_1,\Iy^{(1)})\circ(\kappa_1,\Ix^{(1)}) = (\theta_1,\Iy^{(1)})\circ(\eta_1,\IX^{(1)}),
\]
we may find a unitary $w_1\in\CU(\CM(B_2))$ such that if we set $(\zeta_1,\IY^{(1)})=\ad(w_1)\circ(\theta_1,\Iy^{(1)})$, we have
\[
\max_{b\in\CF_1^B} \| \psi_1(b) - \zeta_1(\eta_1(b)) \| \leq \delta_1
\]
and
\[
\max_{g\in K_1} \| b\big( \Iv^{(1)}_g - \zeta_1^+(\IX^{(1)}_g) \IY^{(1)}_g ) \| \leq \delta_1
\]
for all $b\in\psi_1(\CF^B_1)$.
We set
\[
\CF_2^B = \{b^{(2)}_m\}_{m\leq\ell_2}\cup\zeta_1(\CF^A_1)\cup\psi_1(\CF^B_1)\cdot\set{ \zeta_1^+(\IU^{(0)}_{g_1}) \IY^{(1)}_{g_2} \IV^{(1)}_{g_3} }_{g_1,g_2,g_3\in K_1}.
\]
Here we note that $\IU^{(0)}_\bullet=\eins$ by definition, so this factor is redundant in this particular step.
Using the assumption that
\[
(\phi_1,\Iu^{(1)}) \subue (\kappa_2,\Ix^{(2)})\circ(\theta_1,\Iy^{(1)}) \uee (\kappa_2,\Ix^{(2)})\circ(\zeta_1,\IY^{(1)}),
\]
we may find a unitary $v_2\in\CU(\CM(A_2))$ such that if we set $(\eta_2,\IX^{(2)})=\ad(v_2)\circ(\kappa_2,\Ix^{(2)})$, we have
\[
\max_{a\in\CF_1^A} \| \phi_1(a) - \eta_2(\zeta_1(b)) \| \leq \delta_1
\]
and
\[
\max_{g\in K_1} \| a\big( \Iu^{(1)}_g - \eta_2^+(\IY^{(1)}_g) \IX^{(2)}_g ) \| \leq \delta_1
\]
for all $a\in\phi_1(\CF^A_1)$.
We set
\[
\CF^A_2 = \{ a^{(2)}_m \}_{m\leq\ell_2}\cup\eta_2(\CF^B_2)\cup\phi_1(\CF_1^A)\cdot\set{ \eta_2^+(\IV^{(1)}_{g_1}) \IX^{(2)}_{g_2} \IU^{(1)}_{g_3} }_{g_1,g_2,g_3\in K_1}.
\]
Using the assumption that
\[
(\psi_2,\Iv^{(2)}) \subue  (\theta_2,\Iy^{(2)})\circ(\kappa_2,\Ix^{(2)}) \uee (\theta_2,\Iy^{(2)})\circ(\eta_2,\IX^{(2)}),
\]
we may find a unitary $w_2\in\CU(\CM(B_3))$ such that if we set $(\zeta_2,\IY^{(2)})=\ad(w_2)\circ(\theta_2,\Iy^{(2)})$, we have
\[
\max_{b\in\CF_2^B} \| \psi_2(b) - \zeta_2(\eta_2(b)) \| \leq \delta_2
\]
and
\[
\max_{g\in K_2} \| b\big( \Iv^{(2)}_g - \zeta_2^+(\IX^{(2)}_g) \IY^{(2)}_g ) \| \leq \delta_2
\]
for all $b\in\psi_2(\CF^B_2)$.
We set
\[
\CF_3^B = \{b^{(3)}_m\}_{m\leq\ell_3}\cup\zeta_2(\CF^A_2)\cup\psi_2(\CF^B_2)\cdot\set{ \zeta_2^+(\IU^{(1)}_{g_1}) \IY^{(2)}_{g_2} \IV^{(2)}_{g_3} }_{g_1,g_2,g_3\in K_2}.
\]
We continue like this by induction.
This allows us to find compact subsets $\CF^A_n\subset A_n$ and $\CF^B_n\subset B_n$, unitaries $v_n\in\CU(\CM(A_n))$ and $w_n\in\CU(\CM(B_{n+1}))$ such that if we define $(\eta_n,\IX^{(n)})=\ad(v_n)\circ(\kappa_n,\Ix^{(n)})$ and $(\zeta_n,\IY^{(n)})=\ad(w_n)\circ(\theta_n,\Iy^{(n)})$, then the diagram
\[
\xymatrix{
\dots\ar[rr] && A_n \ar[rd]^{\zeta_n} \ar[rr]^{\phi_n} && A_{n+1} \ar[r] \ar[rd] & \dots\\
\dots\ar[r] & B_n \ar[ru]^{\eta_n} \ar[rr]^{\psi_n} && B_{n+1} \ar[ru]^{\eta_{n+1}} \ar[rr] && \dots
}
\]
becomes an approximate cocycle intertwining in the sense of \autoref{def:intertwining-diagram}.
More precisely, every condition in \autoref{def:intertwining-diagram} except for \ref{ai:8} is automatically ensured by the inductive choice of this data with respect to $\delta_n=2^{-n}$.
As is clear from the above, we can make our choices so that
\[
\{a^{(n)}_m\}_{m\leq\ell_n}\subseteq\CF^A_n \quad\text{and}\quad \{b^{(n)}_m\}_{m\leq\ell_n}\subseteq\CF^B_n
\]
holds for every $n$.
By the initial choice of the elements $a^{(n)}_m\in A_n$ and $b^{(n)}_m\in B_n$ as to satisfy \eqref{eq:ensure-denseness}, this implies for every $n$ that the subsets
\[
\bigcup_{k\geq n} \phi^{-1}_{n,k}(\CF^A_k)\subset A_n  \quad\text{and}\quad \bigcup_{k\geq n} \psi^{-1}_{n,k}(\CF^B_k)\subset B_n
\]
are indeed dense, which verifies condition \ref{ai:8}.

By applying \autoref{prop:general-intertwining}, we obtain mutually inverse cocycle conjugacies $(\theta,\Iy): (A,\alpha,\Fu)\to (B,\beta,\Fv)$ and $(\kappa,\Ix): (B,\beta,\Fv)\to (A,\alpha,\Fu)$.
As noted in \autoref{rem:general-intertwining}, we have for every $n\geq 1$ that
\[
(\theta,\Ix)\circ (\phi,\Iu)_{n,\infty} = \lim_{k\to\infty} (\psi,\Iv)_{k+1,\infty}\circ (\zeta_k,\IY^{(k)})\circ (\phi,\Iu)_{n,k}.
\]

Now let us pass to the stronger assumption in the statement of the theorem, namely that $(\theta_n,\Iy^{(n)})\circ(\kappa_n,\Ix^{(n)}) \ue (\psi_n,\Iv^{(n)})$ and $(\kappa_{n+1},\Ix^{(n+1)})\circ(\theta_n,\Iy^{(n)}) \ue (\phi_n,\Iu^{(n)})$ holds for all $n\geq 1$.
For each $k>n$, we have that $(\zeta_k,\IY^{(k)})$ is unitarily equivalent to $(\theta_k,\Iy^{(k)})$ via $w_k\in\CU(\CM(B_{k+1}))$.
Hence for all $n\geq 1$, we obtain\footnote{This is the only place where it is important that we assume the symmetric version of approximate unitary equivalence.
The following step towards $(\theta,\Iy)\circ (\phi,\Iu)_{n,\infty} \subue (\psi,\Iv)_{n+1,\infty}\circ(\theta_n,\Iy^{(n)})$ uses at the very least the relations $(\kappa_{n+1},\Ix^{(n+1)})\circ(\theta_n,\Iy^{(n)}) \subue (\phi_n,\Iu^{(n)})$ and $(\psi_n,\Iv^{(n)}) \subue (\kappa_{n+1},\Ix^{(n+1)})\circ(\theta_n,\Iy^{(n)})$.
The proof for $(\kappa,\Ix)\circ (\psi,\Iv)_{n,\infty} \subue (\phi,\Iu)_{n,\infty}\circ(\kappa_n,\Ix^{(n)})$ need in turn the same relations in the reverse direction.}
\[
\begin{array}{cl}
\multicolumn{2}{l}{ (\psi,\Iv)_{n+1,k+1}\circ (\theta_n,\Iy^{(n)}) } \\
\ue& (\theta_k,\Iy^{(k)})\circ(\kappa_k,\Ix^{(k)})\circ(\theta_{k-1},\Iy^{(k-1)})\circ\cdots\circ(\kappa_{n+1},\Ix^{(n+1)})\circ (\theta_n,\Iy^{(n)}) \\
\ue& (\theta_k,\Iy^{(k)})\circ (\phi,\Iu)_{n,k} \\
\uee& (\zeta_k,\IY^{(k)})\circ (\phi,\Iu)_{n,k}.
\end{array}
\]
If we compose these morphisms from the left with $(\psi,\Iv)_{k+1,\infty}$ and let $k\to\infty$, it follows that
\[
\begin{array}{ccl}
(\theta,\Ix)\circ(\phi,\Iu)_{n,\infty} &=& \dst \lim_{k\to\infty} (\psi,\Iv)_{k+1,\infty}\circ (\zeta_k,\IY^{(k)})\circ (\phi,\Iu)_{n,k} \\
&\subue& (\psi,\Iv)_{n+1,\infty}\circ (\theta_n,\Iy^{(n)}).
\end{array}
\]
Since $n\geq 1$ was arbitrary, this proves the first half of the statement.
The other half, namely that
\[
(\kappa,\Ix)\circ (\psi,\Iv)_{n,\infty} \subue (\phi,\Iu)_{n,\infty}\circ(\kappa_n,\Ix^{(n)})
\]
holds for all $n\geq 1$, follows in a completely analogous fashion, by exchanging the roles of $A_n$ and $B_n$.
\end{proof}

The following further special case yields a natural setting which showcases what type of existence and uniqueness theorems need to be proved in order to classify (twisted) $G$-\cstar-algebras up to cocycle conjugacy.

\begin{cor}
\label{cor:special-ue-intertwining}
Let $G$ be a second-countable, locally compact group.
Let $(\alpha,\Fu): G\curvearrowright A$ and $(\beta,\Fv): G\curvearrowright B$ be two twisted actions on separable \cstar-algebras.
Suppose that
\[
(\phi,\Iu): (A,\alpha,\Fu) \to (B,\beta,\Fv) \quad\text{and}\quad  (\psi,\Iv): (B,\beta,\Fv)\to (A,\alpha,\Fu)
\]
are two cocycle morphisms such that
\[
\id_A \subue (\psi,\Iv)\circ(\phi,\Iu) \quad\text{and}\quad \id_B \subue (\phi,\Iu)\circ(\psi,\Iv).
\]
Then there exist mutually inverse cocycle conjugacies
\[
(\Phi,\IU): (A,\alpha,\Fu) \to (B,\beta,\Fv) \quad\text{and}\quad  (\Psi,\IV): (B,\beta,\Fv)\to (A,\alpha,\Fu)
\]
such that
\[
(\Phi,\IU) \subue (\phi,\Iu) \quad\text{and}\quad (\Psi,\IV) \subue (\psi,\Iv).
\]
\end{cor}
\begin{proof}
We consider both twisted $G$-\cstar-algebras in the assumption as trivial inductive limits
\[
(A,\alpha,\Fu)=\lim_{\longrightarrow} \set{ (A,\alpha,\Fu), (\id_A,\eins) }
\]
and
\[
(B,\beta,\Fv)=\lim_{\longrightarrow} \set{ (B,\beta,\Fv), (\id_B,\eins) }.
\]
Then the (not necessarily commuting) diagram of cocycle morphisms
\[
\xymatrix{
\dots\ar[rr] && A \ar[rd]^{\phi} \ar[rr]^{\id_A} && A \ar[r] \ar[rd]^{\phi} & \dots\\
\dots\ar[r] & B \ar[ru]^{\psi} \ar[rr]^{\id_B} && B \ar[ru]^{\psi} \ar[rr]^{\id_B} && \dots
}
\]
fits the assumptions in \autoref{thm:general-ue-intertwining}, so we can simply repeat the same argument as in its proof.
Then the resulting mutually inverse cocycle conjugacies $(\Phi,\IU)$ and $(\Psi,\IV)$ arise as certain limits
\[
(\Phi,\IU) = \lim_{k\to\infty} (\zeta_k, \IY^{(k)}),\quad (\Psi,\IV) = \lim_{k\to\infty} (\eta_k,\IX^{(k)}),
\]
where $(\zeta_k,\IY^{(k)})\uee (\phi,\Iu)$ and $(\eta_k,\IX^{(k)})\uee (\psi,\Iv)$ for all $k\geq 1$. 
Hence $(\Phi,\IU)\subue(\phi,\Iu)$ and $(\Psi,\IV)\subue(\psi,\Iv)$.\footnote{Due to the trivial form of the inductive systems, the need for the symmetric assumption in \autoref{thm:general-ue-intertwining} disappears because the relevant step in the proof becomes vacuous.}
This finishes the proof.
\end{proof}

For completeness, we shall also record the analogous result to the above in the proper picture, which one can obtain as a direct consequence and which generalizes \autoref{thm:special-pue-intertwining}.

\begin{cor}
\label{cor:general-pue-intertwining}
Let $G$ be a second-countable, locally compact group.
Let $(\alpha^{(n)},\Fu^{(n)}): G\curvearrowright A_n$ and $(\beta^{(n)},\Fv^{(n)}): G\curvearrowright B_n$ be sequences of gently twisted actions on separable \cstar-algebras.
Let
\[
(\phi_n,\Iu^{(n)}): (A_n,\alpha^{(n)},\Fu^{(n)}) \to (A_{n+1},\alpha^{(n+1)},\Fu^{(n+1)})
\]
and
\[
(\psi_n,\Iv^{(n)}): (B_n,\beta^{(n)},\Fv^{(n)}) \to (B_{n+1},\beta^{(n+1)},\Fv^{(n+1)})
\]
be sequences of proper cocycle morphisms, which we view as two inductive systems in the category $\cstar_{G,t,\Fp}$.
Denote by $(\phi,\Iu)_{n,\infty}: (A_n,\alpha^{(n)},\Fu^{(n)}) \to (A,\alpha,\Fu)$ and $(\psi,\Iv)_{n,\infty}: (B_n,\beta^{(n)},\Fv^{(n)}) \to (B,\beta,\Fv)$ the universal cocycle morphisms into the respective limit.

Consider two sequences of proper cocycle morphisms 
\[
(\kappa_n,\Ix^{(n)}): (B_n,\beta^{(n)},\Fv^{(n)}) \to (A_n,\alpha^{(n)},\Fu^{(n)})
\]
and 
\[
(\theta_n,\Iy^{(n)}): (A_n,\alpha^{(n)},\Fu^{(n)}) \to (B_{n+1},\beta^{(n+1)},\Fv^{(n+1)})
\]
fitting into the (not necessarily commutative) diagram
\[
\xymatrix{
\dots\ar[rr] && A_n \ar[rd]^{\theta_n} \ar[rr]^{\phi_n} && A_{n+1} \ar[r] \ar[rd] & \dots\\
\dots\ar[r] & B_n \ar[ru]^{\kappa_n} \ar[rr]^{\psi_n} && B_{n+1} \ar[ru]^{\kappa_{n+1}} \ar[rr] && \dots
}
\]
Suppose that 
\[
(\theta_n,\Iy^{(n)})\circ(\kappa_n,\Ix^{(n)}) \pue (\psi_n,\Iv^{(n)})
\]
and 
\[
(\kappa_{n+1},\Ix^{(n+1)})\circ(\theta_n,\Iy^{(n)}) \pue (\phi_n,\Iu^{(n)})
\] 
holds for all $n\in\IN$. 
Then there exist mutually inverse and proper cocycle conjugacies $(\theta,\Iy): (A,\alpha,\Fu)\to (B,\beta,\Fv)$ and $(\kappa,\Ix): (B,\beta,\Fv)\to (A,\alpha,\Fu)$ satisfying
\[
(\psi,\Iv)_{n+1,\infty}\circ(\theta_n,\Iy^{(n)}) \pue (\theta,\Iy)\circ (\phi,\Iu)_{n,\infty}
\]
and
\[
(\phi,\Iu)_{n,\infty}\circ(\kappa_n,\Ix^{(n)}) \pue (\kappa,\Ix)\circ (\psi,\Iv)_{n,\infty}
\]
for all $n\in\IN$.
\end{cor}
\begin{proof}
By definition of being a gently twisted action, we have canonical extensions to twisted actions $(\alpha^{(n)},\Fu^{(n)}): G\curvearrowright A_n^\dagger$ and $(\beta^{(n)},\Fv^{(n)}): G\curvearrowright B_n^\dagger$ to the proper unitizations of the involved \cstar-algebras.
Since all of the involved cocycle morphisms in the assumption are proper, they uniquely extend to unital cocycle morphisms
\[
\xymatrix{
&& (B_n^\dagger,\beta^{(n)},\Fv^{(n)}) \ar[dll]_{(\kappa_n^\dagger,\Ix^{(n)})} \ar[dd]^{(\psi_n^\dagger,\Iv^{(n)})} \\
(A_n^\dagger,\alpha^{(n)},\Fu^{(n)}) \ar[dd]_{(\phi_n^\dagger,\Iu^{(n)})} \ar[rrd]^{(\theta_n^\dagger,\Iy^{(n)})} && \\
&& (B_{n+1}^\dagger,\beta^{(n+1)},\Fv^{(n+1)}) \\
(A_{n+1}^\dagger,\alpha^{(n+1)},\Fu^{(n+1)}).
}
\]
Correspondingly, we may identify $(A^\dagger,\alpha,\Fu)$ with the limit with respect to the inductive system given by the cocycle morphisms $(\phi_n^\dagger,\Iu^{(n)})$, and analogously $(B^\dagger,\beta,\Fv)$ is the limit with respect to $(\psi_n^\dagger,\Iv^{(n)})$.
By definition of proper approximate unitary equivalence, our assumptions imply
\[
(\theta_n^\dagger,\Iy^{(n)})\circ(\kappa_n^\dagger,\Ix^{(n)}) \ue (\psi_n^\dagger,\Iv^{(n)})
\]
and 
\[
(\kappa_{n+1}^\dagger,\Ix^{(n+1)})\circ(\theta_n^\dagger,\Iy^{(n)}) \ue (\phi_n^\dagger,\Iu^{(n)})
\] 
for all $n\in\IN$.
Therefore we are in the situation to apply \autoref{thm:general-ue-intertwining} to the (possibly non-commuting) diagram of unitized cocycle morphisms above.
It follows that there exist mutually inverse cocycle conjugacies
$(\theta',\Iy): (A^\dagger,\alpha,\Fu)\to (B^\dagger,\beta,\Fv)$ and $(\kappa',\Ix): (B^\dagger,\beta,\Fv)\to (A^\dagger,\alpha,\Fu)$ satisfying
\[
(\psi^\dagger,\Iv)_{n+1,\infty} \circ(\theta_n^\dagger,\Iy^{(n)}) \ue (\theta',\Iy)\circ (\phi^\dagger,\Iu)_{n,\infty} 
\]
and
\[
(\phi^\dagger,\Iu)_{n,\infty}\circ(\kappa_n^\dagger,\Ix^{(n)}) \ue (\kappa',\Ix)\circ (\psi^\dagger,\Iv)_{n,\infty}
\]
for all $n\in\IN$. (Keep in mind \autoref{rem:ue-not-symmetric} here.)

We first note that the $*$-homomorphism $\psi_{n+1,\infty}^\dagger\circ\theta_n^\dagger$ sends $A_n$ to $B$.
Since $B$ is obviously closed under unitary equivalence inside $B^\dagger$, it follows that the $*$-homomorphism $\theta'\circ\phi_{n,\infty}$ also sends $A_n$ to $B$.
Since this is true for every $n$, we conclude that $\theta=\theta'|_A$ has image in $B$.
Hence $\theta'=\theta^\dagger$ is the unique unital extension of the isomorphism $\theta: A\to B$.
By exchanging the roles of $A_n$ and $B_n$ and repeating this argument, we likewise conclude that $\kappa'=\kappa^\dagger$ is the unique unital extension of the isomorphism $\kappa=\kappa'|_B: B\to A$.

Lastly, we note that for any pair of unitaries $v\in\CU(B^\dagger)$ and $w\in\CU(\eins+B)$, one has $vw\beta_g(v)^*\in\CU(\eins+B)$ for any $g\in G$, since the automorphism $\beta_g$ fixes the scalar part of $v$.
By assumption, the map $G\to\CU(B^\dagger)$ belonging to the composition $(\psi^\dagger,\Iv)_{n+1,\infty}\circ(\theta_n^\dagger,\Iy^{(n)})$ takes values in $\CU(\eins+B)$, and so the same will be true for the composition $(\theta^\dagger,\Iy)\circ(\phi^\dagger,\Iu)_{n,\infty}$.
Since the map $G\to\CU(B^\dagger)$ belonging to this composition is given by $g\mapsto (\theta\circ\phi_{n,\infty})^\dagger(\IU^{(n-1)*}_g)\Iy_g$ and the first factor is in $\CU(\eins+B)$, we hence conclude that $\Iy_g\in\CU(\eins+B)$ for all $g\in G$.
By exchanging the roles of $A_n$ and $B_n$ once again and repeating this argument, we likelywise conclude that $\Ix$ takes values in $\CU(\eins+A)$.

In summary, we have that $(\theta,\Iy): (A,\alpha,\Fu)\to (B,\beta,\Fv)$ and $(\kappa,\Ix): (B,\beta,\Fv)\to (A,\alpha,\Fu)$ are mutually inverse and proper cocycle conjugacies.
The statement that
\[
(\psi,\Iv)_{n+1,\infty} \circ(\theta_n,\Iy^{(n)}) \pue (\theta,\Iy)\circ (\phi,\Iu)_{n,\infty}
\]
and
\[
(\phi,\Iu)_{n,\infty}\circ(\kappa_n,\Ix^{(n)}) \pue (\kappa,\Ix)\circ (\psi,\Iv)_{n,\infty}
\]
holds for all $n\in\IN$, follows from the above.
A priori, the approximate unitary equivalence between these compositions is implemented by sequences of unitaries in $\CU(B^\dagger)$ and $\CU(A^\dagger)$, respectively.
By rescaling with their scalar parts, we may assume that the unitaries are in $\CU(\eins+B)$ and $\CU(\eins+A)$, respectively, so we indeed get proper approximate unitary equivalence as desired.
\end{proof}

%%%

\section{One-sided intertwining}

\subsection{Asymptotic unitary equivalence}

\begin{defi} \label{defi:cocycle-hom-equivalences-2}
Let $(\alpha,\Fu): G\curvearrowright A$ and $(\beta,\Fv): G\curvearrowright B$ be two twisted actions on \cstar-algebras.
Suppose that $A$ is separable and $B$ is $\sigma$-unital.
Let
\[
(\phi,\Iu): (A,\alpha,\Fu) \to (\CM(B),\beta,\Fv)
\]
be a cocycle representation and let $(\psi,\Iv)$ be a pair consisting of a $*$-homomor\-phism $\psi: A\to \CM(B)$ and a strictly continuous map $\Iv: G\to\CU(\CM(B))$. 
\begin{enumerate}[label=\textup{(\roman*)},leftmargin=*]
\item \label{as-G-ue}
We say that $(\psi,\Iv)$ is an asymptotic unitary conjugate of $(\phi,\Iu)$, if there exists a strictly continuous map $v: [0,\infty)\to\CU(\CM(B))$ such that 
\[
\psi(a) = \lim_{t\to\infty} v_t\phi(a)v_t^*,
\]
and 
\[
\max_{g\in K} \| b \big(\Iv_g-v_t\Iu_g\beta_g(v_t)^*\big) \| \stackrel{t\to\infty}{\longrightarrow} 0
\]
for all $a\in A$, $b\in B$ and every compact set $K\subseteq G$.
As before, it follows that $(\psi,\Iv)$ is automatically a cocycle representation.
If $(\phi,\Iu)$ is in fact a cocycle morphism into $(B,\beta,\Fv)$, then so is $(\psi,\Iv)$, and we have that the net $\ad(v_t)\circ(\phi,\Iu)$ converges to $(\psi,\Iv)$ in the topology from \autoref{def:cohom-topology}\ref{def:cohom-topology:1}.
Similar to before, we say that $(\psi,\Iv)$ and $(\phi,\Iu)$ are asymptotically unitarily equivalent, written $(\phi,\Iu)\asue(\psi,\Iv)$, if they are asymptotic unitary conjugates of each other.
\item \label{proper-as-G-ue}
The pairs $(\phi,\Iu)$ and $(\psi,\Iv)$ are properly asymptotically unitarily equivalent, if there exists a norm-continuous map $v: [0,\infty)\to\CU(\eins+B)$ such that
\[
\psi(a) = \lim_{t\to\infty} v_t\phi(a)v_t^*
\]
for all $a\in A$, and 
\[
\max_{g\in K}  \| \Iv_g-v_t\Iu_g\beta_g(v_t)^* \| \stackrel{t\to\infty}{\longrightarrow} 0
\]
for every compact set $K\subseteq G$.
We write $(\phi,\Iu)\pasue(\psi,\Iv)$.
If $(\alpha,\Fu)$ and $(\beta,\Fv)$ are gently twisted actions and $(\phi,\Iu)$ and $(\psi,\Iv)$ are proper cocycle morphisms, then this means that the net $\ad(v_\lambda)\circ(\phi,\Iu)$ converges to $(\psi,\Iv)$ in the topology from \autoref{def:cohom-topology}\ref{def:cohom-topology:2}.
\end{enumerate}
\end{defi}

In this subsection we will characterize when a given (non-degenerate) cocycle morphism is asymptotically unitarily equivalent to a cocycle conjugacy, which will culminate in a continuous version of \autoref{cor:special-ue-intertwining}.

\begin{prop} \label{prop:asue-composition}
Let
\[
(A,\alpha,\Fu) \stackrel{(\phi,\Iu)}{\longrightarrow} (B,\beta,\Fv),\quad (B,\beta,\Fv) \stackrel{(\psi,\Iv)}{\longrightarrow} (C,\gamma,\Fw)
\]
be two cocycle morphisms between separable twisted $G$-\cstar-algebras.
If both $(\phi,\Iu)$ and $(\psi,\Iv)$ are asymptotically unitarily equivalent to cocycle conjugacies, then the composition $(\psi,\Iv)\circ(\phi,\Iu)$ is asymptotically unitarily equivalent to a cocycle conjugacy.
\end{prop}
\begin{proof}
This is a trivial consequence from the proof of \autoref{prop:ue-compositions}.
If there exist cocycle conjugacies
\[
(A,\alpha,\Fu) \stackrel{(\Phi,\IU)}{\longrightarrow} (B,\beta,\Fv),\quad (B,\beta,\Fv) \stackrel{(\Psi,\IV)}{\longrightarrow} (C,\gamma,\Fw)
\]
with
\[
(\Phi,\IU) \asue (\phi,\Iu) \quad\text{and}\quad (\Psi,\IV) \asue (\psi,\Iv),
\]
then it easily follows from \autoref{lem:composition-continuity} that also
\[
(\Psi,\IV)\circ(\Phi,\IU) \asue (\psi,\Iv)\circ(\phi,\Iu).
\]
\end{proof}

The following is a variant of \cite[Proposition 2.1]{Szabo17ssa3}.
Although half of the proof is very similar, we shall give the full argument for completeness.

\begin{prop} \label{prop:asue-to-an-iso}
Let $G$ be a second-countable, locally compact group.
Suppose that $(\alpha,\Fu): G\curvearrowright A$ and $(\beta,\Fv): G\curvearrowright B$ are twisted actions on separable \cstar-algebras.
Let $(\phi,\Iu): (A,\alpha,\Fu)\to (B,\beta,\Fv)$ be a cocycle morphism such that $\phi$ is injective and non-degenerate.
Then $(\phi,\Iu)$ is asymptotically unitarily equivalent to a cocycle conjugacy if and only if the following is true:

For all finite sets $\CF^A\subset A$, $\CF^B\subset B$, compact sets $K\subseteq G$, and $\eps>0$, there exists a (strictly continuous) unitary path $z: [0,1]\to\CU(\CM(B))$ with $z_0=\eins$ and satisfying
	\begin{itemize}
	\item $\dst \max_{0\leq t\leq 1} \|[z_t,\phi(a)]\|\leq\eps$ for all $a\in\CF^A$;
	\item $\dist\big( z_1^*bz_1, \phi(A) \big) \leq \eps$ for all $b\in \CF^B$;
	\item $\dst \max_{g\in K} \max_{0\leq t\leq 1} \|b(\Iu_g-z_t\Iu_g\beta_g(z_t)^*)\| \leq \eps$ for all $b\in\CF^B$.
	\end{itemize}
\end{prop}

\begin{reme}
Before we resume with the proof, let us contemplate the above non-degeneracy assumption on $\phi$.
I do not know if it is a necessary assumption, but it appears to be somewhat relevant for the ``only if'' part, which is important for the application that follows afterwards.
The difficulty essentially lies in arranging the statement in the last bullet point for all $t\in [0,1]$ (and not just for $t=1$), which under the non-degeneracy assumption is possible below because it suffices to consider only $b\in\phi(A)$.
In particular, when there are no group actions involved, the assumption can be relaxed to saying that $\phi$ is extendible.
\end{reme}

\begin{proof}
Let us first show the ``only if'' part.
Suppose that $(\phi,\Iu)$ is asymptotically unitarily equivalent to a cocycle conjugacy $(\Phi,\IU): (A,\alpha,\Fu)\to (B,\beta,\Fv)$.
Since $\phi$ is non-degenerate, this means (cf.\ \autoref{rem:ue-not-symmetric}) that there exists a strictly continuous path of unitaries $w: [0,\infty)\to\CU(\CM(B))$ such that
\[
\Phi(a)=\lim_{t\to\infty} w_t\phi(a)w_t^*, \quad a\in A,
\] 
and
\[
\lim_{t\to\infty} \max_{g\in K}  \|b(\IU_g-w_t\Iu_g\beta_g(w_t)^*)\|  = 0
\]
for all compact sets $K\subseteq G$ and $b\in B$.

Now let $\CF^A, \CF^B, K, \eps$ be given as in the statement.
As pointed out in the remark above, we assume $\phi$ to be non-degenerate, and hence we may substitute the element $b\in\CF^B$ in the last bullet point for $\phi(a)$ for $a\in\CF^A$.
By the above, we may then choose a number $n_1\geq 1$ such that
\begin{equation} \label{eq:asue-iso:n1-1}
\sup_{t\geq n_1} \|\Phi(a)-w_{t}\phi(a)w_{t}^*\| \leq \eps/4,\quad a\in\CF^A,
\end{equation}
and
\begin{equation} \label{eq:asue-iso:n1-2}
\sup_{t\geq n_1} \max_{g\in K} \max_{a\in\CF^A}  \|\Phi(a)(\IU_g-w_{t}\Iu_g\beta_g(w_{t})^*)\| \leq \eps/4.
\end{equation}
Then, choose a bigger number $n_2>n_1$ such that
\begin{equation} \label{eq:asue-iso:n2-1}
\sup_{t\geq n_2} \|\Phi(a)-w_{t}\phi(a)w_{t}^*\| \leq \eps,\quad a\in \Phi^{-1}(w_{n_1}\CF^Bw_{n_1}^*).
\end{equation}
We set $z_t=w_{n_1}^*w_{n_1+t(n_2-n_1)}$ for $t\in [0,1]$.
We claim that this unitary path satisfies the desired properties.
We first observe for all $g\in K$, $a\in\CF^A$ and $t\in [0,1]$ that
\[
\begin{array}{cl}
\multicolumn{2}{l}{ \phi(a)z_t\Iu_g\beta_g(z_t)^* } \\
=& \phi(a) w_{n_1}^*w_{n_1+t(n_2-n_1)} \Iu_g \beta_g(w_{n_1+t(n_2-n_1)}^*w_{n_1}) \\
\stackrel{\eqref{eq:asue-iso:n1-1}}{=}_{\makebox[0pt]{\footnotesize$\eps/4$}}& w_{n_1}^* \Phi(a) w_{n_1+t(n_2-n_1)} \Iu_g \beta_g(w_{n_1+t(n_2-n_1)}^*w_{n_1}) \\
\stackrel{\eqref{eq:asue-iso:n1-2}}{=}_{\makebox[0pt]{\footnotesize$\eps/4$}} & w_{n_1}^* \Phi(a) \IU_g \beta_g(w_{n_1}) \\
\stackrel{\eqref{eq:asue-iso:n1-2}}{=}_{\makebox[0pt]{\footnotesize$\eps/4$}} & w_{n_1}^* \Phi(a) w_{n_1} \Iu_g \\
\stackrel{\eqref{eq:asue-iso:n1-1}}{=}_{\makebox[0pt]{\footnotesize$\eps/4$}} & \phi(a)\Iu_g.
\end{array}
\]
In summary, we obtain
\[
\max_{g\in K} \max_{a\in\CF^A} \max_{0\leq t\leq 1} \|\phi(a)(\Iu_g-z_t\Iu_g\beta_g(z_t^*))\| \leq \eps.
\]
Applying \eqref{eq:asue-iso:n1-1} twice, we clearly have
\[
\max_{a\in\CF^A} \max_{0\leq t\leq 1} \|[z_t,\phi(a)]\| \leq \eps/2.
\]
Furthermore, we have for all $b\in\CF^B$ that
\[
\begin{array}{ccl}
z_1^*bz_1 &=& w_{n_2}^*w_{n_1} b w_{n_1}^* w_{n_2} \\
&=& w_{n_2}^*\Phi(\Phi^{-1}(w_{n_1}bw_{n_1}^*))w_{n_2} \\
&\stackrel{\eqref{eq:asue-iso:n2-1}}{=}_{\makebox[0pt]{\footnotesize\hspace{-2mm}$\eps$}}& \phi( \Phi^{-1}(w_{n_1} b w_{n_1}^*) ) \ \in \ \phi(A).
\end{array}
\]
This shows that $z$ is indeed a unitary path with the desired properties.

Now let us show the ``if'' part.
Let $\set{a_n}_{n\in\IN}\subset A$ and $\set{b_n}_{n\in\IN}\subset B$ be dense sequences. 
Write $G=\bigcup_{n\in\IN} K_n$ for an increasing union of compact subsets $1_G\in K_n$. 
We are going to perturb $(\phi,\Iu)$ with paths of unitaries step by step:

In the first step, choose some $a_{1,1}\in A$ and $z^{(1)}: [0,1]\to\CU(\CM(B))$ with $z^{(1)}_0=\eins$ such that for all $0\leq t\leq 1$, we have
\begin{itemize}
\item $z^{(1)*}_1b_1z^{(1)}_1 =_{1/2} \phi(a_{1,1})$;
\item $\|[z^{(1)}_t,\phi(a_1)]\|\leq 1/2$;
\item $b_1 \Iu_g =_{1/2} b_1 z^{(1)}_t\Iu_g\beta_g(z^{(1)}_t)^*$ for all $g\in K_1$.
\end{itemize}
%%%
In the second step, choose $a_{2,1},a_{2,2}\in A$ and $z^{(2)}: [0,1]\to\CU(\CM(B))$ with $z^{(2)}_0=\eins$ such that for every $0\leq t\leq 1$ we have
\begin{itemize}
\item $z^{(2)*}_1 (z^{(1)*}_1 b_j z^{(1)}_1) z^{(2)}_1 =_{1/4} \phi(a_{2,j})$ for $j=1,2$;
\item $\|[z^{(2)}_t,\phi(a_j)]\|\leq 1/4$ for $j=1,2$;
\item $\|[z^{(2)}_t,\phi(a_{1,1})]\|\leq 1/4$;
\item $(b_jz^{(1)}_1)\Iu_g =_{1/4} (b_jz^{(1)}_1)z^{(2)}_t\Iu_g\beta_{g}(z^{(2)}_t)^*$ for all $g\in K_2$ and $j=1,2$.
\end{itemize}
%%%
Now assume that for some $n\in\IN$, we have found $z^{(1)},\dots,z^{(n)}: [0,1]\to\CU(\CM(B))$ and $\set{a_{m,j}}_{j\leq m\leq n}\subset A$ satisfying for every $0\leq t\leq 1$ that
\begin{equation} \label{eq1:c1}
z^{(n)*}_1(z^{(n-1)*}_1\cdots z^{(1)*}_1b_jz^{(1)}_1\cdots z^{(n-1)}_1)z^{(n)}_1 =_{2^{-n}} \phi(a_{n,j})~ \text{for}~ j\leq n;
\end{equation}
\begin{equation} \label{eq1:c2}
\|[z^{(n)}_t,\phi(a_j)]\|\leq 2^{-n}~ \text{for}~ j\leq n;
\end{equation}
\begin{equation} \label{eq1:c3}
\|[z^{(n)}_t,\phi(a_{m,j})]\|\leq 2^{-n}~ \text{for}~ m<n~ \text{and}~ j\leq m ;
\end{equation}
\begin{equation} \label{eq1:c4}
(b_jz^{(1)}_1\cdots z^{(n-1)}_1)\Iu_g =_{2^{-n}} (b_jz^{(1)}_1\cdots z^{(n-1)}_1)z^{(n)}_t\Iu_g\beta_{g}(z^{(n)}_t)^*
\end{equation}
for all $g\in K_n$ and $j\leq n$.

Then we can again apply our assumptions to find $z^{(n+1)}: [0,1]\to\CU(\CM(B))$ with $z^{(n+1)}_0=\eins$ and $\set{a_{n+1,j}}_{j\leq n+1}\subset A$ so that for every $0\leq t\leq 1$ we have
\begin{enumerate}[label=$\bullet$,leftmargin=*]
\item $z^{(n+1)*}_1(z^{(n)*}_1\cdots z^{(1)*}_1 b_j z^{(1)}_1\cdots z^{(n)}_1)z^{(n+1)}_1 =_{2^{-(n+1)}} \phi(a_{n+1,j})$ for $j\leq n+1$;
\item $\|[z^{(n+1)}_t,\phi(a_j)]\|\leq 2^{-(n+1)}$ for $j\leq n+1$;
\item $\|[z^{(n+1)}_t,\phi(a_{m,j})]\|\leq 2^{-(n+1)}$ for $m<n+1$ and $j\leq m$;
\item $(b_j z^{(1)}_1\cdots z^{(n)}_1)\Iu_g =_{2^{-(n+1)}} (b_j z^{(1)}_1\cdots z^{(n)}_1)z^{(n+1)}_t\Iu_g\beta_{g}(z^{(n+1)}_t)^*$ for all $g\in K_{n+1}$ and $j\leq n+1$.
\end{enumerate}
Carry on inductively. We define a norm-continuous path of unitaries $v: [0,\infty)\to\CU(\CM(B))$ via $v_t = z^{(1)}_1\cdots z^{(n)}_1 z^{(n+1)}_{t-n}$ for $n\geq 0$ with $n\leq t\leq n+1$.
This map is indeed continuous since every path $z^{(n)}$ starts at the unit. 
We obtain a continuous path of cocycle morphisms $(\psi_t,\Iu^{(t)}): (A,\alpha)\to (B,\beta)$ for $t\geq 0$ via $\psi_t = \ad(v_t)\circ\phi$ and $\Iu^{(t)}_g=v_t\Iu_g\beta_g(v_t)^*$. 
We claim that this converges to a cocycle conjugacy when we let $t\to\infty$.

By condition \eqref{eq1:c2}, we can immediately observe that the net $(\psi_t(a_j))_{t\geq 0}$ is Cauchy for all $j\in\IN$. 
Since the set $\set{a_j}_{j\in\IN}\subset A$ was assumed to be dense, this implies that the net $(\psi_t)_{t\geq 0}$ converges to some $*$-homomorphism $\psi: A\to B$ in the point-norm topology. 
Since $\phi$ was assumed to be injective, so is $\psi$.
In order to show that $\psi$ is an isomorphism, it suffices to show that it has dense image.
Let $j\geq 1$ be given.
Then for $n\geq j$, condition \eqref{eq1:c1} tells us that $\|\phi_n(a_{n,j})-b_j\|\leq 2^{-n}$.
At the same time, it follows from condition \eqref{eq1:c3} that
\[
\begin{array}{ccl}
\|\psi(a_{n,j})-\phi_n(a_{n,j})\| &\leq& \dst \sum_{k=n}^\infty \| \phi_{k+1}(a_{n,j})-\phi_k(a_{n,j}) \| \\
&=& \dst \sum_{k=n}^\infty \|z^{(k+1)}_1 \phi(a_{n,j}) z^{(k+1)*}_1 - \phi(a_{n,j}) \| \\
&\leq& \dst \sum_{k=n}^\infty 2^{-k} \ = \ 2^{1-n}.
\end{array}
\]
This culminates in the estimate $\|b_j-\psi(a_{n,j})\|\leq 2^{2-n}$ for all $j\geq 1$ and $n\geq j$.
Since the set $\set{b_j}_{j\in\IN}\subset B$ was assumed to be dense, it follows that $\psi$ has dense image and is therefore an isomorphism.

By condition \eqref{eq1:c4}, we have that the assignment 
\[
t\mapsto b_j\cdot v_t \Iu_g \beta_g(v_t^*) = b \Iu^{(t)}_g
\]
yields a Cauchy net for every $j\in\IN$ and $g\in G$, with uniformity on compact subsets of $G$.
Since $\set{b_j}_{j\in\IN}\subset B$ was assumed to be dense, it follows that this is true for any element $b\in B$ in place of $b_j$.
By \autoref{lem:cauchy-nets}, it follows that the strict limit 
\[
\Iv_g = \lim_{t\to\infty} \Iu_g^{(t)} \in \CU(\CM(B)) ,\quad g\in G
\]
exists, and the convergence is uniform over compact sets in $G$.
Hence $(\psi,\Iv)=\lim_{t\to\infty} (\phi_t,\Iu^{(t)})=\lim_{t\to\infty} \ad(v_t)\circ(\phi,\Iu)$ is the desired cocycle conjugacy.
Since $\phi$ is non-degenerate, it follows from \autoref{rem:ue-not-symmetric} that $(\psi,\Iv) \asue (\phi,\Iu)$.
\end{proof}

\begin{lemma} \label{lem:asue-limits}
Let $G$ be a second-countable, locally compact group.
Let $(\alpha^{(n)},\Fu^{(n)}): G\curvearrowright A_n$ be a sequence of twisted actions on separable \cstar-algebras.
Let
\[
(\phi_n,\Iu^{(n)}): (A_n,\alpha^{(n)},\Fu^{(n)}) \to (A_{n+1},\alpha^{(n+1)},\Fu^{(n+1)})
\]
be a sequence of injective and non-degenerate cocycle morphisms with inductive limit $\dst (A,\alpha,\Fu)=\lim_{\longrightarrow} (A_n,\alpha^{(n)},\Fu^{(n)})$.
Suppose that for every $n\geq 1$, $(\phi_n,\Iu^{(n)})$ is asymptotically unitarily equivalent to a cocycle conjugacy.
Then it follows that
\[
\phi_{1,\infty}: (A_1,\alpha^{(1)},\Fu^{(1)}) \to (A,\alpha,\Fu)
\]
is asymptotically unitarily equivalent to a cocycle conjugacy.\footnote{In this context remember from the construction of the limit that the $*$-homomorphism $\phi_{1,\infty}$ is genuinely equivariant here, so we identify it with the cocycle morphism $(\phi_{1,\infty},\eins)$.}
\end{lemma}
\begin{proof}
We adopt the notation from \autoref{prop:limits}.

Let $\CF_1\subset A_1$ and $\CF_\infty\subset A$ be finite sets.
Let $K\subseteq G$ be compact and $\eps>0$.
In order to obtain the claim, we are going to verify the condition in \autoref{prop:asue-to-an-iso} for $\phi_{1,\infty}$.
After making a small perturbation to $\CF_\infty$, we may assume without loss of generality that there is some large enough $n\geq 1$ and finite set $\CF_n\subset A_n$ such that $\CF_\infty=\phi_{n,\infty}(\CF_n)$.
By \autoref{prop:asue-composition}, the cocycle morphism
\[
(\phi,\Iu)_{1,n} = (\phi_{1,n},\IU^{(n-1)}): (A_1,\alpha^{(1)},\Fu^{(1)}) \to (A_n,\alpha^{(n)},\Fu^{(n)})
\]
is asymptotically unitarily equivalent to a cocycle conjugacy.
By \autoref{prop:asue-to-an-iso}, we hence find a unitary path $y: [0,1]\to\CU(\CM(A_n))$ satisfying
\begin{enumerate}[label=$\bullet$,leftmargin=*]
\item $\dist( y_1^* b y_1, \phi_{1,n}(A_1) )\leq\eps$ for all $b\in\CF_n$;
\item $\dst \max_{0\leq t\leq 1} \|[y_t,\phi_{1,n}(a)]\|\leq\eps$ for all $a\in\CF_1$;
\item $\dst \max_{g\in K} \max_{0\leq t\leq 1} \|b(\IU^{(n-1)}_g-y_t\IU^{(n-1)}_g\alpha^{(n)}_g(y_t)^*)\|\leq\eps$ for all $b\in\CF_n$.
\end{enumerate}
Set $z=\phi_{n,\infty}^+\circ y : [0,1]\to\CU(\CM(A))$.
Then we evidently have the two conditions
\[
\dist( z_1^* b z_1, \phi_{1,\infty}(A_1) )\leq\eps \quad\text{for all } b\in\CF_\infty
\]
and
\[
\max_{0\leq t\leq 1} \|[z_t,\phi_{1,\infty}(a)]\|\leq\eps \quad\text{for all }a\in\CF_1.
\]
Furthermore, we have for all $g\in K$, $t\in [0,1]$ and $b\in\CF_n$ that
\[
\begin{array}{ccl}
\psi_{n,\infty}(b)z_t\alpha_g(z_t)^* &=& \psi_{n,\infty}(by_t)\cdot(\alpha_g\circ\psi_{n,\infty})(y_t)^* \\
&=& \psi_{n,\infty}\big( by_t\IU^{(n-1)}_g\alpha^{(n)}_g(y_t)^*\IU^{(n-1)*}_g \big) \\
&=_{\makebox[0pt]{\footnotesize\hspace{1mm}$\eps$}} & \psi_{n,\infty}(b).
\end{array}
\]
In summary, we have verified the condition in \autoref{prop:asue-to-an-iso} for $(\phi_{1,\infty},\eins)$, which implies that it is indeed asymptotically unitarily equivalent to a cocycle conjugacy.
\end{proof}

The following should be seen as a continuous version of \autoref{cor:special-ue-intertwining}.
Although some version of this statement appears to be folklore in the case $G=\set{1}$ due to Kirchberg's famous unpublished preprint \cite{KirchbergC}, there does to my knowledge not exist a published proof even of that statement anywhere.
The proof presented below is in turn inspired by Kirchberg's original approach.
As mentioned further above, when there are no group actions involved, one can relax the non-degeneracy assumption on the involved $*$-homomorphisms to extendability.

\begin{theorem}
\label{thm:cont-intertwining}
Let $G$ be a second-countable, locally compact group.
Let $(\alpha,\Fu): G\curvearrowright A$ and $(\beta,\Fv): G\curvearrowright B$ be two twisted actions on separable \cstar-algebras.
Suppose that
\[
(\phi,\Iu): (A,\alpha,\Fu) \to (B,\beta,\Fv) \quad\text{and}\quad  (\psi,\Iv): (B,\beta,\Fv)\to (A,\alpha,\Fu)
\]
are two non-degenerate cocycle morphisms such that
\[
(\psi,\Iv)\circ(\phi,\Iu) \asue \id_A \quad\text{and}\quad (\phi,\Iu)\circ(\psi,\Iv) \asue \id_B.
\]
Then there exist mutually inverse cocycle conjugacies
\[
(\phi_0,\Iu^{(0)}): (A,\alpha,\Fu) \to (B,\beta,\Fv) \quad\text{and}\quad  (\psi_0,\Iv^{(0)}): (B,\beta,\Fv)\to (A,\alpha,\Fu)
\]
such that
\[
(\phi,\Iu) \asue (\phi_0,\Iu^{(0)}) \quad\text{and}\quad (\psi,\Iv) \asue (\psi_0,\Iv^{(0)}).
\]
\end{theorem}
\begin{proof}
We consider the two compositions
\[
(\xi,\Ix)=(\psi,\Iv)\circ(\phi,\Iu): (A,\alpha,\Fu)\to(A,\alpha,\Fu)
\]
and
\[
(\theta,\Iy)=(\phi,\Iu)\circ(\psi,\Iv): (B,\beta,\Fv)\to(B,\beta,\Fv).
\]
Then evidently the following represents a commutative diagram of cocycle morphisms:
\[
\xymatrix{
\dots\ar[rr]^\xi && A \ar[rd]^{\phi} \ar[rr]^{\xi} && A \ar[r]^\xi \ar[rd]^\phi & \dots\\
\dots\ar[r]^\theta & B \ar[ru]^{\psi} \ar[rr]^{\theta} && B \ar[ru]^{\psi} \ar[rr]^\theta && \dots
}
\]
Let us we form the stationary inductive limits
\[
(A^{(\infty)},\alpha^{(\infty)},\Fu^{(\infty)})=\lim_{\longrightarrow} \set{ (A,\alpha,\Fu), (\xi,\Ix) }
\]
and
\[
(B^{(\infty)},\beta^{(\infty)},\Fv^{(\infty)})=\lim_{\longrightarrow} \set{ (A,\beta,\Fv), (\theta,\Iy) }.
\]
We consider the universal (equivariant) embeddings
\[
\xi_\infty: (A,\alpha,\Fu)\to (A^{(\infty)},\alpha^{(\infty)},\Fu^{(\infty)}),\quad \theta_\infty: (B,\beta,\Fv)\to (B^{(\infty)},\beta^{(\infty)},\Fv^{(\infty)})
\]
from each respective first building block into the limit.
By the universal properties of both inductive limits, the commutative diagram induces mutually inverse cocycle conjugacies
\[
(\Phi,\IU): (A^{(\infty)},\alpha^{(\infty)},\Fu^{(\infty)}) \to (B^{(\infty)},\beta^{(\infty)},\Fv^{(\infty)}) 
\]
and
\[
(\Psi,\IV): (B^{(\infty)},\beta^{(\infty)},\Fv^{(\infty)}) \to (A^{(\infty)},\alpha^{(\infty)},\Fu^{(\infty)}) 
\]
such that $(\Phi,\IU)\circ \xi_\infty = \theta_\infty\circ(\phi,\Iu)$ and $(\Psi,\IV)\circ\theta_{\infty}=\xi_\infty\circ(\psi,\Iv)$.
By assumption, the cocycle morphism $(\xi,\Ix)$ is asymptotically unitarily equivalent to an isomorphism, namely the identity map on $A$.
By \autoref{lem:asue-limits}, it follows that $\xi_\infty$ is asymptotically unitarily equivalent to a cocycle conjugacy $(\Xi,\IX): (A,\alpha,\Fu)\to (A^{(\infty)},\alpha^{(\infty)},\Fu^{(\infty)})$.
Analogously it follows that $\theta_\infty$  is asymptotically unitarily equivalent to a cocycle conjugacy $(\Theta,\IY): (B,\beta,\Fv)\to (B^{(\infty)},\beta^{(\infty)},\Fv^{(\infty)})$.
Hence we conclude
\[
\begin{array}{ccl}
(\Theta,\IY)^{-1}\circ(\Phi,\IU)\circ(\Xi,\IX) &\asue& (\Theta,\IY)^{-1}\circ(\Phi,\IU)\circ \xi_\infty \\
&=&  (\Theta,\IY)^{-1}\circ\theta_\infty\circ(\phi,\Iu)  \\
&\asue& (\phi,\Iu).
\end{array}
\]
So $(\phi_0,\Iu^{(0)})=(\Theta,\IY)^{-1}\circ(\Phi,\IU)\circ(\Xi,\IX)$ is one of the desired cocycle conjugacies.
For its inverse 
\[
(\psi_0,\Iv^{(0)}) := (\phi_0,\Iu^{(0)})^{-1} = (\Xi,\IX)^{-1}\circ(\Psi,\IV)\circ(\Theta,\IY),
\]
we similarly observe that
\[
\begin{array}{ccl}
(\Xi,\IX)^{-1}\circ(\Psi,\IV)\circ(\Theta,\IY) &\asue& (\Xi,\IX)^{-1}\circ(\Psi,\IV)\circ \theta_\infty \\
&=&  (\Xi,\IX)^{-1}\circ\xi_\infty\circ(\psi,\Iv)  \\
&\asue& (\psi,\Iv).
\end{array}
\]
This finishes the proof.
\end{proof}

\begin{cor}
\label{cor:red-cont-intertwining}
Let $G$ be a second-countable, locally compact group.
Let $(\alpha,\Fu): G\curvearrowright A$ and $(\beta,\Fv): G\curvearrowright B$ be two gently twisted actions on separable \cstar-algebras.
Suppose that
\[
(\phi,\Iu): (A,\alpha,\Fu) \to (B,\beta,\Fv) \quad\text{and}\quad  (\psi,\Iv): (B,\beta,\Fv)\to (A,\alpha,\Fu)
\]
are two proper cocycle morphisms such that
\[
(\psi,\Iv)\circ(\phi,\Iu) \pasue \id_A \quad\text{and}\quad (\phi,\Iu)\circ(\psi,\Iv) \pasue \id_B.
\]
Then $(\phi,\Iu)$ is properly asymptotically unitarily equivalent to a proper cocycle conjugacy.
\end{cor}
\begin{proof}
Since the assumption involves gently twisted actions and proper cocycle morphisms, we may pass to the unitized \cstar-algebras and consider the unital cocycle morphisms
\[
(\phi^\dagger,\Iu): (A^\dagger,\alpha,\Fu)\to (B^\dagger,\beta,\Fv)
\]
and
\[
(\psi^\dagger,\Iv): (B^\dagger,\beta,\Fv)\to (A^\dagger,\alpha,\Fu).
\]
Since $(\psi,\Iv)\circ(\phi,\Iu)$ is properly asymptotically unitarily equivalent to $\id_A$, it follows that $(\psi^\dagger,\Iv)\circ(\phi^\dagger,\Iu)$ is asymptotically unitarily equivalent to $\id_{A^\dagger}$.
Analogously, the composition $(\phi^\dagger,\Iu)\circ(\psi^\dagger,\Iv)$ is asymptotically unitarily equivalent to $\id_{B^\dagger}$.
\autoref{thm:cont-intertwining} implies that $(\phi^\dagger,\Iu)$ is asymptotically unitarily equivalent to a cocycle conjugacy $(\Phi',\IU)$.
Since $\phi^\dagger$ sends $\IC\eins\subset A^\dagger$ to $\IC\eins\subset B^\dagger$ and $A$ to $B$, the same follows for $\Phi'$.
Hence $\Phi'=\Phi^\dagger$ for a unique isomorphism $\Phi: A\to B$.
If $t\mapsto v_t\in\CU(B^\dagger)$ is the unitary path witnessing the asymptotic equivalence between $\phi^\dagger$ and $\Phi^\dagger$, we may divide them by their scalar parts to ensure that $v_t\in\CU(\eins+B)$.
Hence $(\phi,\Iu)$ is indeed properly asymptotically unitarily equivalent to the proper cocycle conjugacy $(\Phi,\IU)$.
\end{proof}

\begin{rem} \label{rem:ue-limits}
The results in this subsection so far are also true if ``asymptotic unitary equivalence'' is replaced with ``approximate unitary equivalence''.
The only difference is that in \autoref{prop:asue-to-an-iso}, the unitary path $z: [0,1]\to\CU(\CM(B))$ needs to be replaced by a single unitary, and the assumption $z_0=\eins$ is dropped; see also \cite[Proposition 2.1]{Szabo18ssa}.
\end{rem}

%%%

\subsection{From approximate morphisms to genuine morphisms}

This subsection should be viewed as a dynamical generalization of \cite[Section 4]{Gabe20}; see also \cite[Proposition 1.3.7]{Phillips00} for a related statement.

\begin{nota} \label{nota:sequence-algebra}
In this subsection we will use the sequence algebra construction.
That is, if $A$ is a \cstar-algebra, then its sequence algebra is denoted by
\[
A_\infty = \ell^\infty(\IN,A)/\set{ (a_n)_n \mid \lim_{n\to\infty} \|a_n\|=0 }.
\]
If $(\alpha,\Fu): G\curvearrowright A$ is a gently twisted action of a locally compact group, then we denote by $\alpha_\infty: G\to\Aut(A_\infty)$ the induced map into the automorphism group arising from applying $\alpha$ componentwise.
In general, $\alpha_\infty$ will not be point-norm continuous.
However, we may define the \cstar-subalgebra
\[
A_{\infty,\alpha} = \set{ x\in A_\infty \mid [g\mapsto \alpha_{\infty,g}(x)] \text{ is continuous} }.
\]
We therefore obtain a gently twisted action $(\alpha_\infty,\Fu): G\curvearrowright A_{\infty,\alpha}$ in the ordinary sense.
The canonical inclusion $A\subseteq A_{\infty,\alpha}$ as constant sequences is then equivariant.
Given a gently twisted $G$-\cstar-algebra $(C,\gamma,\Fw)$, we call a proper cocycle morphism $(\phi,\Iu): (C,\gamma,\Fw)\to (A_{\infty,\alpha},\alpha_\infty,\Fu)$ constant, if it factors through $(A,\alpha,\Fu)$ via the constant inclusion.

If $\kappa: \IN\to\IN$ is any map with $\lim_{n\to\infty} \kappa(n)=\infty$, then it induces an endomorphism $\kappa^*$ on $A_\infty$ via $\kappa^*[(a_n)_n] = [(a_{\kappa(n)})_n]$.
Indeed, we see that $\alpha_{\infty,g}\circ\kappa^*=\kappa^*\circ\alpha_{\infty,g}$ for all $g\in G$, and hence $\kappa^*$ induces an equivariant endomorphism on the gently twisted $G$-\cstar-algebra $(A_{\infty,\alpha},\alpha_\infty,\Fu)$.
\end{nota}

\begin{rem} \label{rem:cocycle-lifting}
Suppose that in the above situation, we are given a continuous map $\Iu: G\to\CU(\eins+A_{\infty})$.
Given any compact set $K\subseteq G$, the restriction $\Iu|_K$ can be viewed as a single unitary
\[
\Iu|_K \in \CU(\eins+\CC(K, A_{\infty})) \subseteq \CU\big( \eins+ \CC(K,A)_\infty \big).
\]
It hence lifts to a sequence of unitaries $\Iu^{(n)}\in \CU\big( \eins+ \CC(K,A) \big)$.
In other words, there is a sequence of continuous maps $\Iu^{(n)}: K\to\CU(\eins+A)$ such that for all $g\in K$, the sequence $\Iu^{(n)}_g$ represents $\Iu_g$.

Now let $K_n\subseteq G$ be an increasing sequence of compact sets with $G=\bigcup_{n\in\IN} K_n$.
For every $n\geq 1$, we can repeat the argument above and find a sequence of continuous maps $\Iu^{(n,k)}: K_n\to\CU(\eins+B)$ such that it uniformly represents $\Iu|_{K_n}\in\CU\big( \eins+ \CC(K_n,B_\infty) \big)$.
Then we have for all $m>n$ that
\[
0= \| (\Iu|_{K_{m}})|_{K_n}-\Iu|_{K_n} \| = \limsup_{k\to\infty} \max_{g\in K_n} \| \Iu^{(m,k)}_g-\Iu^{(n,k)}_g \|
\]
holds for all $n$.
By induction, we can hence choose an increasing sequence $k_n\in\IN$ such that
\[
\sup_{k\geq k_n} \max_{\ell<n} \max_{g\in K_\ell} \| \Iu^{(n,k)}_g-\Iu^{(\ell,k)}_g \| \leq 2^{-n}.
\]
So we see that the sequence of (partially defined) continuous maps on $G$ given by $\IU^{(k)}=\Iu^{(n,k)}$ for $k_n\leq k<k_{n+1}$ is eventually defined on every compact subset of $G$ as $k\to\infty$, and represents the original map $\Iu$ uniformly over compact subsets of $G$.
\end{rem}

\begin{theorem} \label{thm:existence}
Let $G$ be a second-countable, locally compact group.
Let $(\alpha,\Fu) : G\curvearrowright A$ and $(\beta,\Fv) : G\curvearrowright B$ be two gently twisted actions on \cstar-algebras.
Suppose that $A$ is separable.
Let 
\[
(\phi,\Iu): (A,\alpha,\Fu) \to (B_{\infty,\beta},\beta_\infty,\Fv)
\]
be a proper cocycle morphism.
Then the following are equivalent:
\begin{enumerate}[label=\textup{(\roman*)},leftmargin=*]
\item \label{thm:existence:1}
$(\phi,\Iu)$ is properly unitarily equivalent to a constant proper cocycle morphism $(\psi,\Iv): (A,\alpha,\Fu)\to (B,\beta,\Fv)$;
\item \label{thm:existence:2}
For every map $\kappa: \IN\to\IN$ with $\lim_{n\to\infty} \kappa(n)=\infty$, the two proper cocycle morphisms $(\phi,\Iu)$ and $\kappa^*\circ(\phi,\Iu)$ are properly unitarily equivalent;
\item \label{thm:existence:3}
Suppose that $\phi_n: A\to B$ is a sequence of $*$-linear maps lifting $\phi$, and that $\Iu^{(n)}$ is a sequence of partially defined, but eventually everywhere defined, maps from $G$ to $\CU(\eins+B)$ lifting $\Iu$ as in \autoref{rem:cocycle-lifting}.

Then for every $\eps>0$, finite set $\CF\subset A$, compact set $K\subseteq G$ and $m\in\IN$, there exists $k\geq m$ such that for every $n\geq k$, there is a unitary $v\in\CU(\eins+B)$ such that
\[
\max_{a\in\CF} \|v^*\phi_n(a)v-\phi_k(a)\|\leq\eps,\quad \max_{g\in K} \| v^*\Iu^{(n)}_g\beta_g(v)-\Iu^{(k)}_g \|\leq \eps.
\]
\end{enumerate}
\end{theorem}
\begin{proof}
The implications \ref{thm:existence:1}$\Rightarrow$\ref{thm:existence:2} and \ref{thm:existence:3}$\Rightarrow$\ref{thm:existence:2} are trivial.

\ref{thm:existence:2}$\Rightarrow$\ref{thm:existence:3}:
We prove this by contradiction.
Suppose that the claim does not hold.
Then there exists $\eps>0$, finite set $\CF\subset A$, compact set $K\subseteq G$ and $m\in\IN$, such that for every $k\geq m$ there is $n_k\geq k$ such that no unitary $v\in\CU(\eins+B)$ satisfies the above property with $n_k$ in place of $n$.
Since the sequence of partial maps $\Iu^{(n)}$ is eventually defined everywhere, we may assume without loss of generality (by choosing $k$ larger if necessary) that they are always defined on $K$.

If we define $\kappa: \IN\to\IN$ via $\kappa(k)=1$ if $k<m$ and $\kappa(k)=n_k$ if $k\geq m$, then evidently $\lim_{k\to\infty} \kappa(k)=\infty$.
By assumption $(\phi,\Iu)$ is properly unitarily equivalent to $\kappa^*\circ(\phi,\Iu)$.
Find a unitary $v\in\CU(\eins+B_{\infty,\beta})$ with $\kappa^*\circ(\phi,\Iu)=\ad(v)\circ(\phi,\Iu)$, and represent it by a sequence of unitaries $v_k\in\CU(\eins+B)$.
We immediately get for all $a\in A$ that
\[
0=\|\kappa^*\circ\phi(a)-\ad(v)(\phi(a))\| = \limsup_{k\to\infty} \| \phi_{n_k}(a) - v_k\phi_k(a)v_k^* \|.
\]
Furthermore, we have
\[
0 \stackrel{\ref{rem:cocycle-lifting}}{=}\max_{g\in K} \| \kappa^*(\Iu_g)-v\Iu_g\beta_{\infty,g}(v)^* \| = \limsup_{k\to\infty} \max_{g\in K} \| \Iu^{(n_k)}_g - v_k\Iu^{(k)}_g\beta_g(v_k)^* \|.
\]
But now we see that this yields a contradiction to our assumption.

\ref{thm:existence:2}+\ref{thm:existence:3}$\Rightarrow$\ref{thm:existence:1}
Let $\CF_n\subset A$ be an increasing sequence of finite sets with dense union.
Let $K_n\subseteq G$ be an increasing sequence of compact sets with $G=\bigcup_{n\in\IN} K_n$.
By inductively applying the statement in \ref{thm:existence:3}, we may choose an increasing sequence of natural numbers $k_n\in\IN$ with $k_0=1$ such that for every $n\geq 1$, the number $k_n$ satisfies the conclusion of \ref{thm:existence:3} in place of $k$, with respect to $\CF_n$ in place of $\CF$, $K_n$ in place of $K$, $2^{-n}$ in place of $\eps$, and $k_{n-1}+1$ in place of $m$.
This allows us to find a sequence of unitaries $v_n\in\CU(\eins+B)$ such that
\[
\max_{a\in\CF_n} \|v_n^*\phi_{k_n}(a)v-\phi_{k_{n-1}}(a)\|\leq 2^{-n},\quad \max_{g\in K_n} \| v_n^*\Iu^{(k_n)}_g\beta_g(v)-\Iu^{(k_{n-1})}_g \|\leq 2^{-n}.
\]
Define the unitaries $V_n=v_nv_{n-1}\cdots v_1$ in $\CU(\eins+B)$ for $n\geq 1$, and set $V=[(V_n)_{n\geq 1}]\in\CU(\eins+B_\infty)$.

Then evidently it follows for every $a\in\bigcup_{n\in\IN} \CF_n$ that the sequence $n\mapsto V_n^*\phi_{k_n}(a)V_n$ is Cauchy, and therefore has a limit in $B$.
Define $\kappa: \IN\to\IN$ via $\kappa(n)=k_n$.
Since $k_n\to\infty$ as $n\to\infty$, it follows that $\kappa^*\circ\phi=[(\phi_{k_n})_{n\geq 1}]: A\to B_{\infty,\beta}$ is a $*$-homomorphism.
We have justified that the $*$-homomorphism $\ad(V^*)\circ\kappa^*\circ\phi$ maps a dense subset into $B$, which immediately implies that its entire range is in $B$.
We may therefore obtain a well-defined $*$-homomorphism $\psi: A\to B$ via $\psi(a)=\lim_{n\to\infty} V_n^*\phi_{k_n}(a)V_n$ for all $a\in A$.

For every compact set $K\subseteq G$, the sequence of (partially defined) continuous maps $n\mapsto \big[ g\mapsto V_n^*\Iu^{(k_n)}_g\beta_g(V_n) \big]$ is eventually defined on $K$ and uniformly satisfies the Cauchy criterion in norm over $K$.
So we may define $\Iv_g=\lim_{n\to\infty} V_n^*\Iu^{(k_n)}_g\beta_g(V_n)$ for all $g\in G$, which yields a continuous map into $\CU(\eins+B)$.
For the unitary $V\in\CU(\eins+B_\infty)$, this implies the equation $\beta_{\infty,g}(V)=(\kappa^*)^\dagger(\Iu_g^*)V\Iv_g$ for all $g\in G$.
As both maps $(\kappa^*)^\dagger\circ\Iu$ and $\Iv$ are norm-continuous, we may conclude $V\in\CU(\eins+B_{\infty,\beta})$.

In summary, we get that $(\psi,\Iv): (A,\alpha,\Fu) \to (B,\beta,\Fv)$ is a proper cocycle morphism.
By construction, it is properly unitarily equivalent to $\kappa^*\circ(\phi,\Iu): (A,\alpha,\Fu)\to (B_{\infty,\beta},\beta_\infty,\Fv)$, which in turn is properly unitarily equivalent to $(\phi,\Iu)$ by assumption.
This finishes the proof.
\end{proof}

%%%%%%%%%%%%%%%%%%%%%%%%%%%%%%%%%%%%%%%%%%%%%%%%%%%%

\section{Strong self-absorption revisited}

In this section we revisit the concept of strongly self-absorbing \cstar-dynamical systems \cite{Szabo18ssa,Szabo18ssa2,Szabo17ssa3} and make some observation based on the intertwining results obtained so far.

\begin{defi}
Let $G$ be a second-countable, locally compact group.
Let $D$ be a separable unital \cstar-algebra and $\gamma: G\curvearrowright D$ an action.
Let $A$ be a separable \cstar-algebra and $(\alpha,\Fu): G\curvearrowright A$ a twisted action.
We say that $(\alpha,\Fu)$ strongly absorbs $\gamma$, if the equivariant embedding
\[
\id_A\otimes\eins_D: (A,\alpha,\Fu)\to (A\otimes D, \alpha\otimes\gamma, \Fu\otimes\eins)
\]
is approximately unitarily equivalent to a cocycle conjugacy.
\end{defi}

\begin{prop} \label{prop:strong-absorption-infty}
Let $G$ be a second-countable, locally compact group.
Let $D$ be a separable unital \cstar-algebra and $\gamma: G\curvearrowright D$ an action.
Let $A$ be a separable \cstar-algebra and $(\alpha,\Fu): G\curvearrowright A$ a twisted action.
Then $(\alpha,\Fu)$ strongly absorbs $\gamma$ if and only if it strongly absorbs $\gamma^{\otimes\infty}: G\curvearrowright D^{\otimes\infty}$. 
\end{prop}
\begin{proof}
Let us address the ``only if'' part.
We have that
\[
(A\otimes D^{\otimes\infty}, \alpha\otimes\gamma^{\otimes\infty}, \Fu\otimes\eins_D^{\otimes\infty} ) = \lim_{\longrightarrow} \set{ (A\otimes D^{\otimes n}, \alpha\otimes\gamma^{\otimes n}, \Fu\otimes\eins_D^{\otimes n} ), \id\otimes\eins_D } 
\]
By assumption, $(\alpha,\Fu)$ strongly absorbs $\gamma$, which means that $\id_A\otimes\eins_D$ is approximately unitarily equivalent to a cocycle conjugacy.
Hence the same is true for the connecting maps $\id_A\otimes\id_D^{\otimes n}\otimes\eins_D$ above.
By \autoref{rem:ue-limits}, it follows that the canonical embedding
\[
\id_A\otimes\eins_{D^{\otimes\infty}}: (A,\alpha,\Fu)\to (A\otimes D^{\otimes\infty}, \alpha\otimes\gamma^{\otimes\infty}, \Fu\otimes\eins_D^{\otimes\infty} )
\]
is also approximately unitarily equivalent to a cocycle conjugacy, which confirms the claim.

For the ``if'' part, assume that $(\alpha,\Fu)$ strongly absorbs $\gamma^{\otimes\infty}$.
Assume that $\id_A\otimes\eins_{D^{\otimes\infty}}$ is approximately unitarily equivalent to the cocycle conjugacy $(\phi,\Iu): (A,\alpha,\Fu)\to(A\otimes D^{\otimes\infty},\alpha\otimes\gamma^{\otimes\infty},\Fu\otimes\eins)$.
Then 
\[
(\phi\otimes\id_D,\Iu\otimes\eins_D)^{-1}\circ (\phi,\Iu): (A,\alpha,\Fu)\to (A\otimes D,\alpha\otimes\gamma,\Fu\otimes\eins_D)
\]
is also a cocycle conjugacy.
We moreover have
\[
\begin{array}{cl}
\multicolumn{2}{l}{ (\phi\otimes\id_D,\Iu\otimes\eins_D)^{-1}\circ (\phi,\Iu) } \\
\ue& (\phi\otimes\id_D,\Iu\otimes\eins_D)^{-1}\circ (\id_A\otimes\eins_{D^{\otimes\infty}}) \\
=& (\phi^{-1}\otimes\id_D,\phi^{-1}(\Iu)^*\otimes\eins_D)\circ (\id_A\otimes\eins_{D^{\otimes\infty}}) \\
=& (\phi^{-1}\otimes\eins_D,\phi^{-1}(\Iu)^*\otimes\eins_D) \circ (\id_A\otimes\eins_{D^{\otimes\infty}}) \\
=& (\id_A\otimes\eins_D)\circ (\phi,\Iu)^{-1}\circ (\id_A\otimes\eins_{D^{\otimes\infty}}) \\
\ue& (\id_A\otimes\eins_D)\circ\id_A \ = \ \id_A\otimes\eins_D.
\end{array}
\]
Hence $(\alpha,\Fu)$ strongly absorbs $\gamma$.
\end{proof}

The following should be seen as a revised definition from \cite{Szabo18ssa}. 

\begin{defi}[replacing {\cite[Definition 3.1]{Szabo18ssa}}] \label{def:ssa}
Let $\CD$ be a separable unital \cstar-algebra.
Let $G$ be a second-countable, locally compact group, and let $\gamma: G\curvearrowright\CD$ be an action.
We say that $\gamma$ is strongly self-absorbing, if $\gamma$ strongly absorbs $\gamma$ in the above sense.
\end{defi}

Comparing the next proposition below with \cite[Theorem 4.6]{Szabo18ssa} yields that the above definition turns out to coincide with the notion called ``semi-strongly self-absorbing'' in previous work.
In hindsight this is hence the more natural concept compared to the one in \cite[Definition 3.1]{Szabo18ssa}, which was hinted at in its introduction but lacked hard evidence at the time.

\begin{nota}
For the rest of this section, we will use standard terminology involving central sequence algebras inspired by \cite{Kirchberg04}.
Let $(\alpha,\Fu): G\curvearrowright A$ be a gently twisted action on a \cstar-algebra $A$, and recall the terminology from \autoref{nota:sequence-algebra}.
Suppose that $B\subset A_{\infty,\alpha}$ is an $\alpha_\infty$-invariant \cstar-subalgebra containing the set $\set{\Fu_{g,h}-\eins}_{g,h\in G}$.

Then $\alpha_\infty$ restrict to a genuine continuous action $\alpha_\infty: G\curvearrowright A_{\infty,\alpha}\cap B'$, which in turn induces a continuous action
\[
\tilde{\alpha}_\infty: G\curvearrowright F(B,A_{\infty,\alpha}) := (A_{\infty,\alpha}\cap B')/(A_{\infty,\alpha}\cap B^\perp),
\]
where $A_{\infty,\alpha}\cap B^\perp=\set{ x\in A_{\infty,\alpha} \mid xB=\set{0}=Bx}$.
If $B$ is $\sigma$-unital, then $F(B,A_{\infty,\alpha})$ always is a unital \cstar-algebra.
If $B=A$, then we abbreviate $F(A,A_{\infty,\alpha})=F_{\infty,\alpha}(A)$.
\end{nota}

Recall that an action $\gamma: G\curvearrowright D$ on a separable unital \cstar-algebra is said to have approximately $G$-inner half-flip, if the two equivariant unital embeddings
\[
\id_D\otimes\eins_D, \eins_D\otimes\id_D : (D,\gamma)\to (D\otimes D,\gamma\otimes\gamma)
\]
are approximately unitarily equivalent (as cocycle morphisms).

\begin{prop} \label{prop:ssa-old-sssa}
Let $\gamma: G\curvearrowright\CD$ be an action on a separable, unital \cstar-algebra.
The following are equivalent:
\begin{enumerate}[label=\textup{(\roman*)},leftmargin=*]
\item \label{prop:ssa-old-sssa:1}
$\gamma$ is strongly self-absorbing in the sense of \autoref{def:ssa}.
\item \label{prop:ssa-old-sssa:2}
$\gamma$ has approximately $G$-inner half-flip and there exists a unital equivariant $*$-homomorphism from $(\CD,\gamma)$ to $(\CD_{\infty,\gamma}\cap\CD',\gamma_\infty)$.
\end{enumerate}
\end{prop}
\begin{proof}
\ref{prop:ssa-old-sssa:2}$\Rightarrow$\ref{prop:ssa-old-sssa:1}:
This follows directly from \cite[Theorem 2.6]{Szabo18ssa}.
Here note that the proof consisted in showing that the assumptions of \cite[Proposition 2.1]{Szabo18ssa} (which coincides with the analog of \autoref{prop:asue-to-an-iso} referenced in \autoref{rem:ue-limits}) hold for the equivariant embedding $\id_\CD\otimes\eins_D: (\CD,\gamma)\to (\CD\otimes\CD,\gamma\otimes\gamma)$.
Hence it is approximately unitarily equivalent to a cocycle conjugacy, which is what it means for $\gamma$ to be strongly self-absorbing.

\ref{prop:ssa-old-sssa:1}$\Rightarrow$\ref{prop:ssa-old-sssa:2}:
If $\gamma$ is strongly self-absorbing, then it follows by \autoref{prop:strong-absorption-infty} that $\gamma$ is cocycle conjugate to $\gamma^\infty: G\curvearrowright\CD^{\otimes\infty}$.
So the existence of some unital equivariant $*$-homomorphism from $(\CD,\gamma)$ to $(\CD_{\infty,\gamma}\cap\CD',\gamma_\infty)$ is trivial.
In order to prove that $\gamma$ has approximately $G$-inner half-flip, we can proceed exactly as in the original proof by Toms--Winter in \cite[Proposition 1.5]{TomsWinter07}.
Let $(\phi,\Iu): (\CD,\gamma)\to (\CD\otimes\CD,\gamma\otimes\gamma)$ be a cocycle conjugacy that is approximately unitarily equivalent to $\id_\CD\otimes\eins_\CD$.
Define the cocycle morphism $(\psi,\Iv)=(\phi,\Iu)^{-1}\circ(\eins_\CD\otimes\id_\CD)$ on $(\CD,\gamma)$.
Then
\[
\begin{array}{ccl}
\eins_\CD\otimes\id_\CD &=& (\phi,\Iu)\circ(\phi,\Iu)^{-1}\circ (\eins\otimes\id_\CD) \\
 &=& (\phi,\Iu)\circ(\psi,\Iv) \\
&\ue& (\id_\CD\otimes\eins_\CD)\circ(\psi,\Iv) \\
&=& (\psi,\Iv)\otimes\eins_\CD \\
&=& (\phi\otimes\id_\CD, \Iu\otimes\eins)^{-1}\circ(\eins\otimes\id_\CD\otimes\eins_\CD) \\
&\ue& (\phi\otimes\id_\CD, \Iu\otimes\eins)^{-1}\circ(\eins_\CD\otimes\eins_\CD\otimes(\psi,\Iv)) \\
&=& \eins_\CD\otimes(\psi,\Iv) \\
&\ue& \id_\CD\otimes\eins_\CD.
\end{array}
\]
Here we have used the equivalence between $\eins_\CD\otimes\id_\CD$ and $(\psi,\Iv)\otimes\eins_\CD$ from the fifth line onwards and applied the equivariant flip automorphism to it.
This finishes the proof.
\end{proof}

We shall also revisit the equivariant McDuff-type theorem and prove a variant of \cite[Theorem 2.2]{Szabo17ssa3} intended for use in subsequent work, which requires a stronger assumption but yields a stronger statement.

\begin{theorem} \label{thm:McDuff}
Let $G$ be a second-countable, locally compact group.
Let $\gamma: G\curvearrowright\CD$ be a strongly self-absorbing action, and $(\alpha,\Fu): G\curvearrowright A$ a gently twisted action on a separable \cstar-algebra.
Suppose that $\gamma$ is equivariantly $\CZ$-stable, i.e., $\gamma$ is cocycle conjugate to $\gamma\otimes\id_\CZ$ for the Jiang--Su algebra $\CZ$\footnote{See \cite{JiangSu99} for its introduction. In contrast to the redundance of $\CZ$-stability for strongly self-absorbing \cstar-algebras \cite{Winter11}, note that there are examples of strongly self-absorbing actions that fail to be equivariantly $\CZ$-stable; see \cite[Example 6.4]{Szabo18ssa2}.}.
Suppose that there exists a unital equivariant $*$-homomorphism from $(\CD,\gamma)$ to the central sequence algebra $(F_{\infty,\alpha}(A),\tilde{\alpha}_\infty)$.
Then the equivariant second-factor embedding
\[
\eins_\CD\otimes\id_A: (A,\alpha,\Fu)\to ( \CD\otimes A, \gamma\otimes\alpha, \eins_\CD\otimes \Fu)
\]
is properly asymptotically unitarily equivalent to a proper cocycle conjugacy.
\end{theorem}
\begin{proof}
This is almost the same as the proof of either \cite[Theorem 3.7]{Szabo18ssa} or \cite[Theorem 2.2]{Szabo17ssa3}, but with a slight additional twist, so we shall give it for completeness.

Since we assumed $(\alpha,\Fu)$ to be a gently twisted action, i.e., $\set{\Fu_{g,h}}_{g,h\in G}\subset\CU(\eins+A)$, we have that $\gamma\otimes\alpha$ induces a genuine action on $(\CD\otimes A)_\infty\cap (\eins_\CD\otimes A)'$.
Denote by 
\[
\pi: (\CD\otimes A)_{\infty,\gamma\otimes\alpha}\cap (\eins_\CD\otimes A)' \to F(\eins_\CD\otimes A,(\CD\otimes A)_{\infty,\gamma\otimes\alpha})
\] 
the canonical (equivariant) surjection. 
We will use below in a crucial way that $\pi$ induces a surjection between the fixed point algebras with respect to the obvious action induced by $\gamma\otimes\alpha$ on both sides, which follows from the fact that the kernel is a $G$-$\sigma$-ideal; see \cite[Example 4.4 and Proposition 4.5]{Szabo18ssa2}.\footnote{To be precise, only the case of genuine actions is treated in \cite{Szabo18ssa2}, but the general case is true with the same proof verbatim.}

By assumption, we have an equivariant, unital $*$-homomorphism from $(\CD,\gamma)$ to $\big( F_{\infty,\alpha}(A), \tilde{\alpha}_\infty \big)$. Consider the canonical inclusions 
\[
F_{\infty,\alpha}(A) ,\ \CD \ \subset \ F(\eins_\CD\otimes A,(\CD\otimes A)_{\infty,\gamma\otimes\alpha}),
\]
which define commuting \cstar-subalgebras. Since these inclusions are natural, they are equivariant with respect to the induced actions of $\alpha$, $\gamma$ and $\gamma\otimes\alpha$.
By assumption, it follows that we have a unital and equivariant $*$-homomorphism
\[
\phi: (\CD\otimes\CD,\gamma\otimes\gamma)\to \big( F(\eins_\CD\otimes A,(\CD\otimes A)_{\infty,\gamma\otimes\alpha}) , (\gamma\otimes\alpha)^\sim_\infty \big)
\]
satisfying $\phi(d\otimes\eins_\CD)\cdot (\eins_\CD\otimes a)=d\otimes a$ and $\phi(\eins_\CD\otimes d)\cdot (\eins_\CD\otimes a)\in\eins_{\CD}\otimes A_\infty$ for all $a\in A$ and $d\in\CD$.

%Now choose $\eps>0$ and compact sets $\CF^\CD\subset\CD$, $\CF^A\subset A$, and $K\subseteq G$ be a compact set.
%Without loss of generality, assume that $\CF^\CD$ and $\CF^A$ consist of contractions.
Since $\gamma$ is strongly self-absorbing, it follows from \cite[Proposition 3.3]{Szabo18ssa2} that there exist unitaries $u,v\in F(\eins_\CD\otimes A,(\CD\otimes A)_{\infty,\gamma\otimes\alpha})^{(\gamma\otimes\alpha)_\infty^\sim}$ such that $\ad(uvu^*v^*)\circ\phi\circ(\eins_\CD\otimes\id_\CD)= \phi\circ(\id_\CD\otimes\eins_\CD)$.
As $\gamma$ is equivariantly $\CZ$-stable, it is possible to find a unital copy $\CZ\subset F_{\gamma\otimes\alpha}(\eins_\CD\otimes A,(\CD\otimes A)_\infty)^{(\gamma\otimes\alpha)_\infty^\sim}$ that pointwise commutes with both $u$ and $v$.
By \cite[Corollary 2.6]{Jiang97}, it follows that the unitary
\[
z_1=uvu^*v^* \ \in \ \CU\Big( F(\eins_\CD\otimes A,(\CD\otimes A)_{\infty,\gamma\otimes\alpha})^{(\gamma\otimes\alpha)_\infty^\sim} \Big)
\]
is homotopic to the unit.
We observe
\[
z_1^*(d\otimes a)z_1 = \big( \ad(vuv^*u^*)\circ\phi\big)(d\otimes\eins_\CD)\cdot (\eins_\CD\otimes a) = \phi(\eins_\CD\otimes d)\cdot (\eins_\CD\otimes a) \in \eins_\CD\otimes A_\infty
\] 
for all $a\in A$ and $d\in\CD$.

If we write $z_1$ as a finite product of exponentials $\exp(ih_1)\cdots\exp(h_\ell)$ with self-adjoints $h_j\in F(\eins_\CD\otimes A,(\CD\otimes A)_{\infty,\gamma\otimes\alpha})^{(\gamma\otimes\alpha)_\infty^\sim}$, we can lift these elements to self-adjoints in the fixed point algebra $(\CD\otimes A)^{(\gamma\otimes\alpha)_\infty}$ with the same norm, which can in turn be represented by bounded sequences of self-adjoints $h_j^{(n)}\in \CD\otimes A$.
These elements are approximately central relative to $\eins_\CD\otimes A$ as $n\to\infty$, and satisfy $\lim_{n\to\infty} \max_{g\in K} \| h_j^{(n)}-(\gamma\otimes\alpha)_g(h_j^{(n)}) \| = 0$ for every compact set $K\subseteq G$.
We can therefore define a uniformly continuous sequence of maps $z^{(n)}: [0,1]\to\CU\big(\eins+\CD\otimes A\big)$ via $z^{(n)}_t=\exp(ith_1^{(n)})\cdots\exp(ith_\ell^{(n)})$.
Then evidently $z^{(n)}_0=\eins$, and the sequence $z^{(n)}_1$ represents $z_1$.

We may finally observe the following properties:
\begin{itemize}
\item $\dst \lim_{n\to\infty} \max_{0\leq t\leq 1} \|[z^{(n)}_t, \eins_\CD\otimes a]\|=0$ for all $a\in A$;
\item $\dst \lim_{n\to\infty} \dist( z^{(n)*}_1(d\otimes a)z^{(n)}_1, \eins_\CD\otimes A)=0$ for all $d\in \CD$ and $a\in A$;
\item $\dst \lim_{n\to\infty} \max_{g\in K} \max_{0\leq t\leq 1} \| z^{(n)}_t-(\gamma\otimes\alpha)_g(z^{(n)}_t) \| = 0$ for every compact set $K\subseteq G$.
\end{itemize}
We conclude that the unitized equivariant embedding 
\[
(\eins_\CD\otimes\id_A)^\dagger: (A^\dagger,\alpha,\Fu)\to ( (\CD\otimes A)^\dagger, \gamma\otimes\alpha, \eins_\CD\otimes \Fu)
\]
satisfies the assumptions of \autoref{prop:asue-to-an-iso} when viewed as a unital cocycle morphism.
In particular, it is asymptotically unitarily equivalent to a cocycle conjugacy.
Since we may assume that the unitary path witnessing the equivalence is in $\CU(\eins+\CD\otimes A)$ and the resulting isomorphism will respect the scalar part $\IC\cdot\eins$ on both sides, it follows that we obtain a proper cocycle conjugacy between $(A,\alpha,\Fu)$ and $(\CD\otimes A, \gamma\otimes\alpha,\eins_\CD\otimes\Fu)$, which is properly asymptotically unitarily equivalent to $\eins_\CD\otimes\id_A$.
This finishes the proof.
\end{proof}

%%%%%%%%%%%%%%%%%%%%%%%%%%%%%%%%%%%%%%%%%%%%%%%%%%%%

\section{Remarks on $G$-equivariant $KK$-theory}

The aim of this section is to observe that $G$-equivariant $KK$-theory \cite{Kasparov88} is functorial on the cocycle category $\cstar_{G,\sep}$ of separable $G$-\cstar-algebras.
We note that there is little novelty behind this observation, as it is a fairly easy extension of the usual functoriality of $KK^G$ with regard to equivariant $*$-homomorphisms.
Nevertheless, this may turn out to be a much more natural viewpoint in the future to exploit $KK^G$ for classifying group actions up to cocycle conjugacy, so we include it here; cf.\ \cite{Meyer19}.

Moreover, as Thomsen's description of $KK^G$-groups utilizes cocycle representations\footnote{Throughout this section we will keep in mind \autoref{rem:general-coc-reps} and not necessarily assume that the map belonging to a cocycle representation between $G$-\cstar-algebras is extendible, unless we need to compose them, such as in \autoref{prop:KKG-comp}.} in his definition of equivariant Cuntz pairs --- see \cite[Section 3]{Thomsen98} --- we shall describe the functoriality both in the common Fredholm picture and the \emph{Cuntz--Thomsen} picture.

In this section we will assume familiarity of the reader with the standard (Fredholm) picture of $G$-equivariant $KK$-theory and the standard associated terminology; see \cite[Chapter 20]{BlaKK} for details.
The only major deviation is that we will denote $KK^G(\alpha,\beta)$ for the equivariant $KK$-group associated to two actions $\alpha: G\curvearrowright A$ and $\beta: G\curvearrowright B$ rather than $KK^G(A,B)$, in order to avoid confusion when dealing with different actions on the same \cstar-algebra.

\begin{nota} \label{rem:KKG-old-functoriality}
Let $\beta: G\curvearrowright B$ be an action on a \cstar-algebra, and suppose that $\Iu: G\to\CU(\CM(B))$ is a $\beta$-cocycle.
We will denote by $B^\Iu$ the Hilbert $(B,\beta)$-module that is equal to $B$ as an ordinary Hilbert right-$B$-module, but is equipped with the continuous linear $G$-action given by $g\cdot b := \Iu_g\beta_g(b)$ for all $g\in G$ and $b\in B$.
Note that under the canonical isomorphism of \cstar-algebras $\IB(B^\Iu)\cong\CM(B)$, the conjugation $G$-action $[g\mapsto g\circ\bullet\circ g^{-1}]$ is simply given as the canonical extension of the cocycle perturbed action $\beta^\Iu$.
Namely we can see that for all $x\in\CM(B)$, $b\in B$, and $g\in G$, one has
\[
\begin{array}{ccl}
(g\circ x\circ g^{-1})(b) &=& (g\circ x)(\Iu_{g^{-1}}\beta_{g^{-1}}(b)) \\
&=& (g\circ x)\big( \beta_g^{-1}( \Iu_g^* b ) \big) \\
&=& \Iu_g\cdot \beta_g\Big( x \beta_g^{-1} ( \Iu_g^* b) \Big) \\
&=& \beta^\Iu_g(x)\cdot b.
\end{array}
\]  
When we refer to the \emph{obvious} Hilbert $(B,\beta)$-module structure on $B$, it is meant that we consider $B^\eins$ in this way.

Let $\alpha: G\curvearrowright A$ be another action on a \cstar-algebra.
Recall that for an equivariant $*$-homomorphism $\phi: (A,\alpha)\to (B,\beta)$ between separable $G$-\cstar-algebras, its associated element $KK^G(\phi)\in KK^G(\alpha,\beta)$ is the equivalence class of the equivariant Kasparov triple $(B,\phi,0)$, where $B$ carries the obvious right Hilbert $(B,\beta)$-module structure.
\end{nota}

\begin{nota} \label{nota:kkG}
We denote by $kk^G$ the category whose objects are separable $G$-\cstar-algebras, and where morphisms from $(A,\alpha)$ to $(B,\beta)$ are elements in $KK^G(\alpha,\beta)$.
Then it is a tautological fact that the category of separable $G$-\cstar-algebras, if equipped with non-degenerate equivariant $*$-homomorphisms as the arrows, comes with an obvious natural transformation onto $kk^G$ which is the identity map on objects, and which assigns an equivariant $*$-homomorphism $\phi: (A,\alpha)\to (B,\beta)$ to its $KK$-element $KK^G(\phi)\in KK^G(\alpha,\beta)$.
\end{nota}

\begin{defi} \label{def:KKG-new-functoriality}
For a cocycle morphism $(\phi,\Iu): (A,\alpha)\to (B,\beta)$ between separable $G$-\cstar-algebras, we associate the element $KK^G(\phi,\Iu)\in KK^G(\alpha,\beta)$ represented by the Kasparov triple $(B^\Iu,\phi,0)$.\footnote{This is a Kasparov triple precisely by the equivariance condition $\ad(\Iu_g)\circ\beta_g\circ\phi=\phi\circ\alpha_g$.}
\end{defi}

\begin{rem}
Obviously the construction in \autoref{def:KKG-new-functoriality} extends the one in \autoref{rem:KKG-old-functoriality}, which describes the case $\Iu=\eins$.
\end{rem}

We may then come to our main observation:

\begin{prop} \label{prop:KKG-comp}
For any two cocycle morphisms
\[
(A,\alpha) \stackrel{(\phi,\Iu)}{\longrightarrow} (B,\beta) \stackrel{(\psi,\Iv)}{\longrightarrow} (C,\gamma)
\]
between separable $G$-\cstar-algebras, if $\psi$ is non-degenerate, one has
\[
KK^G(\phi,\Iu)\otimes KK^G(\psi,\Iv) = KK^G(\psi\circ\phi, \psi^+(\Iu)\Iv) \in KK^G(\alpha,\gamma).
\]
In particular, the assignment $(\phi,\Iu)\mapsto KK^G(\phi,\Iu)$ defines a natural transformation from the cocycle category $\cstar_{G,\sep}$ onto $kk^G$ which extends the one in \autoref{nota:kkG}.
\end{prop}
\begin{proof}
In the special case at hand, the Kasparov triples representing our two factors are of a very special form, and hence it is easy to form their Kasparov product via considering a balanced tensor product; cf.\ \cite[Example 18.4.2]{BlaKK} or \cite[Paragraph 2.4]{Meyer00}.
The left hand side of the equation is represented by the equivariant $(A,\alpha)-(C,\gamma)$ Kasparov triple given by
\[
(E,\kappa,0) \cong ( B\otimes_\psi C, \phi\otimes\eins, 0).
\]
We may identify $E = B\otimes_\psi C \cong C$ via $b\otimes c\mapsto \psi(b)c$ as a Hilbert right-$C$-module, and under this identification one has $\kappa(a)x=\psi(\phi(a))x$ for all $a\in A$ and $x\in C$.
Furthermore the $G$-action on $E$ is given by $g\cdot (b\otimes c) = (\Iu_g\beta_g(b))\otimes (\Iv_g\gamma_g(c))$, which under this identification becomes $g\cdot (\psi(b)c) = \psi(\Iu_g\beta_g(b))\Iv_g\gamma_g(c) = \psi^+(\Iu_g)\Iv_g \gamma_g\big( \psi(b)c \big)$, leading to the formula $g\cdot x = \psi^+(\Iu_g)\Iv_g\gamma_g(x)$ for the $G$-action turning $C$ into a right Hilbert $(C,\gamma)$-module.
Evidently the right hand side of the equation in our claim is represented by the same Kasparov triple, which shows the claim.
\end{proof}

\begin{rem}
Let $\alpha: G\curvearrowright A$ be an action on a separable \cstar-algebra.
It is well-known that for any $\alpha$-cocycle $\Iu$, its cocycle perturbation $\alpha^\Iu$ is naturally $KK^G$-equivalent to $\alpha$.
Although this has been known to follow from abstract reasons due to $KK^G$ being a stable functor, a more specific description of the equivalence has been given on various occasions. 
By consulting the beginning of \cite[Section 3]{Thomsen98}, one sees that the canonical equivalence in $KK^G(\alpha^\Iu,\alpha)$\footnote{In \cite{Thomsen98} this element is denoted $\Iu_\#^{-1}$, or rather the group homomorphism $KK^G(\delta,\alpha^\Iu)\to KK^G(\delta,\alpha)$ that is induced by multiplication from the right by this element, where $\delta: G\curvearrowright D$ is any other action on a separable \cstar-algebra.} is simply given by $KK^G(\id_A,\Iu)$ in the sense of \autoref{def:KKG-new-functoriality}.
Recall from \autoref{ex:exterior-equivalence} that within our formalism, the exterior equivalence $(\id_A,\Iu): (A,\alpha^\Iu)\to (A,\alpha)$ is the canonical isomorphism in the category $\cstar_G$ induced by the cocycle $\Iu$.
\end{rem}

We shall also briefly review the Cuntz--Thomsen picture \cite{Cuntz83, Cuntz84, Higson87, Thomsen98} of $KK^G$, but we will somewhat adapt the terminology and notation for our purposes.
The reader should consult \cite[Section 3]{Thomsen98} for the theory we are about to recall.

\begin{nota}
We write $\CK$ for the \cstar-algebra of compact operators on some separable infinite-dimensional Hilbert space.
For an action $\beta: G\curvearrowright B$ on a \cstar-algebra, let us write as short-hand $B^s = B\otimes\CK$ and $\beta^s=\beta\otimes\id_\CK: G\curvearrowright B^s$.
Moreover we will abbreviate $B[0,1]=\CC[0,1]\otimes B$ and $\beta[0,1]=\id_{\CC[0,1]}\otimes\beta$.
\end{nota}

\begin{defi} \label{def:equi-Cuntz-pair}
Let $\alpha: G\curvearrowright A$ and $\beta: G\curvearrowright B$ be two actions on separable \cstar-algebras.
An $(\alpha,\beta)$-Cuntz pair $(\phi^\pm, \Iu^\pm)$ consists of two cocycle representations
\[
(\phi^i,\Iu^i): (A,\alpha) \to (\CM(B^s),\beta^s),\quad i\in\set{+,-},
\]
such that one has 
\begin{itemize}
\item $\phi^+(a)-\phi^-(a)\in B^s$ for all $a\in A$;
\item $\Iu^+_g-\Iu^-_g \in B^s$ for all $g\in G$;
\item the map $[g\mapsto \Iu^+_g-\Iu^-_g]$ is norm-continuous on $G$.
\end{itemize}
\end{defi}

First let us use this opportunity to prove that the latter assumption about the cocycles is in fact redundant:

\begin{prop} \label{prop:automatic-continuity}
Let $\beta: G\curvearrowright B$ be an action on a non-unital \cstar-algebra.
\begin{enumerate}[label=\textup{(\roman*)},leftmargin=*]
\item \label{prop:automatic-continuity:1}
If $U\in\CU(\CM(B))$ is a unitary such that the coboundary $U\beta_\bullet(U)^*$ takes values in $\CU(\eins+B)$, then $[g\mapsto U\beta_g(U)^*]$ is a norm-continuous map.
\item \label{prop:automatic-continuity:2}
Let $\Iu, \Iv: G\to\CU(\CM(B))$ be two strictly continuous $\beta$-cocycles, and suppose that $\set{\Iu_g-\Iv_g}_{g\in G}\subseteq B$.
Then the assignment $[g\mapsto\Iu_g-\Iv_g]$ defines a norm-continuous map on $G$.
\end{enumerate}
In particular, the last assumption in \autoref{def:equi-Cuntz-pair} is redundant.
\end{prop}
\begin{proof}
\ref{prop:automatic-continuity:1}:
As one has $\eins-U\beta_g(U)^*=(\beta_g(U)-U)\beta_g(U^*)\in B$ for all $g\in G$, the assumption translates to $\set{U-\beta_g(U)}_{g\in G}\subseteq B$.
In other words, the image of $U$ in the corona algebra $\CM(B)/B$ is a fixed point under the induced (algebraic) $G$-action.
By \cite[Theorem 2.1]{Thomsen00}, it follows that $[g\mapsto\beta_g(U)]$ is norm-continuous, which implies our claim.

\ref{prop:automatic-continuity:2}: 
We may assume without loss of generality that $(B,\beta)\cong (B^s,\beta^s)$, by replacing $\Iu_\bullet$ with $\Iu_\bullet\oplus(\eins-e_{11})$ and $\Iv_\bullet$ with $\Iv_\bullet\oplus(\eins-e_{11})$ inside $\CM(B\otimes\CK)$ if necessary.
Let us choose sequences of isometries $r_n,t_n\in\CM(\CK)\subseteq\CM(B)^\beta$ such that
\[
\eins = \sum_{n=1}^\infty t_nt_n^*+r_nr_n^*
\]
in the strict topology.
For $g\in G$, we define the strict limits
\[
\IU_g = \sum_{n=1}^\infty t_n\Iu_gt_n^* + r_n\Iv_gr_n^*,\quad \IV_g=t_1\Iv_gt_1^* + r_1\Iv_gr_1^*+\sum_{n=2}^\infty t_n\Iu_gt_n^* + r_n\Iv_gr_n^*
\]
and note that both formulas define strictly continuous $\beta$-cocycles.
If we define
\[
U = t_1r_1^* + \sum_{n=2}^\infty t_{n}t_{n-1}^* + r_{n-1}r_n^* \ \in \ \CU(\CM(B)^\beta),
\]
then we see that $\IV_g=U\IU_gU^*$ holds for all $g\in G$.

Furthermore we have
\[
\IU_g-\IV_g = t_1(\Iu_g-\Iv_g)t_1^* \quad\text{and}\quad t_1^*(\IU_g-\IV_g)t_1=\Iu_g-\Iv_g ,\quad g\in G,
\]
so the claim is equivalent to the assertion that $[g\mapsto \IU_g-\IV_g]$ is a norm-continuous map.
This map takes values in $B$.
We observe again that by $\IU_g-\IV_g=(\eins-\IV_g\IU_g^*)\IU_g$, the map $\IX_\bullet:=\IV_\bullet\IU^*_\bullet = [g\mapsto\IV_g\IU_g^*]$ takes values in $\CU(\eins+B)$. 
By strict continuity of $\IU$, it suffices to show that $\IX$ is a norm-continuous map.
Finally we observe for all $g\in G$ that
\[
\IX_g=\IV_g\IU_g^* =U\IU_gU^*\IU_g^* = U\beta^\IU_g(U)^*.
\]
So the claim follows directly from \ref{prop:automatic-continuity:1}.
\end{proof}

\begin{defi}
For two actions $\alpha: G\curvearrowright A$ and $\beta: G\curvearrowright B$ on separable \cstar-algebras,
let $\IE^G(\alpha,\beta)$ denote the set of all $(\alpha,\beta)$-Cuntz pairs, and let $\ID^G(\alpha,\beta)$ denote the subset of those pairs $(\phi^\pm,\Iu^\pm)$ with $(\phi^+,\Iu^+)=(\phi^-,\Iu^-)$, the degenerate Cuntz pairs.

Two elements $(\phi^\pm,\Iu^\pm)$ and $(\psi^\pm,\Iv^\pm)$ in $\IE^G(\alpha,\beta)$ are called homotopic, if there exists an $(\alpha,\beta[0,1])$-Cuntz pair $(\Phi^\pm,\IU^\pm)$ which restricts to $(\phi^\pm,\Iu^\pm)$ upon evaluation at $0\in [0,1]$, and restrict to $(\psi^\pm,\Iv^\pm)$ upon evaluation at $1\in [0,1]$.

For any unital inclusion $\CO_2\subseteq\CM(B^s)^{\beta^s}$ with generating isometries $s_1,s_2$, one can perform the Cuntz addition for two $(\alpha,\beta)$-Cuntz pairs as
\[
(\phi^\pm,\Iu^\pm)\oplus_{s_1,s_2} (\psi^\pm,\Iv^\pm) = \big( s_1\phi^\pm(\bullet)s_1^*+s_2\psi^\pm(\bullet)s_2^*, s_1\Iu^\pm_\bullet s_1^*+s_2\Iv^\pm_\bullet s_2^* \big),
\]
which is independent of the choice of $s_1,s_2$ up to homotopy.

Finally, one may define an equivalence relation $\sim$ on $\IE^G(\alpha,\beta)$ whereby one has $x\sim y$ if there are $d_1,d_2\in\ID^G(\alpha,\beta)$ such that $x\oplus d_1$ and $y\oplus d_2$ are homotopic.
The quotient $\IE^G(\alpha,\beta)/\sim$ then becomes an abelian monoid via Cuntz addition.
\end{defi}

\begin{defi} \label{def:kasparov-triple-construction}
Let $\alpha: G\curvearrowright A$ and $\beta: G\curvearrowright B$ be actions on separable \cstar-algebras.
For an $(\alpha,\beta)$-Cuntz pair $(\phi^\pm,\Iu^\pm)$, one may construct an $\alpha$-$\beta^s$-equivariant Kasparov triple $(E,\kappa,F)=:\Phi(\phi^\pm,\Iu^\pm)$ as follows.
One sets $E=(B^s)^{\Iu^+}\oplus (B^s)^{\Iu^-}=: B^{s+}\oplus B^{s-}$ as a right Hilbert $(B^s,\beta^s)$-module with grading $(x,y)\mapsto (x,-y)$.
One defines $\kappa: A\to\CL_{B^s}(E)$ via
\[
\kappa(a)(x,y) = (\phi^+(a)x, \phi^-(a)y)
\]
and $F\in\CL_{B^s}(E)$ via $F(x,y)=(y,x)$.
We denote by $[\Phi(\phi^\pm,\Iu^\pm)]$ its induced equivalence class in $KK^G(\alpha,\beta^s)$.
\end{defi}

\begin{nota}
Let $\alpha: G\curvearrowright A$ be an action on a separable \cstar-algebra.
As $KK^G$ is a stable functor in both variables, one has that the equivariant corner embedding $\id_A\otimes e_{11}: (A,\alpha)\to (A^s,\alpha^s)$ is always a $KK^G$-equivalence, and we will denote its induced element by $S_\alpha\in KK^G(\alpha,\alpha^s)$.
Note that this defines a natural transformation on the category $kk^G$ consisting of isomorphisms.
\end{nota}

\begin{theorem} \label{thm:Thomsen-iso}
For any two actions $\alpha: G\curvearrowright A$ and $\beta: G\curvearrowright B$ on separable \cstar-algebras, one has that $\IE^G(\alpha,\beta)/\sim$ is in fact a group, and the assignment $[\phi^\pm,\Iu^\pm] \mapsto [\Phi(\phi^\pm,\Iu^\pm)]\otimes S_\beta^{-1}$ defines a natural isomorphism onto $KK^G(\alpha,\beta)$.
\end{theorem}

\begin{prop} \label{prop:KKG-Cuntz-pair-functoriality}
For a cocycle morphism $(\phi,\Iu): (A,\alpha)\to (B,\beta)$ between separable $G$-\cstar-algebras, the element $KK^G(\phi,\Iu)\in KK^G(\alpha,\beta)$ is represented by any $(\alpha,\beta)$-Cuntz pair $(\phi^\pm,\Iu^\pm)$ of the form
\[
\phi^+=\phi\otimes e_{11},\quad \Iu^+=(\Iu\otimes e_{11})\oplus (\eins-e_{11})=:\Iu^{11},\quad \phi^-=0.
\]
So in particular one may choose $\Iu^-=\Iu^{11}$.
\end{prop}
\begin{proof}
By \autoref{def:kasparov-triple-construction}, the element $z\in KK^G(\alpha,\beta^s)$ described by such a Cuntz pair is represented by the equivariant Kasparov triple $(E, \kappa, F)$ given by
\[
E=B^{s+}\oplus B^{s-},\quad g\cdot (x,y)=\big( \Iu_g^{11}\beta^s_g(x), \Iu_g^-\beta^s_g(y) \big),
\]
and
\[
\kappa(a)(x,y)= \big( (\phi(a)\otimes e_{11}) x, 0 \big),\quad F(x,y)=(y,x).
\]
As $\kappa$ clearly takes values in $\CK_{B^s}(E)$, one may pass to the Kasparov triple $(E,\kappa,0)$ without changing its $KK^G$-class; cf.\ \cite[Paragraph 2.4]{Meyer00}.
Moreover we have $(E,\kappa,0)=(B^{s+},(\phi\otimes e_{11}),0)\oplus (B^{s-},0,0)$.
So $z$ is in fact represented by the triple $(B^{s+}, (\phi\otimes e_{11}), 0)$. (We see here already that the specific choice of the cocycle $\Iu^-$ does not matter.)

Next, we use (see \cite[Section 2]{Thomsen98}) that the invertible element $S_\beta^{-1}$ in $KK^G(\beta^s,\beta)$ is represented by the $\beta^s$-$\beta$-equivariant Kasparov triple of the form $(B^se_{11}, \theta, 0)$. 
Here $B^se_{11}$ is viewed as a right Hilbert $(B,\beta)$-module in the obvious way, and $\theta: B^s\to\CL_B(B^se_{11})$ is defined by $\theta(x)y=xy$.
Hence $z\otimes S_\beta^{-1}$ may be represented by the $\alpha$-$\beta$-equivariant Kasparov triple
\[
\big( B^{s+}\otimes_\theta (B^se_{11}) , (\phi\otimes e_{11})\otimes_\theta\eins, 0 \big) \cong ( E^0, \kappa^0, 0).
\]
Here one has $E^0=\ell^2(\IN)\otimes_\IC B$ as a right Hilbert $(B,\beta)$-module carrying the action 
\[
(g\cdot\xi)(n)=\begin{cases} \Iu_g\beta_g(\xi(1)) &,\quad n=1 \\ \beta_g(\xi(n)) &,\quad n\geq 2, \end{cases}
\]
and 
\[
(\kappa^0(a)\cdot\xi)(n) = \begin{cases} \phi(a)\xi(1) &,\quad n=1 \\ 0 &,\quad n\geq 2. \end{cases}
\]
Hence we obtain a direct sum decomposition
\[
( E^0, \kappa^0, 0) \cong (B^\Iu,\phi,0)\oplus (\ell^2(\IN^{\geq 2})\otimes_\IC B, 0, 0),
\] 
so the element $z\otimes S_\beta^{-1}$ is represented by $(B^\Iu,\phi,0)$, which is the same triple representing $KK^G(\phi,\Iu)$ by definition. This shows our claim.
\end{proof}

%%%%%%%%%%%%%%%%%%%%%%%%%%%%%%%%%%%%%%%%%%%%%%%%%%%%%%%%%%%%%%%%%%

\bibliographystyle{gabor}
\bibliography{master}

\begin{thebibliography}{10}
\providecommand{\url}[1]{\texttt{#1}}
\providecommand{\urlprefix}{URL }

\bibitem{BarlakSzaboVoigt17}
S.~Barlak, G.~Szab{\'o}, C.~Voigt: The {R}okhlin property for actions of
  compact quantum groups.
\newblock J. Funct. Anal. 272 (2017), no.~6, pp. 2308--2360.

\bibitem{BlaKK}
B.~Blackadar: {$K$}-theory for {O}perator {A}lgebras.
\newblock Second edition. Cambridge University Press (1998).

\bibitem{Bratteli72}
O.~Bratteli: Inductive limits of finite dimensional \cstar-algebras.
\newblock Trans. Amer. Math. Soc. 171 (1972), pp. 195--234.

\bibitem{BratteliEvansKishimoto95}
O.~Bratteli, D.~E. Evans, A.~Kishimoto: The {R}ohlin property for quasi-free
  automorphisms of the {F}ermion algebra.
\newblock Proc. Lond. Math. Soc. 71 (1995), no.~3, pp. 675--694.

\bibitem{BratteliKishimoto00}
O.~Bratteli, A.~Kishimoto: Trace scaling automorphisms of certain stable {AF}
  algebras, {II}.
\newblock Q. J. Math. 51 (2000), pp. 131--154.

\bibitem{BratteliKishimotoRobinson07}
O.~Bratteli, A.~Kishimoto, D.~W. Robinson: Rohlin flows on the {C}untz algebra
  {$\CO_\infty$}.
\newblock J. Funct. Anal. 248 (2007), pp. 472--511.

\bibitem{BratteliKishimotoRordamStormer93}
O.~Bratteli, A.~Kishimoto, M.~R{\o}rdam, E.~St{\o}rmer: The crossed product of
  a {UHF} algebra by a shift.
\newblock Ergodic Theory Dynam. Systems 13 (1993), no.~4, pp. 615--626.

\bibitem{BusbySmith70}
R.~C. Busby, H.~A. Smith: Representations of twisted group algebas.
\newblock Trans. Amer. Math. Soc. 149 (1970), pp. 503--537.

\bibitem{BussMeyerZhu13}
A.~Buss, R.~Meyer, C.~Zhu: A higher category approach to twisted actions on
  \cstar-algebras.
\newblock Proc. Edinb. Math. Soc. 56 (2013), no.~2, pp. 387--426.

\bibitem{Cuntz83}
J.~Cuntz: Generalized homomorphisms between \cstar-algebras and {$KK$}-theory.
\newblock Lect. Notes Math. 1031 (1983), pp. 31--45.

\bibitem{Cuntz84}
J.~Cuntz: {$K$}-theory and \cstar-algebras.
\newblock Lect. Notes Math. 1046 (1984), pp. 55--79.

\bibitem{Elliott76}
G.~A. Elliott: On the classification of inductive limits of sequences of
  semi-simple finite dimensional algebras.
\newblock J. Algebra 38 (1976), pp. 29--44.

\bibitem{Elliott10}
G.~A. Elliott: Towards a theory of classification.
\newblock Adv. Math. 223 (2010), pp. 30--48.

\bibitem{ElliottEvans93}
G.~A. Elliott, D.~E. Evans: The structure of irrational rotation
  \cstar-algebras.
\newblock Ann. Math. 138 (1993), no.~3, pp. 477--501.

\bibitem{ElliottEvansKishimoto98}
G.~A. Elliott, D.~E. Evans, A.~Kishimoto: Outer conjugacy classes of trace
  scaling automorphisms of stable {UHF} algebras.
\newblock Math. Scand. 83 (1998), pp. 74--86.

\bibitem{EvansKishimoto97}
D.~E. Evans, A.~Kishimoto: Trace scaling automorphisms of certain stable {AF}
  algebras.
\newblock Hokkaido Math. J. 26 (1997), pp. 211--224.

\bibitem{Gabe20}
J.~Gabe: A new proof of {K}irchberg's {$\mathcal O_2$}-stable classification.
\newblock J. reine angew. Math. 761 (2020), pp. 247--289.

\bibitem{Gardella14_1}
E.~Gardella: Classification theorems for circle actions on {K}irchberg
  algebras, {I}  (2014).
\newblock \urlprefix\url{http://arxiv.org/abs/1405.2469}.

\bibitem{Gardella14_2}
E.~Gardella: Classification theorems for circle actions on {K}irchberg
  algebras, {II}  (2014).
\newblock \urlprefix\url{http://arxiv.org/abs/1406.1208}.

\bibitem{GardellaHirshberg18}
E.~Gardella, I.~Hirshberg: Strongly outer actions of amenable groups on
  {$\CZ$}-stable \cstar-algebras  (2018).
\newblock \urlprefix\url{https://arxiv.org/abs/1811.00447}.

\bibitem{GardellaSantiago16}
E.~Gardella, L.~Santiago: Equivariant {$*$}-homomorphisms, {R}okhlin contraints
  and equivariant {UHF}-absorption.
\newblock J. Funct. Anal. 270 (2016), no.~7, pp. 2543--2590.

\bibitem{GiolKerr10}
J.~Giol, D.~Kerr: Subshifts and perforation.
\newblock J. reine angew. Math. 639 (2010), pp. 107--119.

\bibitem{GiordanoPutnamSkau95}
T.~Giordano, I.~F. Putnam, C.~F. Skau: Topological orbit equivalence and
  \cstar-crossed products.
\newblock J. Reine angew. Math. 496 (1995), pp. 51--111.

\bibitem{GoldsteinIzumi11}
P.~Goldstein, M.~Izumi: Quasi-free actions of finite groups on the {C}untz
  algebra {$\CO_\infty$}.
\newblock Tohoku Math. J. 63 (2011), pp. 729--749.

\bibitem{HermanOcneanu84}
R.~Herman, A.~Ocneanu: Stability for integer actions on {UHF} \cstar-algebras.
\newblock J. Funct. Anal. 59 (1984), pp. 132--144.

\bibitem{HermanJones82}
R.~H. Herman, V.~F.~R. Jones: Period two automorphisms of {UHF}
  \cstar-algebras.
\newblock J. Funct. Anal. 45 (1982), no.~2, pp. 169--176.

\bibitem{HermanPutnamSkau92}
R.~H. Herman, I.~F. Putnam, C.~F. Skau: Ordered {B}ratteli diagrams, dimension
  groups and topological dynamics.
\newblock Internat. J. Math. 3 (1992), no.~6, pp. 827--864.

\bibitem{Higson87}
N.~Higson: A characterization of {$KK$}-theory.
\newblock Pacific J. Math. 126 (1987), pp. 253--276.

\bibitem{HirshbergSzaboWinterWu17}
I.~Hirshberg, G.~Szab{\'o}, W.~Winter, J.~Wu: Rokhlin dimension for flows.
\newblock Comm. Math. Phys. 353 (2017), no.~1, pp. 253--316.

\bibitem{HirshbergWinter07}
I.~Hirshberg, W.~Winter: Rokhlin actions and self-absorbing \cstar-algebras.
\newblock Pacific J. Math. 233 (2007), no.~1, pp. 125--143.

\bibitem{HirshbergWinterZacharias15}
I.~Hirshberg, W.~Winter, J.~Zacharias: Rokhlin dimension and \cstar-dynamics.
\newblock Comm. Math. Phys. 335 (2015), pp. 637--670.

\bibitem{HirshbergWu17}
I.~Hirshberg, J.~Wu: The nuclear dimension of \cstar-algebras associated to
  homeomorphisms.
\newblock Adv. Math. 304 (2017), pp. 56--89.

\bibitem{HirshbergWu18}
I.~Hirshberg, J.~Wu: The nuclear dimension of \cstar-algebras associated to
  topological flows and orientable line foliations  (2018).
\newblock \urlprefix\url{https://arxiv.org/abs/1807.02246}.

\bibitem{Izumi04}
M.~Izumi: Finite group actions on \cstar-algebras with the {R}ohlin property
  {I}.
\newblock Duke Math. J. 122 (2004), no.~2, pp. 233--280.

\bibitem{Izumi04II}
M.~Izumi: Finite group actions on \cstar-algebras with the {R}ohlin property
  {II}.
\newblock Adv. Math. 184 (2004), no.~1, pp. 119--160.

\bibitem{Izumi10}
M.~Izumi: Group {A}ctions on {O}perator {A}lgebras.
\newblock Proc. Intern. Congr. Math.  (2010), pp. 1528--1548.

\bibitem{IzumiMatui10}
M.~Izumi, H.~Matui: {$\IZ^2$}-actions on {K}irchberg algebras.
\newblock Adv. Math. 224 (2010), pp. 355--400.

\bibitem{IzumiMatui18}
M.~Izumi, H.~Matui: Poly-{$\IZ$} group actions on {K}irchberg algebras {I}.
\newblock IMRN, to appear  (2018).
\newblock \urlprefix\url{https://arxiv.org/abs/1810.05850}.

\bibitem{IzumiMatui19}
M.~Izumi, H.~Matui: Poly-{$\IZ$} group actions on {K}irchberg algebras {II}.
\newblock Invent. Math., to appear  (2019).
\newblock \urlprefix\url{https://arxiv.org/abs/1906.03818}.

\bibitem{Jiang97}
X.~Jiang: Nonstable {$K$}-theory for {$\CZ$}-stable \cstar-algebras  (1997).
\newblock \urlprefix\url{http://arxiv.org/abs/math/9707228v1}.

\bibitem{JiangSu99}
X.~Jiang, H.~Su: On a simple unital projectionless \cstar-algebra.
\newblock Amer. J. Math. 121 (1999), no.~2, pp. 359--413.

\bibitem{KaliszewskiOmlandQuigg15}
S.~Kaliszewski, T.~Omland, J.~Quigg: Three versions of categorical
  crossed-product duality  (2015).
\newblock \urlprefix\url{https://arxiv.org/abs/1503.00663}.

\bibitem{Kasparov88}
G.~G. Kasparov: Equivariant {$KK$}-theory and the {N}ovikov conjecture.
\newblock Invent. Math. 91 (1988), pp. 147--201.

\bibitem{KatsuraMatui08}
T.~Katsura, H.~Matui: Classification of uniformly outer actions of {$\IZ^2$} on
  {UHF} algebras.
\newblock Adv. Math 218 (2008), pp. 940--968.

\bibitem{Kerr19}
D.~Kerr: Dimension, comparison, and almost finiteness.
\newblock J. Eur. Math. Soc., to appear  (2017).
\newblock \urlprefix\url{https://arxiv.org/abs/1710.00393}.

\bibitem{KerrSzabo20}
D.~Kerr, G.~Szab{\'o}: Almost finiteness and the small boundary property.
\newblock Comm. Math. Phys. 374 (2020), pp. 1--31.

\bibitem{KirchbergC}
E.~Kirchberg: The {C}lassification of {P}urely {I}nfinite \cstar-{A}lgebras
  {U}sing {K}asparov's {T}heory (2003).
\newblock Preprint.

\bibitem{Kirchberg04}
E.~Kirchberg: Central sequences in \cstar-algebras and strongly purely infinite
  algebras.
\newblock Operator Algebras: The Abel Symposium 1 (2004), pp. 175--231.

\bibitem{Kishimoto95}
A.~Kishimoto: The {R}ohlin property for automorphisms of {UHF} algebras.
\newblock J. reine angew. Math. 465 (1995), pp. 183--196.

\bibitem{Kishimoto96_R}
A.~Kishimoto: A {R}ohlin property for one-parameter automorphism groups.
\newblock Comm. Math. Phys. 179 (1996), no.~3, pp. 599--622.

\bibitem{Kishimoto96}
A.~Kishimoto: The {R}ohlin property for shifts on {UHF} algebras and
  automorphisms of {C}untz algebras.
\newblock J. Funct. Anal. 140 (1996), pp. 100--123.

\bibitem{Kishimoto98}
A.~Kishimoto: Automorphisms of {AT} algebras with the {R}ohlin property.
\newblock J. Operator Theory 40 (1998), pp. 277--294.

\bibitem{Kishimoto98II}
A.~Kishimoto: Unbounded derivations in {AT} algebras.
\newblock J. Funct. Anal. 160 (1998), pp. 270--311.

\bibitem{Kishimoto00}
A.~Kishimoto: Rohlin property for shift automorphisms.
\newblock Rev. Math. Phys. 12 (2000), no.~7, pp. 965--980.

\bibitem{Kishimoto01}
A.~Kishimoto: {UHF} flows and the flip automorphism.
\newblock Rev. Math. Phys. 13 (2001), pp. 1163--1181.

\bibitem{Kishimoto02}
A.~Kishimoto: Rohlin flows on the {C}untz algebra {$\CO_2$}.
\newblock Internat. J. Math. 13 (2002), no.~10, pp. 1065--1094.

\bibitem{Kishimoto04}
A.~Kishimoto: Core symmetries of a flow.
\newblock Rev. Math. Phys. 16 (2004), no.~4, pp. 479--507.

\bibitem{Liao16}
H.-C. Liao: A {R}okhlin type theorem for simple \cstar-algebras of finite
  nuclear dimension.
\newblock J. Funct. Anal. 270 (2016), pp. 3675--3708.

\bibitem{Liao17}
H.-C. Liao: Rokhlin dimension of {$\IZ^m$}-actions on simple \cstar-algebras.
\newblock Internat. J. Math. 28 (2017), no.~7.

\bibitem{LinPhillips10}
H.~Lin, N.~C. Phillips: Crossed products by minimal homeomorphisms.
\newblock J. reine angew. Math. 641 (2010), pp. 95--122.

\bibitem{MasudaTomatsu16}
T.~Masuda, R.~Tomatsu: Rohlin flows on von {N}eumann algebras.
\newblock Mem. Amer. Math. Soc. 244 (2016), no.~2.

\bibitem{Matui08}
H.~Matui: Classification of outer actions of {$\IZ^N$} on {$\CO_2$}.
\newblock Adv. Math. 217 (2008), pp. 2872--2896.

\bibitem{Matui10}
H.~Matui: {$\IZ$}-actions on {AH} algebras and {$\IZ^2$}-actions on {AF}
  algebras.
\newblock Comm. Math. Phys. 297 (2010), pp. 529--551.

\bibitem{Matui11}
H.~Matui: {$\IZ^N$}-actions on {UHF} algebras of infinite type.
\newblock J. reine angew. Math 657 (2011), pp. 225--244.

\bibitem{MatuiSato12}
H.~Matui, Y.~Sato: {$\CZ$}-stability of crossed products by strongly outer
  actions.
\newblock Comm. Math. Phys. 314 (2012), no.~1, pp. 193--228.

\bibitem{MatuiSato14}
H.~Matui, Y.~Sato: {$\CZ$}-stability of crossed products by strongly outer
  actions {II}.
\newblock Amer. J. Math. 136 (2014), pp. 1441--1497.

\bibitem{Meyer00}
R.~Meyer: Equivariant {K}asparov theory and generalized homomorphisms.
\newblock K-Theory 21 (2000), no.~3, pp. 201--228.

\bibitem{Meyer19}
R.~Meyer: On the classification of group actions on \cstar-algebras up to
  equivariant {$KK$}-equivalence  (2019).
\newblock \urlprefix\url{https://arxiv.org/abs/1906.11163}.

\bibitem{Nakamura99}
H.~Nakamura: The {R}ohlin property for {$\IZ^2$}-actions on {UHF} algebras.
\newblock J. Math. Soc. Japan 51 (1999), no.~3, pp. 583--612.

\bibitem{Nakamura00}
H.~Nakamura: Aperiodic automorphisms of nuclear purely infinite simple
  \cstar-algebras.
\newblock Ergodic Theory Dynam. Systems 20 (2000), pp. 1749--1765.

\bibitem{Nawata16}
N.~Nawata: Finite group actions on certain stably projectionless
  \cstar-algebras with the {R}ohlin property.
\newblock Trans. Amer. Math. Soc. 368 (2016), no.~1, pp. 471--493.

\bibitem{Nawata19_2}
N.~Nawata: Rohlin actions of finite groups on the {R}azak--{J}acelon algebra.
\newblock IMRN, to appear  (2019).
\newblock \urlprefix\url{https://arxiv.org/abs/1905.09469}.

\bibitem{Nawata19}
N.~Nawata: Trace scaling automorphisms of the stabilized {R}azak--{J}acelon
  algebra.
\newblock Proc. Lond. Math. Soc. 118 (2019), no.~3, pp. 545--576.

\bibitem{PackerRaeburn89}
J.~A. Packer, I.~Raeburn: Twisted crossed products of \cstar-algebras.
\newblock Math. Proc. Cambridge Philos. Soc. 106 (1989), no.~2, pp. 293--311.

\bibitem{Phillips00}
N.~C. Phillips: A classification theorem for nuclear purely infinite simple
  \cstar-algebras.
\newblock Doc. Math. 5 (2000), pp. 49--114.

\bibitem{Putnam89}
I.~F. Putnam: The \cstar-algebras associated with minimal homeomorphisms of the
  {C}antor set.
\newblock Pacific J. Math. 136 (1989), no.~2, pp. 329--353.

\bibitem{Rordam93}
M.~R{\o}rdam: Classification of inductive limits of {C}untz algebras.
\newblock J. reine angew. Math. 440 (1993), pp. 175--200.

\bibitem{Rordam}
M.~R{{\o}}rdam: Classification of {N}uclear \cstar-{A}lgebras. {E}ncyclopaedia
  of {M}athematical {S}ciences.
\newblock Springer (2001).

\bibitem{Sato10}
Y.~Sato: The {R}ohlin property for automorphisms of the {J}iang--{S}u algebra.
\newblock J. Funct. Anal. 259 (2010), no.~2, pp. 453--476.

\bibitem{Sato16}
Y.~Sato: Actions of amenable groups and crossed products of {$\CZ$}-absorbing
  \cstar-algebras.
\newblock Advanced Studies in Pure mathematics 80 (2019), no.~5, pp. 189--210.

\bibitem{Szabo15plms}
G.~Szab{\'o}: The {R}okhlin dimension of topological {$\IZ^m$}-actions.
\newblock Proc. Lond. Math. Soc. 110 (2015), no.~3, pp. 673--694.

\bibitem{Szabo17ssa3}
G.~Szab{\'o}: Strongly self-absorbing \cstar-dynamical systems, {III}.
\newblock Adv. Math. 316 (2017), no.~20, pp. 356--380.

\bibitem{Szabo18kp}
G.~Szab{\'o}: Equivariant {K}irchberg--{P}hillips-type absorption for amenable
  group actions.
\newblock Comm. Math. Phys. 361 (2018), no.~3, pp. 1115--1154.

\bibitem{Szabo18ssa2}
G.~Szab{\'o}: Strongly self-absorbing \cstar-dynamical systems, {II}.
\newblock J. Noncomm. Geom. 12 (2018), no.~1, pp. 369--406.

\bibitem{Szabo18ssa}
G.~Szab\'{o}: Strongly self-absorbing {$\mathrm{C}^*$}-dynamical systems.
\newblock Trans. Amer. Math. Soc. 370 (2018), pp. 99--130.

\bibitem{Szabo19ssa4}
G.~Szab{\'o}: Actions of certain torsion-free elementary amenable groups on
  strongly self-absorbing \cstar-algebras.
\newblock Comm. Math. Phys. 371 (2019), no.~1, pp. 267--284.

\bibitem{Szabo19si}
G.~Szab{\'o}: Equivariant property {(SI)} revisited.
\newblock Anal. PDE, to appear  (2019).
\newblock \urlprefix\url{https://arxiv.org/abs/1904.10897}.

\bibitem{Szabo19rd}
G.~Szab{\'o}: Rokhlin dimension: absorption of model actions.
\newblock Anal. PDE 12 (2019), no.~5, pp. 1357--1396.

\bibitem{Szabo20R}
G.~Szab{\'o}: The classification of {R}okhlin flows on \cstar-algebras.
\newblock Comm. Math. Phys., to appear  (2020).
\newblock \urlprefix\url{http://dx.doi.org/10.1007/s00220-020-03812-2}.

\bibitem{Szabo18ssaCOR}
G.~Szab{\'o}: Corrigendum to "{S}trongly self-absorbing \cstar-dynamical
  systems".
\newblock Trans. Amer. Math. Soc. 373 (2020), pp. 7527--7531.

\bibitem{SzaboWuZacharias19}
G.~Szab{\'o}, J.~Wu, J.~Zacharias: Rokhlin dimension for actions of residually
  finite groups.
\newblock Ergodic Theory Dynam. Systems 39 (2019), no.~8, pp. 2248--2304.

\bibitem{Thomsen98}
K.~Thomsen: The universal property of equivariant {$KK$}-theory.
\newblock J. reine angew. Math. 504 (1998), pp. 55--71.

\bibitem{Thomsen00}
K.~Thomsen: Equivariant {$KK$}-theory and \cstar-extensions.
\newblock K-Theory 19 (2000), no.~3, pp. 219--249.

\bibitem{TomsWinter07}
A.~S. Toms, W.~Winter: Strongly self-absorbing \cstar-algebras.
\newblock Trans. Amer. Math. Soc. 359 (2007), no.~8, pp. 3999--4029.

\bibitem{TomsWinter13}
A.~S. Toms, W.~Winter: Minimal dynamics and {$K$}-theoretic rigidity:\
  {E}lliott's conjecture.
\newblock GAFA 23 (2013), pp. 467--481.

\bibitem{Winter11}
W.~Winter: Strongly self-absorbing \cstar-algebras are {$\CZ$}-stable.
\newblock J. Noncomm. Geom. 5 (2011), no.~2, pp. 253--264.

\end{thebibliography}

\end{document}